\documentclass[11pt,reqno,UTF8,twoside]{amsart}
\usepackage{mathrsfs}
\usepackage{amsfonts,amssymb,amsmath,amsthm}
\usepackage{cite}
\usepackage[colorlinks=true,citecolor=red]{hyperref}
\usepackage{tocvsec2}
\usepackage{titletoc}
\usepackage{geometry}
\usepackage{bm}
\usepackage{indentfirst}
\usepackage{graphicx}
\usepackage{float} 
\usepackage{booktabs}
\usepackage{longtable}
\usepackage{subfigure}
\usepackage{stmaryrd}
\usepackage{enumerate}
\usepackage{tikz}
\usepackage{ulem}
\usepackage{color,soul}
\usepackage{setspace}
\usepackage[pagewise]{lineno} 
\allowdisplaybreaks
\usepackage{appendix}
\usepackage{cases}
\usepackage{todonotes}
\usepackage{enumitem}

\numberwithin{equation}{section}
\newtheorem{theorem}{Theorem}[section]
\newtheorem{lemma}[theorem]{Lemma}

\newtheorem{proposition}[theorem]{Proposition}
\numberwithin{equation}{section}
\newtheorem{remark}[theorem]{Remark}
\newtheorem{definition}[theorem]{Definition}

\newcounter{thm}
\newtheorem{thm}[thm]{}

\newcommand{\T}{\mathbb{T}}
\newcommand{\N}{\mathbb{N}}
\newcommand{\R}{\mathbb{R}}
\newcommand{\Z}{\mathbb{Z}}
\newcommand{\E}{\mathbb{E}}
\newcommand{\Pb}{\mathbb{P}}
\newcommand{\supp}{\operatorname{supp}}
\newcommand{\re}{\operatorname{Re}}

\newcommand{\dist}{\operatorname{dist}}
\newcommand{\Lip}{\operatorname{Lip}}

\tikzstyle{process} = [rectangle, rounded corners, minimum width=4cm, minimum height=1cm, text centered, draw=black, align=center]
\tikzstyle{point} = [coordinate, on grid]
\tikzstyle{arrow} = [->,>=stealth]
\tikzstyle{dasharrow} = [dashed,->,>=stealth]

\hypersetup{linkcolor=blue}

\makeatletter
\@namedef{subjclassname@2020}{\textup{2020} Mathematics Subject
	Classification}
\makeatother

\begin{document}
	
\title[Mixing for random NLS]{\large E{\MakeLowercase{xponential mixing for the randomly forced} NLS \MakeLowercase{equation}}}

\author[Y. Chen, S. Xiang, Z. Zhang, J.-C. Zhao]{Y\MakeLowercase{uxuan} C\MakeLowercase{hen}, S\MakeLowercase{hengquan} X\MakeLowercase{iang}, Z\MakeLowercase{hifei} Z\MakeLowercase{hang}, J\MakeLowercase{ia-}C\MakeLowercase{heng} Z\MakeLowercase{hao}}
		
\address[Yuxuan Chen]{School of Mathematical Sciences, Peking University, 100871, Beijing, China.}
\email{chen\underline{ }yuxuan@pku.edu.cn}

\address[Shengquan  Xiang]{School of Mathematical Sciences, Peking University, 100871, Beijing, China.}
\email{shengquan.xiang@math.pku.edu.cn}

\address[Zhifei  Zhang]{School of Mathematical Sciences, Peking University, 100871, Beijing, China.}
\email{zfzhang@math.pku.edu.cn}

\address[Jia-Cheng Zhao]{School of Mathematical Sciences, Shenzhen University, 518061, Shenzhen, China.}
\email{zjc@szu.edu.cn}

\subjclass[2020]{
35Q55, %NLS equations (nonlinear Schr¨odinger equations)
37A25,  % (2000–now)Ergodicity, mixing, rates of mixing 
37L15, % Stability problems for infinite-dimensional dissipative dynamical systems
93C20. %(1973–now)Control/observation systems governed by partial differential equations
    }
	
\keywords{Exponential mixing; Nonlinear Schr\"odinger equations; Exponential asymptotic compactness; Global stability; Control property}

\begin{abstract}
This paper investigates exponential mixing of the invariant measure 
for randomly forced nonlinear Schr\"{o}dinger equation, with damping and random noise localized in space. Our study emphasizes the crucial role of exponential asymptotic compactness and control properties in establishing the ergodic properties of random dynamical systems. This work extends the series \cite{LWX-24,CLXZ-24} on the statistical behavior of randomly forced dispersive equations.

\end{abstract}
	
\maketitle
\vspace{-.6cm}	
	
\setcounter{tocdepth}{1}
\tableofcontents

\section{Introduction}

{\it Exponential mixing} is a significant topic in statistical mechanics, random PDEs, stochastic processes, and finance; see, e.g.~\cite{DPZ-96}. It describes that the law of random field $x\mapsto u^\omega (t,x)$ converges, as $t\rightarrow\infty$, to the unique invariant measure at an exponential rate. 

While the research on
statistical behaviors of parabolic equations has produced rich results by now, much less is known for dispersive equations.
Our aim of this paper is to investigate the exponential mixing for a randomly forced nonlinear Schr\"odinger (NLS) equation on 1D torus $\T:=\R/2\pi\Z$, reading
\begin{equation}\label{Random-problem}
\left\{\begin{array}{ll}
iu_t+u_{xx}+ia(x)u=|u|^{p-1}u+\eta(t,x),\\
u(0,x)=u_0(x),
\end{array}\right.
\end{equation}
where the order $p\geq 3$ of nonlinearity is odd. The symbols $a(x)\ge 0$ and $\eta(t,x)$ represent the damping and random noise, respectively; both of them may vanish outside an open subset of $\T$.

NLS equations serve as basic models in diverse areas of science, including plasma, nonlinear optics, hydrodynamics, and quantum chemistry. In physical models, the noise represents random spatial influences or temporal fluctuations of certain parameters; and the damping effect corresponds to dissipative phenomena such as wave collapse, Landau damping, and ion cyclotron resonance in plasma. For further physical background, see, e.g.~\cite{SS99}. As the energy is injected to the system by the noise, and simultaneously dissipated by the damping, investigations on the equilibrium state become both meaningful and intriguing.

\medskip

In Section \ref{Section-setup} we state our main theorem, followed by a review of background and previous works in Section \ref{Section-review}. Then in Section \ref{Section-strategy} we overview the strategy and new challenges.

\subsection{Setup and main theorem}\label{Section-setup}

Our setting on $a(x)$ and $\eta(t,x)$ is summarized as follows.

\begin{itemize}
\item[$(\mathbf{S1})$] (Localized structure) 
The damping coefficient $a: \mathbb{T}\rightarrow [0,\infty)$ is smooth, non-negative, and localized: 
there exists a constant $a_0>0$ and an open subset $\mathcal{I}_1$ of $\T$ such that
\[a(x)\geq a_0,\ \forall x\in \mathcal{I}_1.\]
In addition, a smooth function $\chi\colon \mathbb{T}\to \mathbb{R}$ will appear in the noise structure, indicating that the noise is also localized: there exists $\chi_0>0$ and open subset $\mathcal{I}_2$ of $\T$ such that
\[\chi(x)\geq \chi_0,\ \forall x\in\mathcal I_2.\]
\end{itemize}

\noindent To describe the random noise, for arbitrary $T>0$, denote 
by $\{\alpha^T_j;j\in\N^+\}$ an orthonormal basis of $L^2(0,T)$. We also define the trigonometric  basis of $L^2(\T)$ by
\[e_k(x)=\frac{1}{\sqrt{2\pi}}
e^{ikx},\quad k\in \mathbb{Z}.\]
Then $\{\alpha_j^T(t) e_k(x);j\in \mathbb{N}^+, k\in \mathbb{Z}\}$ serves as an orthonormal basis of $L^2([0,T]\times \mathbb{T})$.

\medskip

\begin{itemize}
\item[$(\mathbf{S2})$] (Noise structure) The law of $\eta(t,x)$ is statistically $T$-periodic:
\[\eta(t,x)= \eta_n(t-nT,x),\quad t\in [nT,(n+1)T),\ n\in\N,\]
and $\eta_n$ are i.i.d.~random variables in $L^2([0,T]\times \mathbb{T})$. Specifically,
\[\displaystyle\eta_n(t,x)=\chi(x)\sum_{j\in\N^+,\, k\in\Z}b_{j,k}(\theta^{n}_{j,k,1}+i\theta_{j,k,2}^{n})\alpha^{\scriptscriptstyle T}_j(t)e_k(x),\quad t\in[0,T).\]
Here $b_{j,k}\ge 0$ are deterministic numbers tending to $0$ sufficiently fast, and $\theta_{j,k,l}^{n}$ are independent real random variables. Moreover, $\theta_{j,k,l}^{n}$ admits a probability density function $\rho_{j,k,l}$ supported by the interval $[-1,1]$, which is $C^1$ and satisfies $\rho_{j,k,l}(0)>0$.
\end{itemize}

\medskip

With the above settings, the solution $u(t)$ of (\ref{Random-problem}) at discrete times $nT$, defines a Markov process $u_n:=u(nT)$. We introduce $B(r)$ with $r\ge 1$, representing the strength of noise $\eta(t,x)$:
$$
B(r):=\sum_{j\in\mathbb{N}^+,\, k\in\mathbb{Z}} |b_{j,k}|^2 \langle k\rangle^{2r}.
$$

The following theorem is the main result of this paper.

\begin{thm}\hypertarget{thm1}{}
Let $s\geq 1$ and $T,B_0,\sigma>0$ be arbitrarily given. In addition to the settings $(\mathbf{S1}),(\mathbf{S2})$, assume that the non-negative numbers $b_{j,k}$ satisfy 
\begin{equation}\label{bounded-noise-0}
B(s+\sigma)\leq B_0.
\end{equation}
Then there exists a constant $N\in \N^+$ such that if 
\begin{equation}\label{non-degenerate}
b_{j,k}\neq 0,\ \forall j,|k|\leq N,
\end{equation}
the Markov process $u_n:=u(nT)$ admits a unique invariant measure $\mu$ in $H^s(\mathbb{T})$, and $\mu$ is supported in an $H^{s+\sigma}(\mathbb{T})$-bounded set. Moreover, there exist  constants $C,\gamma>0$ such that 
\begin{equation*}
\|\mathscr D(u_n)-\mu\|_{L}^*\leq V_s(u_0)e^{-\gamma n}
\end{equation*}
for any $u_0\in H^s(\mathbb{T})$ and $n\in\N$,
where $\|\cdot\|_L^*$ is the dual-Lipschitz distance in $H^s(\mathbb{T})$ (defined below in the notations), $\mathscr D(u_n)$ stands for the law of $u_n$, and 
$$
V_s(u_0)=\begin{cases}
C(1+E(u_0))&\text{\rm for } s=1,\\
C(1+\|u_0\|_{_{H^s}})^{C(1+E(u_0))^{(p-1)\lceil 4s-3\rceil/2}}&\text{\rm for } s>1.
\end{cases}
$$
Here $E(\cdot)$ is the $H^1$-energy functional defined by \eqref{energy-function}.
\end{thm}

To the best of our knowledge,
this provides a first result on the exponential mixing for randomly forced NLS. The generality and difficulty of equation \eqref{Random-problem} lies in the localized feature of damping and noise; both $a(\cdot)$ and $\eta(t,\cdot)$ may vanish outside open subsets $\mathcal I_1,\mathcal I_2$ of $\T$, respectively. See Section {\rm\ref{Section-review-NLS}} for a review of previous ergodicity results on random NLS.

\begin{remark}
This paper is part of a serial study toward the statistical behavior of random dispersive PDEs. In the previous work {\rm\cite{LWX-24}}, a general criterion of exponential mixing has been established, and successfully applied to nonlinear wave equations (see also {\rm\cite{CLXZ-24}} for the corresponding large deviation principle). 
We believe that this methodology is effective for a wide range of dispersive and hyperbolic equations. In this paper, we extend the approach to NLS equations, while the case of KdV equation will be addressed in a forthcoming work. It would be of interest to investigate whether this technique can also be applied to the challenging problem involving the 2D Euler equation with a linear zero-order damping (the existence of invariant measure has been established in {\rm\cite{BF-20}}, while the uniqueness remains open).

The highlight of this abstract criterion is the new concept ``exponential asymptotic compactness''\footnote{In \cite{LWX-24} this concept is originally referred to as ``asymptotic compactness''.  Adding the attribute ``exponential'' in the present paper is for concreteness. In fact, the terminology ``asymptotic compactness'' is widely used in the theory of dynamical systems, and does not require an exponential attraction in those settings.}, whose optimality will be illustrated in the main 
content; see Section {\rm\ref{Section-133}} for more discussions. The verification for PDE models involves multiple subjects: asymptotic dynamics, global stabilization and control properties.
\end{remark}

\begin{remark}[Comparison with wave equations]
The study of NLS is more intricate than wave equations involved in {\rm\cite{LWX-24}}. One reason is that, unlike the wave propagator which gains one derivative with respect to the source term (implying extra regularity for nonlinear term in the Duhamel formula), no such regularization holds for the NLS. As a result, the nonlinearity becomes an inherent difficulty.
To confront this, 
we utilize Carleman estimate to study asymptotic dynamics and nonlinear smoothing to study control problems, which are novel methods in this paper. 

Another notable point is that, in the case of wave equations {\rm\cite{LWX-24}}, the time period $T>0$ cannot be too small, due to the finite speed of propagation; while in the \hyperlink{thm1}{\color{black}Main Theorem}, period $T$ is arbitrary. This will be achieved by new observations on damped Schr\"odinger equations.

More explanations on the above statements can be found in Section {\rm\ref{Section-strategy-AC}}.
\end{remark}

To conclude this subsection, we mention some other statistical consequences of exponential mixing. Firstly, owing to the Kolmogorov–Chapman relation, the exponential mixing remains valid when $u_0$ is a random variable (independent of the noise) provided $\mathbb{E} V_s(u_0)<\infty$:
\[\|\mathcal{D}(u_n)-\mu\|_L^*\le \mathbb{E} V_s(u_0) e^{-\gamma n};\] 
see Proposition~\ref{Proposition-EX}. Secondly, the continuous-time version of exponential mixing follows from an argument similar to \cite[Lemma 3.2]{Shi-21}.
Thirdly, the law of large numbers and central limit theorem are corollaries of the exponential mixing \cite[Proposition 2.1]{LWX-24}. In addition, the mixing property is deeply related to large deviations of Donsker--Varadhan type \cite[Theorem 2.4]{CLXZ-24}. 

\subsection{Prior works}\label{Section-review}

We review briefly previous results on the mixing of random PDEs. As the literature is now extensive, we emphasize only the most relevant works.

\subsubsection{Parabolic equations}

The research on mixing starts with the interest on parabolic equations. A typical model is the 2D Navier--Stokes system:
\begin{equation}\label{NS-problem}
u_t-\Delta u+(u\cdot\nabla)u+\nabla p=\eta(t,x),\quad {\rm div}(u)=0.
\end{equation}
First results can be found in, e.g.~\cite{EM-01, Hairer-02,FM-95}. Exploiting Malliavin calculus and smoothing properties of Markov semigroup, Hairer and Mattingly \cite{HM-06,HM-08} establish exponential mixing for (\ref{NS-problem}) when the noise $\eta$ is white in time and highly-degenerate in Fourier modes. Liu and Lu \cite{LL-25} address the issue in which (\ref{NS-problem}) is also driven by a quasi-periodic deterministic force.
Using coupling method and control theory, Shirikyan \cite{Shi-15} proves exponential mixing for (\ref{NS-problem}) with interior localized noise (see also \cite{Shi-21} for boundary noise), Kuksin, Nersesyan and Shirikyan \cite{KNS-20} deal with the setting where $\eta$ is bounded Haar noise,
and Nersesyan \cite{Ner-22} treats the case where (\ref{NS-problem}) is defined on an unbounded domain.
These works \cite{KNS-20,Shi-15,Shi-21,Ner-22} convey a link between mixing and controllability. 

We remark that \cite{Shi-15} provides a first result on the localized noise which we also consider in this paper. More recently, Nersesyan \cite{Nersesyan-24} establishes exponential mixing for the complex Ginzburg--Landau equation driven by spatially localized Haar noise that involves only a few Fourier modes. For further research it is of great interest to consider exponential mixing for the NLS with noise localized both 
in physical space and Fourier modes.

\subsubsection{Wave equations}

A natural question is if the methodologies inspired by parabolic problems can be adapted to dispersive equations. 
There have been several efforts to the wave equation:
\begin{equation}\label{Wave-problem}
\boxempty u+a(x)u_t+f(u)=\eta(t,x).
\end{equation}
An early result on the existence and uniqueness of invariant measure can be found in Barbu and Da Prato \cite{BDP-02}. The exponential mixing is later studied by
Martirosyan \cite{Martirosyan-14}, where the damping coefficient $a(x)\equiv a>0$ is a constant acting on the entire domain, and the nonlinearity $f(u)$ is typically of the form $u^{3-\varepsilon}$. The recent work \cite{LWX-24} establishes exponential mixing for (\ref{Wave-problem}) with
localized damping $a(x)$, cubic nonlinearity $f(u)=u^3$ and localized noise $\eta(t,x)$. 
In particular, a new abstract criterion based on exponential asymptotic compactness is proposed, at the level of random dynamical systems; see also Section \ref{Section-131}.

\subsubsection{Schr\"{o}dinger equations}\label{Section-review-NLS}

In comparison with parabolic equations and wave equations, the research on exponential mixing for Schr\"{o}dinger equations is almost a vacuum. A related contribution is due to Debussche and Odasso \cite{DO-05}, where the authors establish a polynomial mixing for cubic NLS with constant damping and white noise. To the best of our knowledge, there are hardly further results achieving the mixing property for NLS with an exponential rate. 

For other ergodic results of NLS, the reader is referred to, e.g.~\cite{BFZ-23,BFZ-24,EKZ-17}. We also mention that recent years have witnessed considerable interest in Schr\"odinger equations with random initial data; see, e.g. \cite{BT-24,BTT-18,GGKS-23, DNY-24,OPT-21,BKTV-24}.

\subsection{Difficulties and the strategy}\label{Section-strategy}

We briefly illustrate our strategy on the $H^1$-setting (the issue in higher $H^s$ requires extra techniques to be handled in Section~\ref{Section-stability}), which can be summarized and depicted as in Figure~\ref{fig_outline}.

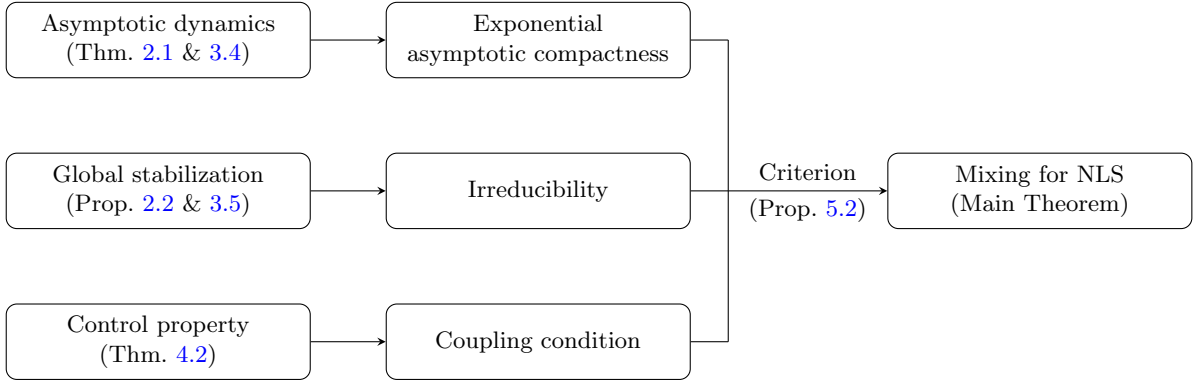
\setlength{\abovecaptionskip}{0.3cm}

\begin{figure}[H]
\centering
\begin{tikzpicture}[node distance=2cm]

{\footnotesize
\node (Dynamic) [process] {Asymptotic dynamics\\
(Thm.~\ref{Theorem-AC} \& \ref{Theorem-AC-Hs})};
\node (EAC) [process, right of=Dynamic, xshift=3cm] {Exponential\\ asymptotic compactness};

\node (Dispersive) [process, below of=Dynamic] {Global stabilization\\(Prop.~\ref{Corollary-stability} \& \ref{Theorem-H^sstability})};
\node (I) [process, right of=Dispersive, xshift=3cm] {Irreducibility};

\node (Control) [process, below of=Dispersive] {Control property\\
(Thm.~\ref{Theorem-control})};
\node (C) [process, right of=Control, xshift=3cm] {Coupling condition};

\node (EMLDP) [process, right of=I, xshift=4.6cm] {Mixing for NLS\\
(\hyperlink{thm1}{\color{black}Main Theorem})};

\draw[arrow] (Dynamic) -- (EAC);
\draw[arrow] (Dispersive) -- (I);
\draw[arrow] (Control) -- (C);

\node (pt1) [point, right of=EAC, xshift=0.5cm]{};
\node (pt2) [point, right of=I, xshift=0.5cm]{};
\node (pt3) [point, right of=C, xshift=0.5cm]{};

\draw (EAC) -- (pt1);
\draw (C) -- (pt3);
\draw (pt1) -- (pt2);
\draw (I) -- (pt2);
\draw (pt3) -- (pt2);
\draw[arrow] (pt2) --node[yshift=2.5mm]{Criterion} node[yshift=-2.5mm]{(Prop.~\ref{Proposition-EX})} (EMLDP) ;

}
\end{tikzpicture}
\caption{Outline of the proof}
\label{fig_outline}
\end{figure}

Roughly speaking, we invoke the previous result \cite[Theorem 2.1]{LWX-24} for a general criterion of exponential mixing in random dynamical systems. This criterion consists of three hypotheses on Markov process: exponential asymptotic compactness, irreducibility and coupling condition on compact set. 

The exponential asymptotic compactness is related to the asymptotic dynamics for the NLS with deterministic force, while the irreducibility is a consequence of global stabilization for the unforced equation. Finally, we interpret the coupling condition as a control property. See Sections \ref{Section-coupling} and \ref{Section-verification} for more details on the connection of the probabilistic hypotheses with these deterministic problems.

\medskip

In Section \ref{Section-131} below, we briefly describe the general criterion, followed by 
corresponding interpretations for deterministic equations in Section \ref{Section-strategy-AC}. Further discussions on exponential asymptotic compactness, the key ingredient of the criterion, are included in Section \ref{Section-133}.

\subsubsection{Abstract criterion}\label{Section-131}

Let us consider $u_n=u(nT)$, where $u\in C(\mathbb{R}^+;H^1)$ stands for the solution of random NLS equation (\ref{Random-problem}). Due to the setting $(\mathbf{S2})$, it is easy to see $u_n$ forms a Markov process. 
The criterion of exponential mixing requires the following three properties on $u_n$. See Section~\ref{Section-framework} for precise statements. 

\begin{itemize}
\item \textbf{Exponential asymptotic compactness:} there exists a bounded subset $\mathcal Y$ of $H^{1+\sigma}$ (which is in particular compact in $H^1$) attracting $u_n$ exponentially in a pathwise manner:
\[\dist_{H^1}(u_n,\mathcal{Y})\le Ce^{-\kappa n}\quad a.s.\]

\medskip

\item \textbf{Irreducibility on $\mathcal Y$:} the trajectory $u_n$ has positive probability to enter into small neighborhoods of $0$: for any $\varepsilon>0$, one can find $n\in \mathbb{N}$ such that
\[\inf_{u_0\in \mathcal{Y}} \mathbb{P}(\|u_n\|_{_{H^1}}<\varepsilon)>0.\]

\medskip

\item \textbf{Coupling condition on $\mathcal Y$:} any two processes in $\mathcal{Y}$ become closer with high probability: for any $u_0,\tilde{u}_0\in \mathcal{Y}$, one can construct a coupling $(\mathcal{R},\mathcal{R}')$ of $(u_1,\tilde{u}_1)$, such that
\[\mathbb{P}(\|\mathcal R-\mathcal R'\|_{_{\tilde H^1}}>q \|u_0-\tilde{u}_0\|_{_{\tilde H^1}})\le C \|u_0-\tilde{u}_0\|_{_{\tilde H^1}}\quad \text{for some }q\in(0,1).\]
Here $\|\cdot\|_{_{\tilde H^1}}$ is an equivalent norm on $H^1$ to be specified later.

\end{itemize} 
\medskip

The irreducibility and coupling condition are widely used in the study of mixing; see, e.g.~the monograph {\rm\cite{KS-12}} and references therein. 
Meanwhile, the exponential asymptotic compactness 
is introduced in {\rm\cite{LWX-24}} towards the exponential mixing for nonlinear wave equations. In addition, a qualitative asymptotic compactness is proposed in \cite{Ner-22} to study Navier--Stokes system. 

We also mention that there have been several approaches applied to exponential mixing for various models. For instance, the reader is referred to \cite{HMS-11,HM-08,GM-05} for abstract results 
classified as Harris-type theorems, and \cite{Hairer-02,Mattingly-02,KS-01,KS-02,KPS-02} for some frameworks based on coupling method.

\subsubsection{Verification of hypotheses}\label{Section-strategy-AC} We exhibit the key results related to the three hypotheses in the abstract criterion, and glance through the difficulties and new ingredients.

\medskip

\noindent\textbf{Exponential asymptotic compactness via dispersive dynamics.} This issue is related to the deterministic equation:
\begin{equation*}
iu_t+u_{xx}+ia(x)u=|u|^{p-1}u+f(t,x),
\end{equation*}
where $f\in L_b^2(\R^+;H^{1+\sigma})$; the space $L_b^2(\mathbb{R}^+;H^{1+\sigma})$ consists of functions $f\colon \R^+\rightarrow H^{1+\sigma}$ with
\[\sup_{t\in \mathbb{R}^+} \int_t^{t+1} \|f(t)\|^2_{_{H^{1+\sigma}}}dt<\infty.\]
This boundedness of $f$ is justified by condition (\ref{bounded-noise-0}) on $\eta$. In Theorem \ref{Theorem-AC} we demonstrate that:

\medskip

\begin{center}
{\it There exists an $H^{1+\sigma}$-bounded set $\mathscr B_{1,1+\sigma}$ such that ${\rm dist}_{H^1}(u(t),\mathscr B_{1,1+\sigma})\lesssim e^{-\kappa t}$.}
\end{center}

\medskip

The existing literature on asymptotic compactness concentrates on the setting where $a(x)\equiv a> 0$ is a constant and $f(t,x)=f(x)$ is time-independent~\cite{McConnell-22,ET-13,Goubet-00}. In the present paper, the localized structure of  $a(x)$ leads to main challenge, causing the standard energy method to fail. 

To overcome this issue, we borrow the idea of Carleman estimates from control theory (see, e.g.~Coron \cite{Coron-07}) and invoke nonlinear smoothing effect.
To the best of our knowledge,
this is the first time that Carleman estimates are applied to the study of asymptotic compactness. See Section \ref{Section-Scheme-AC} for more information on our scheme. 

The nonlinear smoothing means the nonlinear part of the solution possesses extra regularity compared to the initial data. Among others, Bourgain \cite{Bourgain-98} and Keraani and Vargas \cite{KV-09} prove nonlinear smoothing on $\mathbb{R}^d$. After removing the troublesome resonance, the revised version of nonlinear smoothing on $\mathbb{T}$ is demonstrated by Erdo\u{g}an and Tzirakis \cite{ET-13} for $p=3$, and McConnell \cite{McConnell-22} for odd $p\geq 5$. Invoking Bourgain spaces (or restricted norm spaces), the crux is 
\[\sup_{k\in \mathbb{Z}} \left[\sum_{\substack{k=k_1-k_2+\cdots +k_p\\\text{No single resonance}}} \frac{\langle k\rangle^{2(s+\sigma)}}{\langle k^2-k_1^2+k_2^2-\cdots -k_p^2\rangle^{1-\varepsilon} \prod_{l=1}^p\langle k_l\rangle^{2s}}\right]<\infty\]
The difficulty lies in the extra power $\sigma>0$ of numerator. Arguments in \cite{McConnell-22} involve normal form reduction and multilinear Strichartz estimate; instead, we provide in Lemma~\ref{Lemma-smoothing} a more elementary proof, the conclusion of which suffices for our purpose. We also mention that nonlinear smoothing will also participate our arguments of $H^s$-dynamics $(s>1)$ and control property.

\medskip

\noindent \textbf{Irreducibility via global stabilization.} Owing to setting $(\mathbf{S2})$, the vanishing force $0$ belongs to the support of $\mathscr{D}(\eta_n)$. As a consequence, the irreducibility for the process $u_n$ follows from the global stabilization for the unforced equation, stated in Proposition \ref{Corollary-stability}:

\medskip

\begin{center}
{\it When $\eta(t,x)\equiv 0$, the energy decays: $E_u(t)\leq CE_u(0) e^{-\beta t}$.}
\end{center}

\medskip

To achieve it, we use Carleman estimate to bound from below the flux term:
\[\int_0^T \int_{\mathbb{T}} a(x)\left(|u(\tau,x)|^2+| u_x(\tau,x)|^2+|u(\tau,x)|^{p+1}\right)dxd\tau\ge cE_u(0).\]
Substituting into the energy identity leads to global stabilization. Furthermore, we extend global stabilization to any higher Sobolev norm $H^s\,(s\ge 1)$, which is new to the literature.

Some prior works along this line can be found in  Laurent \cite{Laurent-ECOCV}, Le Balc'h and Martin \cite{LM-23}, and Rosier and Zhang \cite{RZ-09}, among others. See also Dehman, G\'erard and Lebeau \cite{DGL-06} and Laurent \cite{Laurent-10} for another type of localized damping $a(x)(1-\Delta)^{-1}a(x)\partial_tu$. These results are demonstrated by invoking controllability, observability and unique continuation from control theory.
In particular, the authors in \cite{LM-23} prove the global stabilization in $H^1$ for cubic NLS, and claim that the same method can be adapted for general odd $p\ge 3$.

\medskip

\noindent\textbf{Coupling condition via control.} We verify the coupling condition by establishing a stabilization property for a controlled system associated with (\ref{Random-problem}), reading
\begin{equation}\label{Linear-problem-4}
iu_t+u_{xx}+ia(x)u=|u|^{p-1}u+h(t,x)+\zeta(t,x).
\end{equation}
Here, the force $h$ is fixed, and $\zeta$ stands for the control having the structure similar to that of random force $\eta_n$. In particular, the control $\zeta$ acts essentially on the low frequency of the system and on the open set $\mathcal{I}_2$ in space. 

Let $\tilde u(t)$ stand for an uncontrolled solution (satisfying \eqref{Linear-problem-4} with $\zeta(t,x)\equiv0$), which has extra regularity $H^{1+\sigma}$.
After changing to an equivalent norm $\|\cdot \|_{_{\tilde{H}^1}}$ on $H^1$, Theorem \ref{Theorem-control} amounts to the following control property: 

\medskip

\begin{center}
{\it Given any $u_0\in H^1$, if $\|u_0-\tilde u_0\|_{_{H^1}}$ is small enough, then there is a control $\zeta$ such that 
\begin{equation}\label{stabilization-form}
\|u(T)-\tilde u(T)\|_{_{\tilde{H}^1}}\leq q\|u_0-\tilde u_0\|_{_{\tilde{H}^1}}\quad \text{for some }q\in (0,1).
\end{equation}
}
\end{center}

\medskip

\noindent This is referred to as {\it the stabilization along trajectory} (see Definition \ref{Definition-control}), which has also been invoked for studying, e.g.~the Navier--Stokes equations \cite{BRS-11} and wave equations \cite{ADS-16}. We point out that in some references, this type of property is called ``squeezing'' \cite{Shi-15} or ``$\alpha$-controllability'' \cite{TWY-20}.

When the control is effective on each frequency, the exact controllability is available, namely one can find a control $\zeta$ so that
$$
u(T)=\tilde u(T),
$$ 
even if $\tilde u(t)$ is merely of $H^1$; this gives rise to (\ref{stabilization-form}). The controllability problems for Schr\"{o}dinger equations have attracted considerable attention; see, e.g. \cite{BBZ-13,BZ-19,Laurent-10,DGL-06,RZ-09,WWZ-19,CXZ-23,BL-10}.

In comparison, the main feature of the present problem is the high-frequency degeneracy of control $\zeta$. To overcome it, 
the property (\ref{stabilization-form}) will be achieved in the spirit of ``frequency analysis". Roughly speaking, the controllability is valid for the low-frequency system, where the Hilbert uniqueness method and some results from microlocal analysis are employed.

More importantly, the dissipation property in high frequency is produced, for which the $H^{1+\sigma}$-regularity of $\tilde u(t)$ comes into play. To this end, we observe that nonlinear smoothing contributes to treating the potential terms. We also realize that the extra regularity of $\tilde{u}(t)$ is necessary.
See Section~\ref{Section-scheme} for further discussions.

\begin{remark}
The frequency-analysis strategy has been also used in {\rm\cite{Xiang-23}} for 2D Navier--Stokes system \eqref{NS-problem} and in {\rm\cite{LWX-24}} for nonlinear wave equation \eqref{Wave-problem}. There are two new elements in our treatment for Schr\"odinger equations:

\begin{enumerate}
\item[$(1)$] The issue of high-frequency dissipation is much more difficult, mainly because of the lack of derivative gain in the Duhamel formula. The nonlinear smoothing effect serves as a new ingredient for this problem, which is applied to the potential terms
\[\tfrac{p+1}{2}|\tilde{u}|^{p-1} v+\tfrac{p-1}{2}|\tilde{u}|^{p-3} \tilde{u}^2 \bar{v}.\]

\item[$(2)$] The property of type \eqref{stabilization-form} holds for any given $T>0$, which differs from the case of wave equations {\rm\cite{LWX-24}}.
Our new observation is that, for the locally damped linear Schr\"odinger equation $iu_t+u_{xx}+ia(x)u=0$, we can construct an equivalent norm $\|\cdot \|_{\tilde{H}^s}$ on $H^s\,(s\geq 0)$, so that $\|u(T)\|_{_{\tilde{H}^s}}\leq q_0\|u(0)\|_{_{\tilde{H}^s}}$ with $q_0\in(0,1)$ depending only on $T$. For the explicit definition of $\tilde{H}^s$, see Remark~{\rm\ref{Remark-norm}} and the proof of Lemma~{\rm\ref{Lemma-Hslineardecay}}.
\end{enumerate}
\end{remark}

\subsubsection{Further comments on exponential asymptotic compactness}\label{Section-133}

We conclude this subsection with some remarks on exponential asymptotic compactness from several directions.

\begin{enumerate}[leftmargin=2em]
\item[$\bullet$] (Generality in dispersive equations) A merit of asymptotic compactness is that the behavior of solutions on $\mathcal Y$ (the $H^{1+\sigma}$-bounded attracting set) dominates 
the long-time dynamics of the whole system, even though the solutions do not necessarily enter into $\mathcal Y$. This
accords with the dynamics of dispersive equations, since the smoothing effect lacks and thus the semiflow is non-compact. In comparison, for parabolic PDEs, the asymptotic compactness is often superfluous, as the smoothing effect allows solutions to enter compact sets.

\medskip

\item[$\bullet$] (Optimality in control theory) The reduction to compact space $\mathcal Y$ provides extra regularity required in the stabilization along trajectory. Actually, the $H^{1+\sigma}$-regularity of reference trajectory $\tilde u(t)$ turns out to be both sufficient and necessary for high-frequency dissipation. We deliver a counterexample in Remark~\ref{counterexample} illustrating the sharpness of extra regularity.

\medskip

\item[$\bullet$] 
(Motivation to $H^s$-dynamics) 
We propose nonlinear smoothing as a new tool for $H^s$-dynamics.
Indeed, the idea of exponential asymptotic compactness allows us to promote global dissipation from $H^1$ to higher Sobolev spaces $H^s\, (s>1)$, and specifically implies the existence of corresponding absorbing and attracting sets. 

\medskip

\item[$\bullet$] (Application to probabilistic problems) According to the above observations, asymptotic compactness applies to the mixing problem for random dispersive PDEs. The previous work \cite{LWX-24} and this paper illustrate its applicability to nonlinear wave and Schr\"odinger equations. 
\end{enumerate}

\subsection{Organization of the paper}

In Section~\ref{Section-AC}, we exploit Carleman estimate and nonlinear smoothing to demonstrate the exponential asymptotic compactness and global stabilization at the scale of $H^1$. These results are extended in Section~\ref{Section-stability} to higher Sobolev spaces, with the help of nonlinear smoothing. Next,
we investigate the stabilization along trajectory in Section~\ref{Section-control}, where a number of techniques from control theory come into play, including observability, Hilbert uniqueness method, propagation results from microlocal analysis and the idea of frequency analysis. 

Finally, putting the above results together, we accomplish the proof of the \hyperlink{thm1}{\color{black}Main Theorem} in Section~\ref{Section-proofmain}. This is based on an abstract criterion introduced in the earlier work \cite{LWX-24}. We also state a general method for verifying the coupling condition via control. 

\medskip

The Appendix collects an introduction to Bourgain spaces, as well as some auxiliary results and proofs that are needed in our PDE
and control analyses of the main text.

\subsection{A guide to notations}

We gather here some repeatedly used notations in this paper.

\noindent\textit{$\bullet$ Fourier transform.} For complex-valued function $u\colon \mathbb{T}\to \mathbb{C}$, the Fourier coefficients are
\[\mathcal{F}u(k)=\widehat{u}(k):=\frac{1}{\sqrt{2\pi}} \int_{\mathbb{T}} u(x) e^{-ikx} dx=(u,e_k).\]
Here $(\cdot,\cdot)$ stands for complex $L^2$-inner product. For $u(t,x)$ defined on $\mathbb{R}\times \mathbb{T}$, the space-time Fourier transform is $\mathcal{F}u(\tau,k)$ or $\widehat{u}(\tau,k)$. The inverse Fourier transform is
\[u(t,x)=\frac{1}{2\pi}\sum_{k\in\Z}\int_\R \widehat u(\tau,k) e^{i(kx+t\tau)}d\tau.\]

\medskip

\noindent\textit{$\bullet$ Function spaces.} The Sobolev space $H^s(\mathbb{T})$ is equipped with norm $\|u\|_{_{H^s}}^2=\sum_{k\in\Z}\langle k \rangle^{2s}|\widehat{u}(k)|^2$, where $\langle x\rangle:=\sqrt{1+|x|^2}$.
The orthogonal projection to finite-dimensional subspace
\[H_m:={\rm span}\{e_k;|k|\le m\}\]
is denoted with $P_m$. And set $Q_m:=I-P_m$. Let $S_a(t)$ (and $S(t)$) for $t\in \mathbb{R}$ be the $C_0$-group of operators on $H^s$ generated by $i\partial_x^2-a(x)$ (and $i\partial_x^2$).

For $T>0$, the space-time cylinder is $Q_T:=[0,T]\times \mathbb{T}$. The symbol $L_t^p H_x^s$ is shorthand for $L^p(0,T;H^s(\mathbb{T}))$, when there is no danger of confusion. We also use the space $L_b^2(\mathbb{R}^+;H^s)$ of translation-bounded functions $f\colon \mathbb{R}^+\rightarrow H^s$ such that 
$$
\|f\|^2_{_{L_b^2(\mathbb{R}^+;H^s)}}:=\sup_{t\in\R^+}\int_t^{t+1}\|f(\tau)\|^2_{_{H^s}}d\tau<+\infty
$$
(see, e.g. \cite[Chapter V]{CV-02}). Note that $\int_t^{t+T} \|f(\tau)\|_{_{H^s}}^2 d\tau\le \lceil T\rceil \|f\|_{_{L_b^2(\mathbb{R}^+;H^s)}}^2$ for any $t\in \mathbb{R}^+$.

The Bourgain space $X^{s,b}$ consists of tempered distributions $u$ on $\R\times\T$ for which
\[\|u\|^2_{_{X^{s,b}}}:=\sum_{k\in\Z}\int_\R \langle k\rangle^{2s}\langle \tau+k^2\rangle^{2b}|\widehat u(\tau,k)|^2d\tau<\infty.\]
For the basic properties, see Appendix~\ref{Appendix-Bourgainspace}.
For $T>0$, define $X^{s,b}_T$ to be the restriction space to the time interval $[0,T]$ with norm
$$
\|u\|_{_{X^{s,b}_T}}=\inf\{
\|\tilde u\|_{_{X^{s,b}}};\tilde u=u\ {\rm on\ }[0,T]\times\T
\}<\infty.
$$
For a bounded interval $I$, the  associated restriction space $X^{s,b}_I$ can be defined similarly.

\medskip

\noindent\textit{$\bullet$ Functional analysis.} Let $X$ be a Banach space. Then $B_X(R)$ and $\overline{B}_X(R)$ denote the open and closed balls of radius $R$ centered at the origin, respectively. The pairing between $X$ and $X^*$ is written as $\langle\cdot,\cdot\rangle_{X,X^*}$. The distance from $x\in X$ to a subset $A\subset X$ is $\dist_X(x,A)$. We write $\mathcal{B}(X)$ for the Borel $\sigma$-algebra. The space of bounded continuous functions is $C_b(X)$, equipped with supreme norm $\|\cdot \|_\infty$. And the bounded Lipschitz functions constitute $L_b(X)$, with norm $\|f\|_{_{L_b(X)}}:=\|f\|_{_{\infty}}+\Lip(f)$, where $\Lip(f):=\sup_{x\not =y} |f(x)-f(y)|/\|x-y\|$.

If $X,Y$ are (complex) Banach spaces, the space of bounded linear operators from $X$ to $Y$ is denoted by $\mathcal{L}(X,Y)$ (or $\mathcal{L}(X)$ if $Y=X$).
We define $\mathcal{L}_{\R}(X,Y)$ for real-linear bounded operators. The symbols $X \hookrightarrow Y$ and $X\Subset Y$ refer to continuous and compact embeddings, respectively.

\medskip

\noindent\textit{$\bullet$ Random variables.} Let $X$ be a Polish space (i.e.~separable metric space). The law of $X$-valued random variable $\eta$ is $\mathscr{D}(\eta)$, which belongs to the space of probability measure $\mathcal{P}(X)$. The weak convergence in $\mathcal{P}(X)$ is compatible with the dual Lipschitz distance:
\[\|\mu-\nu\|_L^*:=\sup_{\|f\|_{_{L_b(X)}}\le 1} |\langle f,\mu\rangle-\langle f,\nu\rangle|,\quad \mu,\nu\in \mathcal{P}(X).\]
A coupling between $\mu$ and $\nu$ is a pair of $X$-valued random variables with marginal distributions equal to $\mu$ and $\nu$, respectively. The set of all couplings is denoted by $\mathscr{C}(\mu,\nu)$.

\medskip

\noindent\textit{$\bullet$ Constants.} Various constant $C$ may change from line to line. The dependence on parameters are represented by $C(\cdot)$, which always means a non-decreasing function of such parameter. 

The parameters $b,b'$, which appear in Bourgain spaces $X^{s,b}$ and $X^{s,-b'}$, satisfy $0<b'<1/2<b<1$ and $b+b'\leq 1$. The only additional assumption on $b,b'$ occurs in Lemma~\ref{Lemma-smoothing}. Hence it is reasonable to consider them as given, once and for all.

\section{Exponential asymptotic compactness in $H^1$}\label{Section-AC}

In this section, we investigate exponential asymptotic compactness for a deterministic version of NLS equation (\ref{Random-problem}), at the scale of $H^1$. A by-product is the global stabilization in $H^1$. As described in Section~\ref{Section-strategy}, these results indicate the exponential asymptotic compactness and irreducibility of (\ref{Random-problem}) in our abstract criterion; see Section~\ref{Section-verification} for details. 

The equation considered here reads
\begin{equation}\label{Nonlinear-problem}
\left\{\begin{array}{ll}
iu_t+u_{xx}+ia(x)u=|u|^{p-1}u+f(t,x),\\
u(0,x)=u_0(x),
\end{array}\right.
\end{equation}
where $u_0\in H^s$ with $s\geq 1$, and $f\colon [0,T]\rightarrow H^s$ (or $f\colon \R^+\rightarrow H^s$) is a deterministic force. Unless otherwise stated, we consider mild solutions, namely $u\in C([0,T];H^s)$ satisfying the Duhamel formula. As $H^s$ is a Banach algebra, standard argument implies that \eqref{Nonlinear-problem} admits a unique solution $u\in C([0,T];H^s)$. We define the functional $E\colon H^1\rightarrow\R^+$ to be the $H^1$-energy:
\begin{equation}\label{energy-function}
E(u):=\frac{1}{2}\int_{\T}|u|^2+\frac{1}{2}\int_{\T}| u_x|^2+\frac{1}{p+1}\int_{\T}|u|^{p+1}.
\end{equation}
If $u(t)$ is a solution, we denote $E_u(t)=E(u(t))$.

The main result of this section is stated as follows.

\begin{theorem}[Exponential asymptotic compactness]\label{Theorem-AC}
Let $R_0>0$ and $\sigma\in(0,1/4]$ be arbitrarily given. Then there exists a bounded subset $\mathscr B_{1,1+\sigma}$ of $H^{1+\sigma}$ and constants $C,\kappa>0$ such that 
\begin{equation*}
{\rm dist}_{H^1}(u(t),\mathscr B_{1,1+\sigma})\leq C\left(
1+E(u_0)
\right)e^{-\kappa t},\ \forall t\geq 0
\end{equation*}
for any $u_0\in H^1$ and $f\in \overline{B}_{L^2_b(\R^+;H^{1+\sigma})}(R_0)$, where $u(t)$ stands for the solution of {\rm(\ref{Nonlinear-problem})}. 
\end{theorem}

In the sequel we also deduce the global stabilization for unforced equation (i.e.~\eqref{Nonlinear-problem} with $f(t,x)\equiv 0$), in which case the attracting set $\mathscr B_{1,1+\sigma}$ in Theorem \ref{Theorem-AC} reduces to singleton $\{0\}$.

\begin{proposition}[Global stabilization]\label{Corollary-stability}
There exist constants $C,\beta>0$ such that 
\begin{equation}\label{H^1stability}
E_u(t)\leq CE_u(0)e^{-\beta t},\ \forall t\geq 0
\end{equation}
for any $u_0\in H^1$, where $u(t)$ stands for the solution of \eqref{Nonlinear-problem} with $f(t,x)\equiv 0$.
\end{proposition}

\begin{remark}
In Section {\rm\ref{Section-stability}}, the conclusions of Theorem~{\rm\ref{Theorem-AC}} and Proposition~{\rm\ref{Corollary-stability}} will be extended to $H^s$ with $s\ge 1$ and any $\sigma>0$. 
Let us present brief statements on this aspect beforehand.
\begin{enumerate}
\item[$(1)$] Roughly speaking, asymptotic compactness in $H^s$ means that given any $s\geq 1$ and $\sigma>0$, there exists an $H^{s+\sigma}$-bounded set attracting exponentially the solutions of \eqref{Nonlinear-problem} in $H^s$. 

\item[$(2)$] When $f(t,x)\equiv 0$, the $H^s$-norm of solutions will be proved to decay exponentially.
\end{enumerate}
\end{remark}

At some stage of the proof, we need to exploit Bourgain spaces $X^{s,b}$. Unlike the space-time Sobolev spaces $L_t^p H_x^s$, the Bourgain spaces $X^{s,b}$ reflect the connection between the temporal and spatial frequencies of solution by means of the linear dispersive relation. In particular, Bourgain spaces allow one to exploit smoothing properties of nonlinearity $|u|^{p-1}u$, which are implicit under the setting of usual Sobolev spaces, and play a crucial role to our analysis of dynamics and control property.
Some basic properties of Bourgain spaces, as well as the global well-posedness of NLS equation \eqref{Nonlinear-problem} in $X^{s,b}$ can be found in Appendix~\ref{Appendix-GWP}. Specifically, according to Proposition~\ref{Proposition-nonlinear}, the mild solution $u$ of NLS equation \eqref{Nonlinear-problem} belongs to $X_T^{s,b}$ for any $b\in (1/2,1)$.

\medskip

An overview of the proof of Theorem \ref{Theorem-AC} is presented in Section~\ref{Section-Scheme-AC}, and the details are included in the later Sections \ref{Section-Carleman}--\ref{Section-AC-2}. The proof of Proposition~\ref{Corollary-stability} lies at the end of Section~\ref{Section-dissipation}.

\subsection{Scheme of proof}\label{Section-Scheme-AC}

As in the general theory of global attractor, the proof of Theorem~\ref{Theorem-AC} can be summarized as two parts: 

\medskip

\noindent $\bullet$ {\bf Global dissipation}. We  construct an $H^1$-bounded set $\mathscr{B}_1$ which is absorbing for (\ref{Nonlinear-problem}):
$$
u(t)\in \mathscr{B}_1,\ \forall t\geq T_0(\|u_0\|_{_{H^1}}).
$$
The core issue to be addressed is the localized structure of $a(x)$; standard energy estimate via integration by parts is not sufficient.

\medskip

\noindent $\bullet$ {\bf Nonlinear smoothing}. By exploiting the nonlinear smoothing effect typical in dispersive equations, we  demonstrate that the nonlinear part in the Duhamel formula
$$
\int_0^tS_a(t-\tau)(|u|^{p-1}u)d\tau,
$$
after removing the single-resonant terms, gains extra regularity and is bound uniformly in $H^{1+\sigma}$ for $u_0\in\mathscr{B}_1$. The desired conclusion is then obtained.

\medskip

\noindent {\it Part I: Global dissipation}. 
A classical strategy of finding absorbing set is energy dissipation, which is standard when $a(x)\equiv a>0$. However, the case where $a(x)$ is localized in $\mathcal I_1$ is much more complicated, as the flux term in energy identity \eqref{energy-2}, which reads
\[\int_{\T}a(x)\left(|u|^2+| u_x|^2+|u|^{p+1}\right),\]
is no longer bounded from below by $E_u(t)$. 
To address the above issue, we propose a two-step procedure inspired by the controllability and stabilization theory of NLS (see, e.g.~\cite{Laurent-10,LM-23}).

\medskip

{\it Step 1: Carleman estimate}.
We first in Section \ref{Section-Carleman} propose a Carleman-type estimate for the general nonlinear Schr\"odinger equation 
\[iu_t+u_{xx}=g(|u|^2)u+h(t,x),\]
where $g$ stands for the nonlinearity and $h$ represents the source or lower-order terms.
The general estimate is valid not only for the nonlinear NLS equation (\ref{Nonlinear-problem}) considered here, but also for the linear Schr\"{o}dinger equation with complex-valued potentials. Its application in the latter case will participate the analysis of control problems (Section \ref{Section-control}).

In particular, applying the general Carleman estimate to (\ref{Nonlinear-problem}) gives rise to 
\begin{equation}\label{Carleman-form}
\begin{aligned}
&\int_{0}^T\int_\T\left(\theta_1|u|^2+\theta_2| u_x|^2+\theta_3|u|^{p+1}\right)\\
&\lesssim \int_{0}^T\int_{\mathcal I_1}\left(\theta_1|u|^2+\theta_2| u_x|^2+\theta_3|u|^{p+1}\right)+(\text{lower-order terms}),
\end{aligned}
\end{equation}
where the weight functions $\theta_i=\theta_i(t,x)>0$ will be appropriately chosen. The estimate (\ref{Carleman-form}) indicates that the localized dissipation originated from the subdomain $\mathcal I_1$ can spread to the entire system, which is fundamental to establishing the desired global dissipation.

\medskip

{\it Step 2: dissipative estimate}. With the Carleman estimate above in hand, we are able to propose a dissipative estimate for $E_u(t)$. More precisely, we derive from (\ref{Carleman-form}) the flux estimate
\begin{equation}\label{flux-form}
E_u(t)\lesssim \int_0^T \int_{\mathbb{T}} a(x)\left(|u|^2+| u_x|^2+|u|^{p+1}\right)+(\text{lower-order terms}),\ \forall t\in[0,T].
\end{equation}
If it is the unforced case (i.e. $f(t,x)\equiv 0$), the lower-order terms vanish, and the corresponding estimate has been used in \cite{LM-23} for deriving the global stabilization of cubic NLS. The same argument based on \eqref{flux-form} would also yield Proposition~\ref{Corollary-stability}.

In the forced case, we apply (\ref{flux-form}) to deduce a dissipative estimate:
\begin{equation*}
\tilde E_u(T)\leq q\tilde E_u(0)+C(T,\|f\|_{_{L^2_b(\mathbb{R}^+;H^1)}})\quad \text{for some }q\in(0,1),
\end{equation*}
where $\tilde E_u(t)$ represents a modified $H^1$-energy function being equivalent to $E_u(t)$. This estimate enables us to construct the desired global absorbing set for (\ref{Nonlinear-problem}). See Section \ref{Section-dissipation} for the details.

\medskip

\noindent {\it Part II: Nonlinear smoothing}. We exploit the phenomenon of nonlinear smoothing for $|u|^{p-1}u$, at the scale of Bourgain space $X^{s,b}$. Define the $p$-multiplication operator by setting
\[\mathcal T(u_1,\cdots,u_p)=\prod_{l\, {\rm odd}}u_l \prod_{l\, {\rm even}}\bar u_l.\]
Inspired by \cite{ET-13,McConnell-22} we propose that for $s\geq 1$, $\sigma\in (0,1/4]$, $b>1/2$ and $b'\in [\sigma,1/2)$,
\begin{equation}\label{smoothing-form}
\|\mathcal T(u_1,\cdots,u_p)-\mathcal T_R(u_1,\cdots,u_p)\|_{_{X^{s+\sigma,-b'}}}\lesssim \prod_{l=1}^p\|u_l\|_{_{X^{s,b}}}
\end{equation}
Here, the term $\mathcal T_R$ represents the frequencies of ``single resonance'', defined rigorously later and having the {\it same spatial regularity} as $u_l$. See Section \ref{Section-AC-1} for more details. In addition, we observe that such an estimate will be also useful in studying the $H^s$-dynamics (Section \ref{Section-stability}) and the  stabilization along trajectory for  associated controlled system (Section \ref{Section-control}).

The estimate (\ref{smoothing-form}) (in the case of $s=1$ and $u_l=u$) applied to equation (\ref{Nonlinear-problem}) yields
$$
\||u|^{p-1}u-\tfrac{p+1}{4\pi}\|u\|_{_{L^{p-1}(\mathbb{T})}}^{p-1}u\|_{_{X^{1+\sigma,-b'}}}\lesssim \|u\|_{_{X^{1,b}}}^p.
$$
The term $\frac{p+1}{4\pi}\|u\|_{_{L^{p-1}(\mathbb{T})}}^{p-1}u$ can be dealt with by exploiting the norm-preserving transformation
\begin{equation}\label{w-variable}
U(t,x)=e^{i\theta(t)}u(t,x)\quad \text{with }\theta(t)=\frac{p+1}{4\pi}\int_0^t\|u(s)\|^{p-1}_{_{L^{p-1}(\T)}}ds\in\R.
\end{equation}
As a result, one can construct a bounded subset $\mathscr B_{1,1+\sigma}$ of $H^{1+\sigma}$, such that
$$
u(t)-e^{-i\theta(t)}S_a(t)u_0\in \mathscr B_{1,1+\sigma},\ \forall t\geq 0
$$
whenever $u_0\in\mathscr{B}_1$. This together with the absorbing property of $\mathscr{B}_1$ (Part I) and decay of $S_a(t)$ concludes the desired result of asymptotic compactness. See Section \ref{Section-AC-2} for more details.

\medskip

\noindent{\it Convention: In the remainder of this section, unless otherwise stated, the generic constant $C$ used in the proofs depends only on the time period $T$ and the size $R_i$ of force. For the sake of simplicity we shall omit such dependence if there is no danger of confusion.}

\subsection{General Carleman estimate}\label{Section-Carleman}

In this subsection, we state a technical result which is crucial to both global dissipation and control theory. We temporarily consider the general nonlinear Schr\"odinger equations of the form
\begin{equation}\label{Problem-Carleman}
iu_t+u_{xx}=g(|u|^2)u+h(t,x).
\end{equation}
Our setting for the functions $g$ and $h$ will cover the following two cases:
\begin{itemize}
\item[$(1)$] $g(r)=r^{(p-1)/2}$ and $h(t,x)=-ia(x)u+f(t,x)$, which is the case of \eqref{Nonlinear-problem}. The corresponding conclusion is fundamental to proving the flux estimate in Section {\rm\ref{Section-dissipation}}.

\item[$(2)$] $g(r)\equiv 0$ and $h(t,x)=V_1(t,x)u+V_2(t,x) \bar{u}$, where $V_i$ stand for complex-valued potentials. 
The corresponding conclusion will indicate the unique continuation property for the linear Schr\"odinger equation, which is useful in  establishing observability inequalities for the associated controlled system; see Section {\rm\ref{Section-LF}} and Appendix {\rm\ref{Appendix-control}}.
\end{itemize}

For an arbitrarily given open subset $\mathcal I$ of $\T$,
we pick up a function $\phi\in C^\infty(\T;\R)$ which is linear outside of $\mathcal{I}$. More precisely, there exists a constant $c>0$ with
\begin{equation}\label{weight-function-0}
\phi'(x)=c,\ \phi''(x)=0,
\ \forall x\in \T\setminus\overline{\mathcal I}.
\end{equation}
For $T>0$ and $\lambda\geq 1$ we also introduce the functions $\alpha,\beta\colon Q_T\rightarrow \R^+$ by
\begin{equation*}
\alpha(t,x)=\frac{e^{4\lambda \|\phi\|_{_{\infty}}}-e^{\lambda(\phi(x)+2\|\phi\|_{_{\infty}})}}{t(T-t)},
\quad
\beta(t,x)=\frac{e^{\lambda(\phi(x)+2\|\phi\|_{_{\infty}})}}{t(T-t)}.
\end{equation*}

\begin{proposition}[Carleman estimate]\label{Proposition-Carleman}
Let $h\in L^2(0,T;H^1)$. Assume that the nonlinear function $g\colon \mathbb{R}^+\to \mathbb{R}$ is smooth, and its primitive function $G(r):=\int_0^rg(\tau)d\tau$ satisfies 
\[\Psi(r):=rg(r)-G(r)\ge 0,\ \forall r\ge 0.\]
Then there exists $C>0$ such that for every $T>0$, there are constants $s_0,\lambda_0\geq 1$ satisfying 
\begin{equation*}
\begin{aligned}
&s^3\lambda^4\int_{Q_T}\beta^3 e^{-2s\alpha}|u|^2+s\lambda^2 \int_{Q_T}\beta e^{-2s\alpha}| u_x|^2+s^2\lambda^2\int_{Q_T}\beta^2e^{-2s\alpha}\Psi(|u|^2)\\
&\le C\left[s^3\lambda^4\int_{q_T}\beta^3e^{-2s\alpha}|u|^2+s\lambda^2\int_{q_T}\beta e^{-2s\alpha}| u_x|^2+s^2\lambda^2\int_{q_T}\beta^2e^{-2s\alpha}\Psi(|u|^2)+\int_{Q_T}e^{-2s\alpha}|h|^2\right]
\end{aligned}
\end{equation*}
for $s\geq s_0$ and $\lambda\geq \lambda_0$, where $q_T:=[0,T]\times\mathcal I$ and $u$ stands for a smooth solution of {\rm(\ref{Problem-Carleman})}.
\end{proposition}

We point out that when the equation \eqref{Problem-Carleman} is well-posed, then this proposition actually holds for any mild solution $u\in C([0,T];H^1)$, according to standard approximations. 

The proof of Proposition \ref{Proposition-Carleman} follows by adapting the arguments from \cite{Laurent-10,LM-23} (see also, e.g.~\cite{IK-04}, for related work on Carleman inequalities for Schr\"{o}dinger equations). Indeed, \cite{Laurent-10} addresses the linear Schr\"{o}dinger equations with potentials ($g\equiv 0$ and $h=V_1u+V_2 \bar{u}$), while \cite{LM-23} considers cubic NLS ($g(r)=r$ and $h=-iau+f$). Since the method is independent from our main purpose, we leave the details to Appendix~\ref{appendix_Carleman}.

\subsection{Global dissipation}\label{Section-dissipation}

This subsection establishes $H^1$-absorbing set for NLS equation \eqref{Nonlinear-problem}.

\begin{proposition}[$H^1$-absorbing set]\label{Theorem-absorbing}
Let $R_1>0$ be arbitrarily given. Then there exists a bounded subset $\mathscr B_1$ of $H^1$ and a constant $C_1>0$ such that
\begin{equation}\label{inclusion-absorbing}
u(t)\in\mathscr B_1,\ \forall t\geq C_1\log (1+E_u(0))
\end{equation}
for any $u_0\in H^1$ and $f\in \overline{B}_{L^2_b(\R^+;H^1)}(R_1)$, where $u(t)$ stands for the solution of {\rm(\ref{Nonlinear-problem})}.
\end{proposition}

We recall two identities on the $L^2$- and $H^1$-energies of (\ref{Nonlinear-problem}), respectively. 
Firstly, multiplying \eqref{Nonlinear-problem} by $u$ and taking the imaginary part, we get
\begin{equation}\label{energy-1}
\frac{1}{2}\frac{d}{dt} \int_{\T}|u|^2=-\int_{\T}a(x)|u|^2+\int_{\T}{\rm Im}(f\bar u).
\end{equation}
Secondly, multiplying \eqref{Nonlinear-problem} by $\bar{u}_t$ and taking the real part; then in order to remove the term $\re \, (\int_{\mathbb{T}} ia(x)u\bar{u}_t)$, multiplying \eqref{Nonlinear-problem} by $a(x)\bar{u}$ and taking the real part, we get
\begin{equation}\label{energy-2}
\begin{aligned}
\frac{d}{dt}E_u(t)&=-\int_{\T}a(x)\left(|u|^2+| u_x|^2+|u|^{p+1}\right)+\frac{1}{2}\int_{\T}a''(x)|u|^2\\
&\quad +\int_{\T}\left(
{\rm Im}(f\bar u)-a(x){\rm Re}(f\bar u)-{\rm Re}(f\bar u_t)
\right).
\end{aligned}
\end{equation}

As previously mentioned, we need the following technical lemma.

\begin{lemma}[Flux estimate]\label{Lemma-fluxestimate}
Given $T>0$, there exists a constant $C_2>0$ such that
\begin{equation*}
\begin{aligned}
E_u(t)\leq  &\, C_2\left[\int_{Q_T}a(x)\left(|u|^2+| u_x|^2+|u|^{p+1}\right)+\int_{Q_T}|f|^2+\int_0^T\|f(t)\|_{_{H^1(\mathbb{T})}}\left(
E_u^{1/2}+E_u^{p/(p+1)}
\right)\right]
\end{aligned}
\end{equation*}
for any $t\in [0,T]$, $u_0\in H^1$ and $f\in L^2(0,T;H^1)$, where $u(t)$ stands for the solution of {\rm(\ref{Nonlinear-problem})}.
\end{lemma}

\begin{proof}[{\bf Proof of Lemma \ref{Lemma-fluxestimate}}]
Due to well-posedness in $H^s\,(s\geq 1)$, it suffices to consider smooth solutions. Indeed, one can use smooth initial data and forces to approximate $H^1$-solution. 

Substituting $u_t$ into \eqref{energy-2} via \eqref{Nonlinear-problem}, we observe that
\begin{equation*}
\begin{aligned}
\left|\frac{d}{dt} E_u(t)\right|&\le C\left[\int_{\mathbb{T}}a(x)(|u|^2+| u_x|^2+|u|^{p+1})+\int_{\mathbb{T}}|u|^2 \right.\\
&\qquad\quad+\|f(t)\|_{_{L^2}}
E_u^{1/2}+\|f(t)\|_{_{L^\infty}}E_u^{p/(p+1)}\\
&\qquad\quad\left.+|\int_\T \re [f\left(-i\bar u_{xx}-a\bar u+i|u|^{p-1}\bar u+i\bar f\right)]|\right].\\
\end{aligned}
\end{equation*}
To treat the last integral, note that
\[|\int_{\mathbb{T}} \re (-if\bar{u}_{xx})|\le \int_{\mathbb{T}} |f_x \bar{u}_x|\le \|f(t)\|_{_{H^1}} \|u(t)\|_{_{H^1}} \le \sqrt{2}\|f(t)\|_{_{H^1}} E_u^{1/2}(t).\]
Owing to $\re (f\cdot i\bar{f})=0$ and $H^1(\mathbb{T})\hookrightarrow L^\infty(\mathbb{T})$, we obtain
\begin{equation*}
\begin{aligned}
E_u(t)&\leq  E_u(t')+C\left[\int_{Q_T}a(x)(|u|^2+| u_x|^2+|u|^{p+1})+\int_{Q_T}|u|^2\right.\\ 
&\qquad\qquad\qquad\quad\left.+\int_0^T\|f\|_{_{H^1}}\left(
E_u^{1/2}+E_u^{p/(p+1)}
\right)\right]
\end{aligned}
\end{equation*}
for any $t,t'\in [0,T]$. Integrating with respect to $t'\in[T/4,3T/4]$, it then follows that
\begin{equation}\label{estimate-5}
\begin{aligned}
E_u(t)&\leq  C\left[\int_{T/4}^{3T/4}E_u+\int_{Q_T}a(x)(|u|^2+| u_x|^2+|u|^{p+1})\right.\\
&\,\quad\quad\left.+\int_{Q_T}|u|^2 +\int_0^T\|f\|_{_{H^1}}\left(
E_u^{1/2}+E_u^{p/(p+1)}
\right)\right].
\end{aligned}
\end{equation}
Obviously, it suffices to deal with the first and third integrals on the RHS of (\ref{estimate-5}).

Applying Proposition \ref{Proposition-Carleman} (with $\mathcal I=\mathcal I_1$) to NLS equation (\ref{Nonlinear-problem}), one can 
deduce that
\begin{equation*}
\begin{aligned}
&s^3\lambda^4\int_{Q_T}\beta^3e^{-2s\alpha}|u|^2+s\lambda^2\int_{Q_T}\beta e^{-2s\alpha}| u_x|^2+s^2\lambda^2\int_{Q_T}\beta^2e^{-2s\alpha}|u|^{p+1}\\
&\le C\left[ s^3\lambda^4\int_{q_T}\beta^3e^{-2s\alpha}|u|^2+s\lambda^2\int_{q_T}\beta e^{-2s\alpha}| u_x|^2+s^2\lambda^2\int_{q_T}\beta^2e^{-2s\alpha}|u|^{p+1}\right.\\
&\quad \quad\quad\left.+\int_{Q_T}e^{-2s\alpha}|-ia(x)u+f|^2\right],
\end{aligned}
\end{equation*}
for $s,\lambda$ large enough. Thanks to setting $(\mathbf{S1})$, as $a(x)\ge a_0>0$ on $\mathcal{I}_1$, we have
\begin{equation}\label{Obs-1}
\begin{aligned}
\int_{T/4}^{3T/4}E_u&\leq C \left[\int_{Q_T}a(x)\left(|u|^2+| u_x|^2+|u|^{p+1}\right)+\int_{Q_T}|f|^2\right].
\end{aligned}
\end{equation}

At the same time, one can recall (\ref{energy-1}) to find that
\begin{align*}
\|u(t)\|^2_{_{L^2}}&\leq \|u(t')\|^2_{_{L^2}}+2\int_{Q_T}a(x)|u|^2+2\int_{Q_T}|f||u|\\
&\leq \|u(t')\|^2_{_{L^2}}+2\int_{Q_T}a(x)|u|^2+2\int_0^T \|f(t)\|_{_{L^2}} \|u(t)\|_{_{L^2}}
\end{align*}
for any $t,t'\in[0,T]$. Integrating with respect to $t'\in[T/4,3T/4]$ and noticing (\ref{Obs-1}), we find
\begin{equation}\label{L2-estimate}
\|u(t)\|^2_{_{L^2}}\leq C\left[\int_{Q_T}a(x)\left(|u|^2+| u_x|^2+|u|^{p+1}\right)+\int_{Q_T}|f|^2+\int_0^T \|f\|_{_{H^1}} E_u^{1/2}\right].
\end{equation}

Finally, substituting (\ref{Obs-1}) and (\ref{L2-estimate}) into (\ref{estimate-5}), we conclude the proof.
\end{proof}

With the help of the flux estimate, we proceed to verify a dissipative estimate. To this end, for $T>0$ arbitrarily given, we introduce a modified energy functional
\begin{equation}\label{modifiedenergy}
\tilde{E}_u(t):=E_u(t)+\frac{L}{2}\|u(t)\|_{L^2}^2.
\end{equation}
Here, $L=L(T)$ is a constant specified as follows:
\begin{itemize}
\item Let $\varepsilon_0:=C_2^{-1}T^{-1}/2>0$, where $C_2=C_2(T)$ appeared in Lemma~\ref{Lemma-fluxestimate}. One can derive by contradiction that, there exists a constant $L>0$ such that (cf.~\cite[Lemma 4.1]{LM-23})
\begin{equation}\label{estimate-17}
|a'(x)|^2\leq \varepsilon_0^2+2\varepsilon_0 La(x),\ \forall x\in\T.
\end{equation}
\end{itemize}
In particular, the modified energy is equivalent to $E_u(t)$, as obviously
\begin{equation}\label{energy-equivalent}
E_u(t)\le \tilde{E}_u(t)\le (1+L)E_u(t).
\end{equation}

\begin{lemma}[$H^1$-dissipation]\label{Lemma-discretemono}
Let $T,R_1>0$ be arbitrarily given. Then there exist constants $q\in (0,1)$ (depending only on $T$) and $C>0$, such that
\[\tilde{E}_u(T)\le q\tilde{E}_u(0)+C\|f\|_{_{L^2_b(\mathbb{R}^+;H^1)}}\]
for any $u_0\in H^1$ and $f\in\overline{B}_{L^2_b(\R^+;H^1)}(R_1)$, where $u(t)$ stands for the solution of \eqref{Nonlinear-problem}.
\end{lemma}

\begin{proof}[{\bf Proof of Lemma \ref{Lemma-discretemono}}]
The proof will be divided into two steps.

\textit{Step 1: apriori estimate for $\tilde{E}_u(t)$.} Taking the energy identities \eqref{energy-1} and \eqref{energy-2} into account, and substituting $u_t$ via \eqref{Nonlinear-problem}, and noticing $H^1(\mathbb{T})\hookrightarrow L^\infty(\mathbb{T})$,
we find that for $f\in\overline{B}_{L^2_b(\R^+;H^1)}(R_1)$,
\begin{equation*}
\begin{aligned}
\frac{d}{dt}\tilde{E}_u(t)&=L\cdot \text{(RHS of \eqref{energy-1})}+\text{(RHS of \eqref{energy-2})}\\
&\leq C\left[\|f(t)\|_{_{L^2}}E_u^{1/2}(t)+E_u(t)+|\int_\T \re [f\left(-i\bar u_{xx}-a\bar u+i|u|^{p-1}\bar u+i\bar f\right)]|\right]\\
&\le C\left[E_u(t)+C\|f(t)\|_{_{H^1}}(E_u^{1/2}(t)+E_u^{p/(p+1)}(t))
\right]\\
&\le C(\|f\|_{_{H^1}}+1)(E_u(t)+1)\le C(\|f\|_{_{H^1}}+1)(\tilde{E}_u(t)+1).
\end{aligned}
\end{equation*}
Thanks to the Gronwall inequality, we find a constant $K=K(T,R_1)>0$ satisfying
\begin{equation}\label{E_uapriori}
\tilde{E}_u(t)\le K(\tilde{E}_u(0)+1),\ \forall t\in [0,T].
\end{equation}

\textit{Step 2: apply flux estimate.} Using (\ref{energy-2}) again we derive that
\begin{equation}\label{estimate-15}
\begin{aligned}
E_{u}(T)-E_{u}(0)
\leq &\,-\int_{Q_T}a(x)(|u|^2+| u_x|^2+|u|^{p+1})+\frac{1}{2}\int_{Q_T}a''(x)|u|^2\\
&\, +C \int_0^T \|f\|_{_{H^1}}\left(E_{u}^{1/2}+E_{u}^{p/(p+1)}\right).
\end{aligned}
\end{equation}
Then, we apply Lemma \ref{Lemma-fluxestimate} in two ways: first by integrate over $t\in [0,T]$, and second by taking $t=0$. Summing up two estimates yields that
\begin{equation}\label{estimate-14}
\begin{aligned}
-\int_{Q_T}a(x)(|u|^2+| u_x|^2+|u|^{p+1})\le &-\frac{C_2^{-1}T^{-1}}{2}\int_0^TE_{u}-\frac{C_2^{-1}}{2} E_u(0)\\
&+C\int_{Q_T} |f|^2+C\int_0^T \|f\|_{_{H^1}}\left(E_{u}^{1/2}+E_{u}^{p/(p+1)}\right).
\end{aligned}
\end{equation}
We also deduce by (\ref{estimate-17}) that
\begin{align*}
    &\frac{1}{2}\int_{Q_T}a''(x)|u|^2=-\frac{1}{2}\int_{Q_T} a'(x)\partial_x|u|^2\leq \int_{Q_T}|a'(x)||u_x||u|\\
    &\leq \frac{\varepsilon_0}{2} \int_{Q_T}| u_x|^2+\frac{1}{2\varepsilon_0}\int_{Q_T}| a'(x)|^2|u|^2\leq \frac{\varepsilon_0}{2}\int_{Q_T} |u_x|^2 +\frac{\varepsilon_0
}{2}\int_{Q_T}|u|^2  +L\int_{Q_T}a(x)|u|^2.
\end{align*}
Substituting the last expression by \eqref{energy-1} yields
\begin{equation}\label{estimate-16}
\frac{1}{2}\int_{Q_T}a''(x)|u|^2\leq \varepsilon_0 
\int_0^TE_{u}+
\frac{L}{2}\|u(0)\|^2_{_{L^2}}-\frac{L}{2}\|u(T)\|^2_{_{L^2}}+C\int_0^T \|f\|_{_{H^1}}E_{u}^{1/2}.
\end{equation}

Putting (\ref{estimate-15})-(\ref{estimate-16}) together, and invoking \eqref{energy-equivalent} and \eqref{E_uapriori}, we obtain  (recall $\varepsilon_0=C_2^{-1}T^{-1}/2$)
\begin{align*}
    \tilde{E}_{u}(T)-\tilde{E}_{u}(0)&\le -\frac{C_2^{-1}}{2} E_{u}(0)+C\int_{Q_T} |f|^2+C \int_0^T \|f\|_{_{H^1}}\left(E_{u}^{1/2}+E_{u}^{p/(p+1)}\right)\\
    &\le -\frac{C_2^{-1}}{2(1+L)} \tilde{E}_{u}(0)+C\|f\|_{_{L^2_b(\mathbb{R}^+;H^1)}}^2+C\|f\|_{_{L^2_b(\mathbb{R}^+;H^1)}}(\tilde{E}^{1/2}_{u}(0)+\tilde{E}^{p/(p+1)}_{u}(0)+1).
\end{align*}
Thanks to the Young inequality and $\|f\|_{_{L^2_b(\mathbb{R}^+;H^1)}}\le R_1$, we have
\begin{align*}
C\|f\|_{_{L^2_b(\mathbb{R}^+;H^1)}}\left(\tilde{E}^{1/2}_{u}(0)+\tilde{E}^{p/(p+1)}_{u}(0)\right)&\le \frac{C_2^{-1}}{4(1+L)} \tilde{E}_u(0)+C(\|f\|_{_{L^2_b(\mathbb{R}^+;H^1)}}^2+\|f\|_{_{L^2_b(\mathbb{R}^+;H^1)}}^{p+1})\\
&\le \frac{C_2^{-1}}{4(1+L)} \tilde{E}_u(0)+C\|f\|_{_{L^2_b(\mathbb{R}^+;H^1)}}.
\end{align*}
We conclude that
\[\tilde{E}_u(T)\le \left(1-\frac{C_2^{-1}}{4(1+L)}\right) \tilde{E}_u(0)+C\|f\|_{_{L^2_b(\mathbb{R}^+;H^1)}}.\]
Thus the conclusion of this Lemma follows with $q=1-\frac{C_2^{-1}}{4(1+L)}\in (0,1)$.
\end{proof}

Now both Proposition~\ref{Corollary-stability} and Proposition~\ref{Theorem-absorbing} are direct consequences of Lemma~\ref{Lemma-discretemono}.

\begin{proof}[{\bf Proof of Proposition~\ref{Corollary-stability}}]
    By setting $T=1$ and $f(t,x)\equiv 0$ in Lemma~\ref{Lemma-discretemono}, we get $\tilde{E}_u (n+1)\le q\tilde{E}_u(n)$ for some $q\in (0,1)$. In view of the apriori estimate
    \[\tilde{E}_u(n+t)\le K\tilde{E}_u(n),\ \forall t\in [0,1],\]
    whose proof is similar to that of \eqref{E_uapriori}, the conclusion easily follows with the help of \eqref{energy-equivalent}.
\end{proof}

\begin{proof}[{\bf Proof of Proposition \ref{Theorem-absorbing}}]
Let $q\in(0,1)$ and $C>0$ be established by Lemma \ref{Lemma-discretemono} with $T=1$ and the arbitrarily given $R_1$. Iterating for $n$ times, we find
\begin{equation}\label{uniform-bound}
\tilde{E}_u(n)\le q^n\tilde{E}_u(0)+\sum_{k=0}^{n-1} q^k C R_1=q^n\tilde{E}_u(0)+\frac{CR_1}{1-q}.
\end{equation}
In particular, if $n\ge |\log q|^{-1} \log (1+\tilde{E}_u(0))-1$, then $\tilde{E}_u(n)\le q^{-1}+CR_1/(1-q)$.

Thanks to the apriori estimate \eqref{E_uapriori}, there exists a constant $K=K(R_1)>0$ such that
\begin{equation*}
\tilde{E}_u(n+t)\le K (\tilde{E}_u(n)+1),\ \forall t\in [0,1].
\end{equation*}
Therefore, we conclude that with the bounded set $\mathscr{B}_1\subset H^1$ defined by
\[\mathscr{B}_1=\{v\in H^1; \tilde{E}(v)\le K(q^{-1}+CR_1/(1-q)+1)\},\]
we have $u(t)\in \mathscr{B}_1$ for $t\ge |\log q|^{-1}\log (1+\tilde{E}_u(0))$. And the proof is complete owing to \eqref{energy-equivalent}.
\end{proof}

\begin{remark}\label{boundednessH^1}
From the proof one can in fact derive the uniform-in-time boundedness of solution map $(u_0,f)\mapsto u(t)${\rm:} given any $R>0$, there exists a constant $C=C(R)>0$ such that 
\begin{equation*}
\|u(t)\|_{_{H^1}}\le C,\ \forall t\ge 0
\end{equation*}
for any $u_0\in \overline{B}_{_{H^1}}(R)$ and $f\in \overline{B}_{_{L^2_b(\mathbb{R}^+;H^1)}}(R)$.
\end{remark}

\subsection{Nonlinear smoothing}\label{Section-AC-1}

This subsection includes a result on nonlinear smoothing. To this end, we first introduce the resonant decomposition of the multiplication operator $\mathcal T$ defined by
\begin{equation}\label{multilinear-operator}
\mathcal T(u_1,\cdots,u_p)=\prod_{l\, {\rm odd}}u_l \prod_{l\, {\rm even}}\bar u_l.
\end{equation}
The Fourier modes (in space) of $\mathcal{T}$ can be represented as
\begin{equation}\label{Fourier-transform-2}
\mathcal{F}\mathcal T(u_1,\cdots,u_p)(k)=c_p\sum_{k=k_1-k_2+\cdots+k_p}\prod_{l\, {\rm odd}}\widehat u_l(k_l) \prod_{l\, {\rm even}} \overline{\widehat u_l(k_l)},
\end{equation}
where $c_p=(2\pi)^{-(p-1)/2}$.
A configuration of frequencies $(k_1,\cdots,k_p)$ with $k_l\in\Z$ and $k=k_1-k_2+\cdots +k_p$ is called resonant, if there is an odd $m\in\{1,\cdots,p\}$ such that 
$k=k_m$. Define an auxiliary $p$-linear form $\mathcal T_R$, in which all single-resonances (i.e.~$k=k_m$ for exactly one odd $m$) appear exactly once, by setting
\begin{equation}\label{resonantfrequency}
\mathcal{F}\mathcal T_R(u_1,\cdots,u_p)(k)=c_p\sum_{\substack{m=1\\
		{\rm odd}}}^{p}\sum_{\substack{ k=k_1-k_2+\cdots+k_p\\
		k=k_m}
}\prod_{l\, {\rm odd}}\widehat u_l(k_l) \prod_{l\, {\rm even}} \overline{\widehat u_l(k_l)}.
\end{equation}
Then, the difference 
\begin{equation}\label{smoothing-frequency}
\mathcal{F}\mathcal T_N(u_1,\cdots,u_p)(k):=\mathcal{F}\mathcal T(u_1,\cdots,u_p)(k)-\mathcal{F}\mathcal T_R(u_1,\cdots,u_p)(k)
\end{equation}
involves only those $(k_1,\cdots,k_p)$ in which there is no single resonance.

It turns out that the smoothing effect arises in the operator $\mathcal T_{N}$.

\begin{lemma}[Nonlinear smoothing]\label{Lemma-smoothing}
Let $T>0,s\geq 1,b>1/2$ and $\sigma\in(0,1/4]$ be arbitrarily given. Then for every $b'\in [\sigma,1/2)$, there exists a constant $C>0$ such that 
\begin{equation*}
\|\mathcal T_N(u_1,\cdots,u_p)\|_{_{X_T^{s+\sigma,-b'}}}\leq  C\prod_{l=1}^{p}\|u_l\|_{_{X_T^{s,b}}}
\end{equation*}
for any $u_1,\cdots,u_p\in X_T^{s,b}$.
\end{lemma}

\begin{proof}[{\bf Proof of Lemma \ref{Lemma-smoothing}}]
The case of $p=3$ has been addressed in \cite{ET-13}, where $u_1,u_2,u_3$ are taken to be identical. Except for 
trivial modifications, the proof given there also works for distinct $u_l$. 
In the situation of odd $p\ge 5$, \cite{McConnell-22} has proved a smoothing result with $T$ sufficiently small, $u_2=\cdots=u_p$, and a wider range of $s$, using normal form reduction and multilinear Strichartz estimate. The proof presented here is more elementary, and is decomposed into four steps. 

We only prove the corresponding estimate for $u_1,\cdots,u_p\in X^{s,b}$, without restricted to time interval $[0,T]$. The restricted version follows easily by considering extensions of $u_l$.

\medskip

\textit{Step 1: reduction to frequency inequality.} This is a standard application of duality (cf.~\cite{McConnell-22}).
For $k\in\Z$ and $\tau\in\R$ we define the $(p-1)$-dimensional hyperplanes $\Gamma^{\mathbb{Z}}_k\subset \mathbb{Z}^p$ and $\Gamma^{\mathbb{R}}_\tau\subset\R^p$ by
\begin{align*}
&\Gamma^{\mathbb{Z}}_k=\{(k_1,\dots ,k_p)\in \mathbb{Z}^p;k=k_1-k_2+\cdots+k_p\},\\
&\Gamma^{\mathbb{R}}_\tau=\{(\tau_1,\dots ,\tau_p)\in \mathbb{R}^p;\tau=\tau_1-\tau_2+\cdots+\tau_p\}.
\end{align*}
Due to \eqref{Fourier-transform-2}-\eqref{smoothing-frequency}, we write the space-time Fourier transform of $\mathcal{T}_N(u_1,\cdots ,u_p)$ as
\begin{equation*}
\begin{aligned}
&\mathcal{F} \mathcal T_N(u_1,\cdots,u_p)(\tau,k)=c_p^2 \sum_{\Gamma^{\mathbb{Z}}_k}\int_{\Gamma^{\mathbb{R}}_\tau} m(k_1,\cdots ,k_p) \prod_{l\, {\rm odd}}\widehat u_l(\tau_l,k_l) \prod_{l\, {\rm even}} \overline{\widehat u_l(\tau_l,k_l)}d\tau_1\cdots d\tau_p,
\end{aligned}
\end{equation*}
where $m(k_1,\dots ,k_p)=1-(\text{number of odd }l\text{ such that }k=k_l)\in \mathbb{Z}$. 

Since the dual space of $X^{s+\sigma,-b'}$ is $X^{-s-\sigma,b'}$ (see Lemma~\ref{Lemma-Bourgainspace}), we have
\begin{equation}\label{smoothing-estimate7}
    \|\mathcal{T}_N (u_1,\cdots,u_p)\|_{_{X^{s+\sigma,-b'}}}=\sup_{\|v\|_{_{X^{-s-\sigma,b'}}}=1} \left|\int_{\mathbb{R}} \int_{\mathbb{T}} \mathcal{T}_N (u_1,\dots ,u_p)(t,x) \overline{v(t,x)} dxdt\right|.
\end{equation}
Reformulate the right-hand side via Plancherel theorem:
\begin{equation}\label{smoothing-estimate8}
\begin{aligned}
&\int_{\mathbb{R}} \int_{\mathbb{T}} \mathcal{T}_N (u_1,\dots ,u_p)(t,x) \overline{v(t,x)} dxdt\\
&= C\sum_{k\in \mathbb{Z}} \int_{\mathbb{R}} 
\left[\sum_{\Gamma^{\mathbb{Z}}_k} \int_{\Gamma^{\mathbb{R}}_\tau} m(k_1,\cdots ,k_p)\prod_{l\, {\rm odd}}\widehat u_l(\tau_l,k_l) \prod_{l\, {\rm even}} \overline{\widehat u_l}(\tau_l,k_l)d\tau_1\cdots d\tau_p\right] \overline{\widehat{v}}(\tau,k) d\tau.
\end{aligned}
\end{equation}
Let us define auxiliary functions
\begin{align*}
    &\widehat U_l(\tau_l,k_l):=\langle k_l\rangle^{s} \langle \tau_l+k_l^2\rangle^{b} \widehat u_l(\tau_l,k_l),\\
    &\widehat{V} (\tau,k):=\langle k\rangle^{-(s+\sigma)} \langle \tau+k^2\rangle^{b'}\widehat v(\tau,k),
\end{align*}
and
\[M(\tau,k)=\sum_{\Gamma^{\mathbb{Z}}_k}\int_{\Gamma^{\mathbb{R}}_\tau} |m(k_1,\cdots ,k_p)|^2\frac{\langle k\rangle^{2(s+\sigma)}\langle \tau+k^2\rangle^{-2b'}}{\prod_{l=1}^p\langle k_l\rangle^{2s}\langle \tau_l+k_l^2\rangle^{2b}}d\tau_1\cdots d\tau_p.\]
Then $\|u_l\|_{_{X^{s,b}}}=\|U_l\|_{_{L^2_{t,x}}}$ and $\|v\|_{_{X^{-s-\sigma,b'}}}=\|V\|_{L^2_{t,x}}$. Applying Cauchy--Schwarz inequality,
\begin{align*}
&\left|\sum_{k\in \mathbb{Z}} \int_{\mathbb{R}} 
\left[\sum_{\Gamma^{\mathbb{Z}}_k} \int_{\Gamma^{\mathbb{R}}_\tau} m(k_1,\cdots ,k_p)\prod_{l\, {\rm odd}}\widehat u_l(\tau_l,k_l) \prod_{l\, {\rm even}} \overline{\widehat u_l}(\tau_l,k_l)d\tau_1\cdots d\tau_p\right] \overline{\widehat{v}}(\tau,k) d\tau\right|\\
&\le \sum_{k\in \mathbb{Z}} \int_{\mathbb{R}} M(\tau,k)^{1/2} \left(\sum_{\Gamma^{\mathbb{Z}}_k} \int_{\Gamma^{\mathbb{R}}_\tau} \prod_{l=1}^p |\widehat{U}_l(\tau_l,k_l)|^2 d\tau_1\cdots d\tau_p\right)^{1/2} |\widehat{V}(\tau,k)| d\tau\\
&\le \sup_{\tau,k} M(\tau,k)^{1/2} \left(\sum_{k\in \mathbb{Z}} \int_{\mathbb{R}} \sum_{\Gamma^{\mathbb{Z}}_k} \int_{\Gamma^{\mathbb{R}}_\tau} \prod_{l=1}^p |\widehat{U}_l(\tau_l,k_l)|^2 d\tau_1\cdots d\tau_p d\tau\right)^{1/2}\left(\sum_{k\in \mathbb{Z}} \int_{\mathbb{R}}  |\widehat{V}(\tau,k)|^2 d\tau\right)^{1/2}\\
&=\sup_{\tau,k} M(\tau,k)^{1/2} \left(\sum_{k_1,\dots ,k_p\in \mathbb{Z}} \int_{\tau_1,\dots ,\tau_p\in \mathbb{R}} \prod_{l=1}^p |\widehat{U}_l(\tau_l,k_l)|^2 d\tau_1\cdots d\tau_p \right)^{1/2}\|V\|_{L^2_{t,x}}\\
&=\sup_{\tau,k} M(\tau,k)^{1/2}\|V\|_{L^2_{t,x}}\prod_{l=1}^p \|U_l\|_{L^2_{t,x}} =\sup_{\tau,k} M(\tau,k)^{1/2} \|v\|_{X^{-s-\sigma,b'}} \prod_{l=1}^p \|u_l\|_{_{X^{s,b}}}.
\end{align*}
Plugging this into \eqref{smoothing-estimate7} and \eqref{smoothing-estimate8}, we thus obtain
\[\|\mathcal{T}_N (u_1,\cdots,u_p)\|_{_{X^{s+\sigma,-b'}}}\le C \sup_{\tau,k} M(\tau,k)^{1/2} \prod_{l=1}^p \|u_l\|_{X^{s,b}}.\]
Therefore it suffices to show $\sup_{\tau,k} M(\tau,k)<\infty$.

To this end, we invoke the following estimate:
\begin{equation}\label{smoothing-estimate2}
\int_\mathbb{R} \frac{d\tau}{\langle \tau-s\rangle^{2b} \langle \tau-t\rangle^{2b}}\le \frac{C}{\langle s-t\rangle^{2b}},\ \forall s,t\in \mathbb{R}.
\end{equation}
where the constant $C$ is determined by $b>1/2$; see, e.g.~\cite[Lemma 1]{ET-13}. Owing to this, we can first sort out $\tau_1,\cdots, \tau_p$ in $M(\tau,k)$, and then apply $\langle x\rangle \langle y\rangle\gtrsim \langle x\pm y\rangle$ to get
\begin{equation*}
\begin{aligned}
M(\tau,k)\le C\sum_{\Gamma^{\mathbb{Z}}_k} \frac{|m(k_1,\cdots,k_p)|^2\langle k\rangle^{2(s+\sigma)}\langle \tau+k^2\rangle^{-2b'}}{\langle \tau-\sum_{l=1}^p (-1)^{l}k_l^2\rangle^{2b} \prod_{l=1}^p\langle k_l\rangle^{2s}}\le C\sum_{\Gamma^{\mathbb{Z}}_k} \frac{|m(k_1,\cdots,k_p)|^2\langle k\rangle^{2(s+\sigma)}}{\langle \Phi(k,k_1,\dots ,k_p)\rangle^{2b'} \prod_{l=1}^p\langle k_l\rangle^{2s}},
\end{aligned}
\end{equation*}
where $\Phi(k,k_1,\cdots,k_p):=k^2-k_1^2+k_2^2-\cdots-k_p^2$. So, it remains to address:
\begin{equation}\label{smoothing-estimate1}
\sup_{k} \left[\sum_{k=k_1-k_2+\cdots+k_p} \frac{|m(k_1,\cdots,k_p)|^2\langle k\rangle^{2(s+\sigma)}}{\langle \Phi(k,k_1,\dots ,k_p)\rangle^{2b'} \prod_{l=1}^p\langle k_l\rangle^{2s}}\right]<\infty.
\end{equation}

\medskip
\textit{Step 2: frequency decomposition.} Let $k_l^*\ (l=1,\cdots,p)$ denote the $l$-th largest number among $|k_1|,\cdots,|k_p|$. Then if $k=k_1-k_2+\cdots+k_p$, at least one of the following properties holds:
\begin{enumerate}
\item[$(A)$] There exists exactly one odd $l$ such that $k=k_l$;
	
\item[$(B)$] $|\Phi(k,k_1,\cdots,k_p)|\ge k_1^*/p$;
	
\item[$(C)$] $(k_2^*)^2\ge k_1^*/p$
\end{enumerate}
(cf. \cite[Lemma 2]{McConnell-22}). In fact, if $(A)$ fails and resonance occurs, then there are at least two odd $l_1$ and $l_2$ with $k=k_{l_1}=k_{l_2}$. Now the algebraic relation implies $k_2^*\ge k_1^*/p$ (leading to $(C)$), since otherwise $|k|\ge k_1^*-(p-1)k_2^*> k_1^*/p$ tells us $k_2^*< k_1^*/p<|k|$, a contradictory. And if there is no resonance and $(C)$ fails, let us prove $(B)$. Assume $k_1^*=|k_j|$. Then either $j$ is even, and thus
\[|\Phi|\ge k_j^2-\sum_{l\, {\rm odd}}k_l^2
\ge (k_1^*)^2-\tfrac{p+1}{2} (k_2^*)^2\ge k_1^*-\tfrac{p-1}{p}k_1^*= k_1^*/p;\]
or $j$ is odd, and thus (note that $|k-k_j|\ge 1$ as $k\not =k_j$)
\begin{equation*}
\begin{aligned}
|\Phi|&\ge |k-k_j||k+k_j|-(p-1)(k_2^*)^2\ge |k+k_j|-(p-1)(k_2^*)^2\\
&\ge 2|k_j|-(p-1)k_2^*-(p-1)(k_2^*)^2\ge k_1^*/p.
\end{aligned}
\end{equation*}

Define $\mathcal A,\mathcal B,\mathcal C$ to be the sets of $(k_1,\cdots,k_p)\in \Gamma^{\mathbb{Z}}_k$ satisfying properties $(A),(B),(C)$, respectively. Note that $m(k_1,\cdots ,k_p)=0$ on $\mathcal A$, and $|m|\le \frac{p-1}{2}$ on $\mathcal{B}$ and $\mathcal{C}$. Therefore, it suffices prove the upper bound \eqref{smoothing-estimate1} with $m(k_1,\cdots ,k_p)$ replaced by $1$ after restricted to $\mathcal B$ and $\mathcal C$.

\medskip

\textit{Step 3: estimate on $\mathcal B$.} Since $|k|\le pk_1^*$, on $\mathcal B$ we have $|\Phi|\ge |k|/p^2$. We also need the following estimate, which is similar to \eqref{smoothing-estimate2} and can be found in \cite[Lemma 2.1]{ET-13}:
\begin{equation}\label{smoothing-estimate3}
    \sum_{k\in \mathbb{Z}} \frac{1}{\langle k-m\rangle^{2s} \langle k-n\rangle^{2s}}\le \frac{C}{\langle m-n\rangle^{2s}},\ \forall m,n\in \mathbb{Z}.
\end{equation}
Let $b'\in [\sigma,1/2)$ be fixed. Then exploiting \eqref{smoothing-estimate3}, we obtain
\begin{equation}\label{smoothing-estimate4}
\begin{aligned}
&\sup_{k}\left[\sum_{\substack{(k_1,\cdots,k_p)\in \mathcal B}}\frac{\langle k\rangle^{2(s+\sigma)}}{\langle\Phi(k,k_1,\cdots,k_p)\rangle^{2b'}\prod_{l=1}^p\langle k_l\rangle^{2s}} \right]\\
&\le C \sup_{k}\left[\sum_{\substack{k=k_1-k_2+\cdots+k_p}}\frac{\langle k\rangle^{2(s+\sigma-b')}}{\prod_{l=1}^p\langle k_l\rangle^{2s}} \right]\leq C \sup_{k}\langle k\rangle^{2(\sigma-b')}<\infty.
\end{aligned}
\end{equation}

\textit{Step 4: estimate on $\mathcal C$.} We only consider the case $|k_1|=k_1^*$ and $|k_2|=k_2^*$, since other cases can be treated in the same manner, up to the change of plus and minus signs. In the present situation we have $|k_1|\ge |k|/p$ and $|k_2|^2\ge |k_1|/p\ge |k|/p^2$. Note that for any $a\ge 0$,
\begin{equation}\label{smoothing-estimate6}
\sum_{n\in \mathbb{Z}} \frac{1}{\langle n\rangle^{2s}}<\infty\quad \text{and} \sum_{n\in \mathbb{Z},\ |n|\ge a} \frac{1}{\langle n\rangle^{2s}}\le C \langle a\rangle^{1-2s}.
\end{equation}
We use \eqref{smoothing-estimate3} to sum over $k_4,\cdots ,k_p$, and then use \eqref{smoothing-estimate6} to find
\begin{align}
&\sup_{k}\left[\sum_{\substack{(k_1,\cdots,k_p)\in \mathcal C\\
|k_1|=k_1^*,\ |k_2|=k_2^*}}\frac{\langle k\rangle^{2(s+\sigma)}}{\langle\Phi(k,k_1,\cdots,k_p)\rangle^{2b'}\prod_{l=1}^p\langle k_l\rangle^{2s}} \right]\le \sup_{k}\left[\sum_{\substack{(k_1,\cdots,k_p)\in \mathcal C\\
|k_1|=k_1^*,\ |k_2|=k_2^*}}\frac{\langle k\rangle^{2(s+\sigma)}}{\prod_{l=1}^p\langle k_l\rangle^{2s}} \right]\notag\\
&\le C \sup_{k}\left[\sum_{\substack{|k_1|\ge |k|/p,\ |k_2|\ge \sqrt{|k|}/p}}\frac{\langle k\rangle^{2(s+\sigma)}}{\langle k_1\rangle^{2s}\langle k_2\rangle^{2s}\langle k-k_1+k_2\rangle^{2s}} \right]\label{smoothing-estimate5}\\
&\le C\sup_k \left[\sum_{|k_2|\ge \sqrt{|k|}/p} \frac{\langle k\rangle^{2\sigma}}{\langle k_2\rangle^{2s}} \sum_{k_1} \frac{1}{\langle k-k_1+k_2\rangle^{2s}}\right]=C\sup_k \left[\sum_{|k_2|\ge \sqrt{|k|}/p} \frac{\langle k\rangle^{2\sigma}}{\langle k_2\rangle^{2s}} \sum_{n\in \mathbb{Z}} \frac{1}{\langle n\rangle^{2s}}\right]\notag\\
&\le C\sup_k \sum_{|k_2|\ge \sqrt{|k|}/p} \frac{\langle k\rangle^{2\sigma}}{\langle k_2\rangle^{2s}}\le C\sup_k \langle k\rangle^{2\sigma}\langle \sqrt{|k|}/p\rangle^{1-2s}\le C\sup_k \langle k\rangle^{2\sigma+\frac{1-2s}{2}}<\infty,\notag
\end{align}
provided $\sigma\le 1/4$ so that $2\sigma+\frac{1-2s}{2}\le 0$.

\medskip

Finally, combining (\ref{smoothing-estimate4}) and \eqref{smoothing-estimate5}, we conclude \eqref{smoothing-estimate1}. Now the proof is complete.
\end{proof}

\subsection{Completing the proof of Theorem~\ref{Theorem-AC}}\label{Section-AC-2}

We return to equation (\ref{Nonlinear-problem}). In the special case of $u_1=\cdots=u_p=u$, one can check that $\mathcal{T}$ and $\mathcal{T}_R$ defined in \eqref{multilinear-operator} and \eqref{resonantfrequency} reduce to
\begin{equation*}
\mathcal T(u):=\mathcal T(u,\cdots ,u)=|u|^{p-1}u\quad \text{and}\quad \mathcal T_R(u):=\mathcal T_R(u,\cdots,u)=\frac{p+1}{4\pi}\|u\|_{_{L_x^{p-1}}}^{p-1}u;
\end{equation*}
cf.~\cite{ET-13,McConnell-22}. Also, recall that the $C_0$-group $S_a(t)$ possesses a uniform exponential stability. That is, there exists a constant $\beta>0$ such that for every $s\geq 0$,
one can find $C>0$ satisfying
\begin{equation}\label{Decay-semigroup}
\|S_a(t)u_0\|_{_{H^s}}\leq Ce^{-\beta t}\|u_0\|_{_{H^s}},\ \forall t\geq 0;
\end{equation}
see \cite[Proposition 4.1]{RZ-09}.

In order to drop the terms in $\mathcal{T}_R$, we make use of the norm-preserving transformation $u\mapsto U$ given by (\ref{w-variable}) and derive the equation for the new variable $U$ (recall $\mathcal{T}_N:=\mathcal{T}-\mathcal{T}_R$):
\begin{equation}\label{Problem-w}
\left\{\begin{array}{ll}
iU_t+ U_{xx} +ia(x)U =
\mathcal{T}_N(U)+F(t,x),\\
U(0,x)=U_0(x)(=u_0(x)),
\end{array}\right.
\end{equation}
where $F(t,x)=e^{i\theta(t)}f(t,x)$. In view of the Duhamel formula,
\begin{equation}\label{Duhamel-formula-3}
U(t)=S_a(t)U_0-i\int_0^tS_a(t-\tau)\left(\mathcal{T}_N(U)+F
\right)d\tau.
\end{equation}

The desired compact attracting set is constructed via the following proposition.

\begin{proposition}\label{Proposition-AC-w}
Let $R>0$ and $\sigma\in (0,1/4]$ be arbitrarily given. Then there exists a constant $C_3>0$ such that 
$$
\|U(t)-S_a(t)U_0\|_{_{H^{1+\sigma}}}\leq C_3,\ \forall t\geq 0
$$
for any $U_0\in \overline{B}_{H^1}(R)$ and $F\in \overline{B}_{L^2_b(\R^+;H^{1+\sigma})}(R)$, where $U(t)$ stands for the solution of {\rm(\ref{Problem-w})}.
\end{proposition}

\begin{proof}[{\bf Proof of Proposition \ref{Proposition-AC-w}}]
In view of the arguments of Proposition~\ref{Theorem-absorbing} (especially Remark~\ref{boundednessH^1}),
\begin{equation*}
\|U(t)\|_{_{H^1}}=\|u(t)\|_{_{H^1}}\leq C,\ \forall t\geq 0
\end{equation*}
for any $U_0\in \overline{B}_{H^1}(R)$ and $F\in \overline{B}_{L^2_b(\R^+;H^{1+\sigma})}(R)$. Let us arbitrarily fix $b\in (1/2,1-\sigma)$ and $T>0$. Then similar to Proposition~\ref{Proposition-nonlinear}, one can deduce that 
\begin{equation}\label{X-boundedness}
\|U\|_{_{X^{1,b}_T}}\leq C(\|U_0\|_{_{H^1}}+\|F\|_{_{L^2(0,T;H^1(\mathbb{T}))}})\le C.
\end{equation}

Take $b'=1-b$, which suits into Lemma \ref{Lemma-smoothing}; and recall from Lemma~\ref{Lemma-Bourgainspace} the embedding $X_T^{1+\sigma,b}\hookrightarrow L_t^\infty H_x^{1+\sigma}$. Thanks to (\ref{Decay-semigroup})-(\ref{X-boundedness}) and Proposition \ref{Proposition-S(t)estimate}(2), for $t\in [0,T]$ we have
\begin{equation}\label{estimate-2}
\begin{aligned}
\|U(t)-S_a(t)U_0\|_{_{H^{1+\sigma}}}
&\leq C\|\int_0^tS_a(t-\tau)\mathcal{T}_N(U) d\tau\|_{_{X^{1+\sigma,b}_T}}+\|\int_0^tS_a(t-\tau)Fd\tau\|_{_{H^{1+\sigma}}}\\
&\leq C\|\mathcal{T}_N(U)\|_{_{X^{1+\sigma,-b'}_T}}+C\|F\|_{_{L^2(0,T;H^{1+\sigma})}}\\
&\leq C\|U\|_{_{X^{1,b}_T}}^p+C\|F\|_{_{L^2_b(\R^+;H^{1+\sigma})}}\le C.
\end{aligned}
\end{equation}
This means that the desired conclusion is obtained for $t\in[0,T]$.

To proceed further, one can repeat the argument for (\ref{estimate-2}), with $[0,T]$ replaced by $[(j-1)T,jT]\ (j\in\N^+)$, which yields that
\begin{equation}\label{estimate-6}
\|U(t)-S_a(t-(j-1)T)U((j-1)T)\|_{_{H^{1+\sigma}}}\le C,\ \forall t\in[(j-1)T,jT].
\end{equation}
We emphasize here that the constant $C$ in (\ref{estimate-6}) does not depend on $j$, as the boundedness of $f$ in $L^2((j-1)T,T;H^{1+\sigma})$ and $U$ in $X^{1,b}_{[(j-1)T,jT]}$ are uniform with respect to $j$. Accordingly, making use of (\ref{Decay-semigroup}) again we compute for every $J\in\N^+$ that
\begin{align}
&\|U(JT)-S_a(JT)U_0\|_{_{H^{1+\sigma}}}=\|\sum_{j=1}^J\left[
S_a((J-j)T)U(jT)-S_a((J-j+1)T)U((j-1)T)
\right]\|_{_{H^{1+\sigma}}}\notag\\
&\le C\sum_{j=1}^J e^{-\beta(J-j)T} \|U(jT)-S_a(T)U((j-1)T)\|_{_{H^{1+\sigma}}}\le C\sum_{j=0}^\infty e^{-\beta jT}<\infty.\label{estimate-7}
\end{align}
Now, for $t>T$, we write $t=JT+t'$ with $t'\in[0,T)$, and deduce that
\begin{align*}
&\|U(t)-S_a(t)U_0\|_{_{H^{1+\sigma}}}\\
&\leq \|U(t)-S_a(t-JT)U(JT)\|_{_{H^{1+\sigma}}} +\|S_a(t-JT)(U(JT)-S_a(JT)U_0)\|_{_{H^{1+\sigma}}}.
\end{align*}
This implies the conclusion of this proposition, as the first term can be bounded by (\ref{estimate-6}) and the second by (\ref{estimate-7}). The proof is then complete.
\end{proof}

With Propositions~\ref{Theorem-absorbing} and \ref{Proposition-AC-w} in hand, we are in a position to demonstrate Theorem~\ref{Theorem-AC}.

\begin{proof}[{\bf Proof of Theorem \ref{Theorem-AC}}]
We first take inclusion (\ref{inclusion-absorbing}) into account. Thus, any global solution $u(t)$ of (\ref{Nonlinear-problem}), with $u_0\in H^1$ and $f\in \overline{B}_{L^\infty(\R^+;H^{1+\sigma})}(R_0)$, enters into the absorbing set $\mathscr{B}_1$ at time $$T_0=C_1\log (1+E_u(0)).$$ Obviously, the translated function $u(t+T_0)$ remains a solution of (\ref{Nonlinear-problem}) with $(u_0,f)$ replaced by $(u(T_0),f(\cdot+T_0))$. Then it follows from Proposition \ref{Proposition-AC-w}  that 
$$
\|u(t+T_0)-e^{-i\theta_{T_0}(t)}S_a(t)u(T_0)\|_{_{H^{1+\sigma}}}\le C_3,\ \forall t\geq 0
$$
with $\theta_{T_0}(t)=\frac{p+1}{4\pi}\int_0^t\|u(s+T_0)\|^{p-1}_{_{L^{p-1}(\T)}}ds$.
We construct a bounded subset $\mathscr{B}_{1,1+\sigma}$ of $H^{1+\sigma}$ by
\[\mathscr{B}_{1,1+\sigma}:=\{v\in H^{1+\sigma};\|v\|_{_{H^{1+\sigma}}}\le C_3\}.\]
Owing to (\ref{Decay-semigroup}) and $u(T_0)\in \mathscr{B}_1$, we obtain
$$
{\rm dist}_{H^1}(u(t),\mathscr B_{1,1+\sigma})\le \|S_a(t-T_0)u(T_0)\|_{_{H^1}}\le Ce^{-\beta (t-T_0)}\|u(T_0)\|_{_{H^1}}\le Ce^{-\beta (t-T_0)}
$$
for any $t\geq T_0$. Moreover, if we choose $\kappa=\beta$ and shrink $\kappa$ so that $C_1\kappa \le 1$, then
\[{\rm dist}_{H^1}(u(t),\mathscr B_{1,1+\sigma})\le C e^{-\kappa t}(1+E_u(0))^{C_1\kappa} \le Ce^{-\kappa t}(1+E_u(0)),\ \forall t\ge t_0.\]

To continue, we invoke the apriori estimates (\ref{E_uapriori}) (with $T=1$) and (\ref{uniform-bound}), in order to infer  
\begin{equation}\label{H1attractor-1}
E_u (t)\leq C(1+E_u(0)),\ \forall t\geq 0,
\end{equation}
where the constant $C=C(R)>0$.
Then, thanks to $\|u(t)\|_{_{H^1}}\le E_u^{1/2}(t)$, for $t\in [0,T_0]$ we have
\begin{equation}\label{H1attractor-2}
{\rm dist}_{H^1}(u(t),\mathscr B_{1,1+\sigma})
\leq (C+\|u(t)\|_{_{H^1}})e^{\kappa T_0} e^{-\kappa t}\leq C(1+E_u(0))^{1/2+C_1\kappa }e^{-\kappa t}.
\end{equation}
Finally, we shrink $\kappa$ once again so that $C_1\kappa \leq 1/2$. The proof is then complete.
\end{proof}

\section{Global dynamics in higher Sobolev space}\label{Section-stability}

This quick section is devoted to extending exponential asymptotic compactness and global stabilization
in Section~\ref{Section-AC} to higher Sobolev norm $H^s\,(s\ge 1)$. Recall that in Theorem~\ref{Theorem-AC}, we assumed the initial data $u_0\in H^1$, and found a bounded attracting set in $H^{1+\sigma}$ with $\sigma\in (0,1/4]$. In Theorem~\ref{Theorem-AC-Hs} below, we establish a similar result for $u_0\in H^s$ and arbitrary $\sigma>0$.

To this end, we first derive the existence of $H^s$-absorbing set, generalizing Proposition~\ref{Theorem-absorbing}. 

\begin{proposition}[$H^s$-absorbing set]\label{Prop-Hsabsorbing}
Let $s\ge 1$ and $R_2>0$ be arbitrarily given. Then there exists a bounded subset $\mathscr{B}_s$ of $H^s$ and a constant $C>0$ such that
\[u(t)\in \mathscr{B}_s,\ \forall t\ge C(1+E_u(0))^{(p-1)\lceil 4(s-1)\rceil /2}\log (1+\|u_0\|_{_{H^s}})\]
for any $u_0\in H^s$ and $f\in \overline{B}_{_{L^2_b(\mathbb{R}^+;H^s)}}(R_2)$, where $u(t)$ stands for the solution of \eqref{Nonlinear-problem}.
\end{proposition}

The obstacle here is the absence of $H^s$-conservation law for the (undamped and unforced) NLS, except for $E_u(t)$ when $s=1$. Hence it is unclear whether an $H^s$-dissipation similar to Lemma~\ref{Lemma-discretemono} is valid. Instead, we perform an induction argument and promote the global absorbing set from $H^1$ to $H^{s}$ by exploiting nonlinear smoothing (Lemma \ref{Lemma-smoothing}). As in the proof of Theorem~\ref{Theorem-AC}, we divide the solution into three parts according to Duhamel formula: linear evolution (with phase shift), nonlinearity and source term. The linear evolution decays exponentially and the source term is bounded. In addition, thanks to the nonlinear smoothing, the nonlinear term has extra regularity of order $1/4$. Thus its $H^{s+1/4}$-norm can be bounded by the $H^{s}$-norm of $u$, which eventually depends on the $H^{s}$-absorbing set. This observation allows us to perform an induction from $H^{s}$ to $H^{s+1/4}$.

\begin{proof}[{\bf Proof of Proposition~\ref{Prop-Hsabsorbing}}]
When $s=1$, this follows from Proposition~\ref{Theorem-absorbing}. It remains to treat the inductive step: if a global $H^s$-absorbing set $\mathscr{B}_s$ for (\ref{Nonlinear-problem}) exists with $4s\in \mathbb{N}$, then for any $\sigma\in (0,1/4]$, we can find an $H^{s+\sigma}$-absorbing set $\mathscr{B}_{s+\sigma}$. 
Similarly to the proof of Proposition~\ref{Proposition-AC-w}, we exploit the norm-preserving transformation $u\mapsto U$ as in \eqref{w-variable}. Then $U$ satisfies the equation \eqref{Problem-w}. Thanks to \eqref{Decay-semigroup}, we can choose $T_1>0$ (depending on $s+\sigma$) such that 
\begin{equation*}
\|S_a(T_1)\|_{_{\mathcal L(H^{s+\sigma})}}\le 1/2.
\end{equation*}
By inductive hypothesis, there exists a constant 
\begin{equation}\label{absorbingHs-1}
T_2=C(1+E_u(0))^{2(p-1)(s-1)}\log (1+\|u_0\|_{_{H^s}}),
\end{equation}
such that $\|U(t)\|_{_{H^s}}\le C$ for any $t\ge T_2$, where the constant $C>0$ depends on $R_2$. Thus, due to Proposition~\ref{Proposition-nonlinear}, we can find a constant $C$ depending on $s,s+\sigma$ and $R_2$, such that
\[\|U\|_{_{X^{s,b}_{[t,T_1+t]}}}\le C,\ \forall t\ge T_2.\]
Then, using the same argument as in \eqref{estimate-2}, we have
\[\|U(T_1+t)\|_{_{H^{s+\sigma}}}\le \frac{1}{2}\|U(t)\|_{_{H^{s+\sigma}}}+C\|U\|_{_{X^{s,b}_{[t,T_1+t]}}}^p+CR_2\le \frac{1}{2} \|U(t)\|_{_{H^{s+\sigma}}}+C\]
for any $t\ge T_2$. This dissipative estimate implies the existence of bounded absorbing set $\mathscr{B}_{s+\sigma}$ in $H^{s+\sigma}$ (cf.~the proof of Proposition~\ref{Theorem-absorbing}), and $u(t)$ enters $\mathscr{B}_{s+\sigma}$ for $t\ge T_3$, where
\begin{equation}\label{absorbingHs-2}
T_3=T_2+C\log (1+\|U(T_2)\|_{_{H^{s+\sigma}}}).
\end{equation}

Finally, in order to estimate $\|U(T_2)\|_{_{H^{s+\sigma}}}=\|u(T_2)\|_{_{H^{s+\sigma}}}$, we need the following lemma.

\begin{lemma}\label{Lemma-KatoPonce}
Let $s\ge 0$ be arbitrarily given. Then there exists a constant $C>0$, such that
\begin{equation*}
\|fg\|_{_{H^s}}\le C(\|f\|_{_{H^1}}\|g\|_{_{H^s}}+\|f\|_{_{H^s}}\|g\|_{_{H^1}}),\ \forall f,g\in H^s(\mathbb{T})\cap H^1(\mathbb{T}).
\end{equation*}
\end{lemma}

\begin{proof}[{\bf Proof of Lemma~\ref{Lemma-KatoPonce}}]
It turns out that Lemma~\ref{Lemma-KatoPonce} is a consequence of Kato--Ponce inequality \cite{KP-88}, which states a stronger inequality (note that $H^1\hookrightarrow L^\infty$ since we are working on $\mathbb{T}$):
\[\|fg\|_{_{H^s}}\le C(\|f\|_{_{L^\infty}}\|g\|_{_{H^s}}+\|f\|_{_{H^s}}\|g\|_{_{L^\infty}}).\]
We provide a short proof for the reader's convenience. First note that 
\[\langle k\rangle^s\le C(\langle l\rangle^s +\langle k-l \rangle^s)\]
for any $k,l\in \mathbb{Z}$.
Then thanks to the Young inequality for convolutions, it follows that
\begin{equation}\label{KP-1}
\begin{aligned}
\|fg\|_{_{H^s}}&=\|\langle k \rangle (\widehat{f}*\widehat{g})(k)\|_{_{l^2}}\le C\left(\|\sum_{l\in \mathbb{Z}} \langle l\rangle^s \widehat{f}(l)\widehat{g}(k-l)\|_{_{l^2}}+\|\sum_{l\in \mathbb{Z}} \widehat{f}(l)\langle k-l\rangle^s \widehat{g}(k-l)\|_{_{l^2}}\right)\\
&\le C\left(\|\langle l\rangle^s \widehat{f}(l)\|_{_{l^2}}\|\widehat{g}\|_{_{l^1}}+\|\widehat{f}(l)\|_{_{l^1}}\|\langle l\rangle^s \widehat{g}(l)\|_{_{l^2}}\right)=C\left(\|f\|_{_{H^s}}\|\widehat{g}\|_{_{l^1}}+\|\widehat{f}\|_{_{l^1}}\|g\|_{_{H^s}}\right).
\end{aligned}
\end{equation}
Then the lemma follows since for any $h\in H^1$,
\[\|\widehat{h}\|_{_{l^1}}\le \|\langle k\rangle^{-1}\|_{_{l^2}}\|\langle k\rangle \widehat{h}(k)\|_{_{l^2}}=C\|h\|_{_{H^1}}.\qedhere\]
\end{proof}

Now let us come back to the proof of Proposition~\ref{Prop-Hsabsorbing}. We iterate Lemma~\ref{Lemma-KatoPonce} to find
\begin{equation}\label{Kato-Ponce}
\||u|^{p-1}u\|_{_{H^{s+\sigma}}}\le C\|u\|_{_{H^1}}^{p-1}\|u\|_{_{H^{s+\sigma}}}.
\end{equation}
Moreover, combining \eqref{E_uapriori} (with $T=1$) and \eqref{H1attractor-1}, it follows that
$$
\|u(t)\|_{_{H^1}}^2\leq 2E_u(t)\le C(1+E_u(0)),\ \forall t\geq 0,
$$
where the constant $C>0$ does not depend on $t$.
This together with (\ref{Kato-Ponce}) implies that
\begin{align*}
\|u(t)\|_{_{H^{s+\sigma}}}&\le \|S_a(t)u_0\|_{_{H^{s+\sigma}}}+\|\int_0^t S_a(t-\tau)(|u|^{p-1}u+f)d\tau\|_{_{H^{s+\sigma}}}\\
&\le Ce^{-\beta t}\|u_0\|_{_{H^{s+\sigma}}}+C\int_0^t e^{-\beta (t-\tau)} \left[(1+E_u(0))^{(p-1)/2} \|u(\tau)\|_{_{H^{s+\sigma}}}+\|f(\tau)\|_{_{H^{s+\sigma}}}\right] d\tau.
\end{align*}
Note that for any $t\ge 0$,
\[\int_0^t e^{-\beta (t-\tau)}\|f(\tau)\|_{_{H^{s+\sigma}}} d\tau\le \sum_{n=0}^{\lfloor t \rfloor} \int_n^{\min\{n+1,t\}} e^{-\beta \tau} \|f(t-\tau)\|_{_{H^s}} d\tau\le\sum_{n=0}^\infty e^{-\beta n} R_2\le CR_2.\]
Using the Gronwall inequality for $t\mapsto e^{\beta t} \|u(t)\|_{_{H^{s+\sigma}}}$, the H\"older inequality and \eqref{absorbingHs-1}, we have
\begin{equation}\label{Hsatrractor-1}
\begin{aligned}
&\|u(T_2)\|_{_{H^{s+\sigma}}}\le C\left[\|u_0\|_{_{H^{s+\sigma}}}+\int_0^{T_2}e^{-\beta(T_2-\tau)}\|f(\tau)\|_{_{H^{s+\sigma}}}d\tau\right]e^{C(1+E_u(0))^{(p-1)/2} T_2}\\
&\le C(1+\|u_0\|_{_{H^{s+\sigma}}})e^{C(1+E_u(0))^{(p-1)/2} T_2}\le C(1+\|u_0\|_{_{H^{s+\sigma}}})^{C(1+E_u(0))^{(p-1)(4s-3)/2}}.
\end{aligned}
\end{equation}

Substituting this into \eqref{absorbingHs-2}, and since $4s-3=\lceil 4(s+\sigma-1)\rceil$, we obtain
\[T_3\le T_2+C(1+E_u(0))^{(p-1)\lceil 4(s+\sigma-1)\rceil/2}\log (1+\|u_0\|_{_{H^{s+\sigma}}}).\]
Now that $T_2$ is also bounded by the last term, the proof is complete.
\end{proof}

\begin{remark}\label{boundednessH^s}
Analogously to Remark~{\rm\ref{boundednessH^1}}, we can derive that given any $R>0$, there exists a constant $C=C(R)>0$ such that 
\begin{equation}\label{H^s-boundedness}
\|u(t)\|_{_{H^s}}\le C,\ \forall t\ge 0
\end{equation}
for any $u_0\in \overline{B}_{_{H^s}}(R)$ and $f\in \overline{B}_{_{L^2_b(\mathbb{R}^+;H^s)}}(R)$.
\end{remark}

Now we state the main result of this section, which is an $H^s$-extension of Theorem~\ref{Theorem-AC}.

\begin{theorem}\label{Theorem-AC-Hs}
Let $s\ge 1$ and $R_2,\sigma>0$ be arbitrarily given. Then there exists a bounded subset $\mathscr{B}_{s,s+\sigma}$ of $H^{s+\sigma}$ and a constant $\kappa_{s,s+\sigma}>0$, such that
\begin{equation}\label{AC-Hs}
\dist_{_{H^s}}(u(t),\mathscr{B}_{s,s+\sigma})\le C(1+\|u_0\|_{_{H^s}})^{C(1+E(u_0))^{(p-1)\lceil 4s-3\rceil/2}}e^{-\kappa_{s,s+\sigma} t},\ \forall t\geq 0
\end{equation}
for any $u_0\in H^s$ and $f\in \overline{B}_{_{L^2_b(\mathbb{R}^+;H^{s+\sigma})}}(R_2)$, where $u(t)$ stands for the solution of \eqref{Nonlinear-problem}. 

Moreover, if $s=1$, then the estimate \eqref{AC-Hs} can be improved as {\rm(}cf.~Theorem~{\rm \ref{Theorem-AC})}
\[\dist_{H^1}(u(t),\mathscr{B}_{1,1+\sigma})\le C(1+E(u_0))e^{-\kappa_{1,1+\sigma} t},\ \forall t\geq 0.\]
\end{theorem}

\begin{proof}[{\bf Proof of Theorem~\ref{Theorem-AC-Hs}}]
When $\sigma\in (0,1/4]$, the reasoning in the proof of Theorem~\ref{Theorem-AC} ($s=1$) can be easily adapted to 
the case of $s>1$. This is mainly because the result of nonlinear smoothing (Lemma~\ref{Lemma-smoothing}) is valid for the general scale $s$, and the $H^s$-absorbing set has been established in Proposition~\ref{Prop-Hsabsorbing}. We omit the details of this adaptation.

In the sequel, we induct on $\sigma$, employing the idea of ``transitivity of attraction" (cf. the paper \cite{Zelik-04} and also \cite{LWX-24}). More precisely, assuming that (\ref{Nonlinear-problem}) admits an $H^{s+\sigma}$-attracting set $\mathscr{B}_{s,s+\sigma}$, we proceed to construct an $H^{s+\sigma'}$-attracting set with an arbitrarily given $\sigma'\in (\sigma,\sigma+1/4]$.

Let $\mathscr{B}_s$ be the $H^s$-absorbing set in Proposition~\ref{Prop-Hsabsorbing}. By inductive hypothesis, if $u_0\in \mathscr{B}_s$, then
\begin{equation}\label{absorbingestimate2}
\dist_{H^s}(u(t),\mathscr{B}_{s,s+\sigma})\le Ce^{-\kappa_{s,s+\sigma}t},\ \forall t\geq 0,
\end{equation}
where the constant $C$ is uniform due to the boundedness of $\mathscr{B}_s$ in $H^s$. Next, we invoke again the inductive hypothesis with $(s,s+\sigma)$ replaced by $(s+\sigma,s+\sigma')$. Thus there exists a bounded subset $\mathscr{B}_{s+\sigma,s+\sigma'}$ of $H^{s+\sigma'}$ such that if $u_0\in \mathscr{B}_{s,s+\sigma}$, then
\begin{equation*}
\dist_{H^{s+\sigma}}(u(t),\mathscr{B}_{s+\sigma,s+\sigma'})\le Ce^{-\kappa_{s+\sigma,s+\sigma'}t},\ \forall t\geq 0.
\end{equation*}

In what follows we assume $u_0\in\mathscr{B}_s$ and set $t=t_1+t_2$, where $t_1,t_2$ will be determined below. Let us first apply (\ref{absorbingestimate2}) to deduce that there exists $\phi\in \mathscr{B}_{s,s+\sigma}$ such that
\begin{equation}\label{absorbingestimate4}
\|u(t_1)-\phi\|_{_{H^s}}\leq Ce^{-\kappa_{s,s+\sigma}t_1}.
\end{equation}
Meanwhile, the attraction (\ref{absorbingestimate2}) implies that there exists $\psi\in \mathscr{B}_{s+\sigma,s+\sigma'}$ such that
\begin{equation}\label{absorbingestimate5}
\|u^\phi(t)-\psi\|_{_{H^{s+\sigma}}}\leq Ce^{-\kappa_{s+\sigma,s+\sigma'}t_2},
\end{equation}
where $u^\phi(t)$ stands for the solution of (\ref{Nonlinear-problem}) with the initial condition replaced by $u^\phi(t_1)=\phi$. Furthermore, thanks to Remark~\ref{boundednessH^s}
and Gronwall inequality, it is easy to see that
\begin{equation}\label{absorbingestimate6}
\|u(t)-u^\phi(t)\|_{_{H^s}}\le Ce^{Ct_2}\|u(t_1)-\phi\|_{_{H^s}}\le Ce^{Ct_2-\kappa_{s,s+\sigma}t_1},
\end{equation}
where we have tacitly used (\ref{absorbingestimate4}), and that $u(t_1)$ and $\phi$ belongs to a bounded subset of $H^s$. Combining (\ref{absorbingestimate5}) and (\ref{absorbingestimate6}), one concludes that 
\[\|u(t)-\psi\|_{_{H^s}}\leq \|u(t)-u^\phi(t)\|_{_{H^s}}+\|u^\phi(t)-\psi\|_{_{H^{s}}}\leq Ce^{Ct_2-\kappa_{s,s+\sigma}t_1}+Ce^{-\kappa_{s+\sigma,s+\sigma'}t_2}.\]
Accordingly, taking $t_1=(1-\varepsilon)t$ and $t_2=\varepsilon t$ with $\varepsilon>0$ sufficiently small, it follows that there exists a constant $\kappa_{s,s+\sigma'}>0$ such that provided $u_0\in\mathscr{B}_s$, we have
\begin{equation}\label{absorbingestimate7}
\dist_{_{H^{s}}}(u(t),\mathscr{B}_{s+\sigma,s+\sigma'})\leq Ce^{C\varepsilon t-(1-\varepsilon)\kappa_{s,s+\sigma}t}+Ce^{-\varepsilon\kappa_{s+\sigma,s+\sigma'} t}\leq Ce^{-\kappa_{s,s+\sigma'}t}.
\end{equation}

The remainder of the proof is easy, since $\mathscr{B}_s$ is an absorbing set for any $u_0\in H^s$, and the elapsed time $T_3$ in which $u(t)$ enters $\mathscr{B}_s$ has the form specified in Proposition~\ref{Prop-Hsabsorbing}, i.e.
\[T_3=C(1+E_u(0))^{(p-1)\lceil 4(s-1)\rceil /2}\log (1+\|u_0\|_{_{H^s}}).\]
For $t\in [0,T_3]$, we can reproduce \eqref{Hsatrractor-1} (with $(s+\sigma,T_1)$ replaced by $(s,T_3)$) to find
\begin{equation}\label{Hsattractor-2}
\|u(t)\|_{_{H^s}}\le C(1+\|u_0\|_{_{H^s}})e^{C(1+E_u(0))^{(p-1)/2}T_3}\le C(1+\|u_0\|_{_{H^s}})^{C(1+E_u(0))^{(p-1)\lceil 4s-3\rceil/2}}.
\end{equation}
As a consequence, 
\begin{equation}\label{Hsattractor-3}
\begin{aligned}
\dist_{H^s} (u(t),\mathscr{B}_{s+\sigma,s+\sigma'})&\le (C+\|u(t)\|_{_{H^s}}) e^{\kappa_{s,s+\sigma'} T_3} e^{-\kappa_{s,s+\sigma'} t}\\
&\le C(1+\|u_0\|_{_{H^s}})^{C(1+E_u(0))^{(p-1)\lceil 4s-3\rceil/2}}e^{-\kappa_{s,s+\sigma'}t}
\end{aligned}
\end{equation}
for any $t\in [0,T_3]$. This together with (\ref{absorbingestimate7}) implies the desired inequality (\ref{AC-Hs}).

Moreover, if $s=1$, then the last two estimates \eqref{Hsattractor-2} and \eqref{Hsattractor-3} can be improved by \eqref{H1attractor-1} and \eqref{H1attractor-2}, respectively. Now the proof is complete.
\end{proof}

Before ending this section, we also mention the global $H^s$-stabilization, as an easy corollary of $H^1$-stabilization (Proposition~\ref{Corollary-stability}) and Kato--Ponce inequality (Lemma~\ref{Lemma-KatoPonce}).

\begin{proposition}[Global stabilization in $H^s$]\label{Theorem-H^sstability}
There exists a constant $\beta>0$ such that for every $s\ge 1$, there is a constant $C>0$ satisfying that
\begin{equation*}
\|u(t)\|_{_{H^s}}\leq C\|u_0\|_{_{H^s}}e^{CE(u_0)^{(p-1)/2}}e^{-\beta t},\ \forall t\geq 0
\end{equation*}
for any $u_0\in H^s$, where $u(t)$ stands for the solution of \eqref{Nonlinear-problem} with $f(t,x)\equiv 0$.
\end{proposition}

\begin{proof}[{\bf Proof of Proposition~\ref{Theorem-H^sstability}}]
We can find a constant $\beta>0$ such that \eqref{H^1stability} and \eqref{Decay-semigroup} hold simultaneously. Then, thanks to \eqref{Kato-Ponce} and 
$$
\|u(t)\|_{_{H^1}}^2\le CE_u(t)\le CE_u(0)e^{-\beta t},\ \forall t\geq 0,
$$ 
we have
\begin{align*}
\|u(t)\|_{_{H^s}}&\le Ce^{-\beta t} \|u(0)\|_{_{H^s}}+C\int_0^t e^{-\beta (t-\tau)}\cdot e^{-\beta \tau\cdot \frac{p-1}{2}} E_u^{(p-1)/2}(0) \|u(\tau)\|_{_{H^s}} d\tau\\
&\le Ce^{-\beta t} \|u(0)\|_{_{H^s}}+CE^{(p-1)/2}_u(0) e^{-\beta t}\int_0^t \|u(\tau)\|_{_{H^s}} d\tau.
\end{align*}
Using the Gronwall inequality for $t\mapsto e^{\beta t}\|u(t)\|_{_{H^s}}$, the conclusion follows immediately.
\end{proof}

\section{Control property for coupling conditions}\label{Section-control}

In this section, we shall investigate a controlled system associated with (\ref{Random-problem}), reading
\begin{equation}\label{Control-Problem-0}
\left\{\begin{array}{ll}
iu_t+u_{xx}+ia(x)u=|u|^{p-1}u+h(t,x)+\chi(x)\mathcal P_N\xi(t,x),\\
u(0,x)=u_0(x),
\end{array}\right.
\end{equation}
at the scale of $H^s$ with $s\geq 1$.
Here, $h\in L^2(0,T;H^{s+\sigma})$ denotes a given external force\footnote{The external force $h$ represents a realization of random noise $\eta$ in \eqref{Random-problem}, and our setting on the noise structure ensures its $L_t^2 H_x^{s+\sigma}$-regularity. As a result, the uncontrolled trajectory $\tilde{u}$ solving \eqref{Uncontrolled-Problem} belongs to $X_T^{s+\sigma,b}$ provided $\tilde{u}_0\in H^{s+\sigma}$. At the same time, our control result (Theorem~\ref{Theorem-control}) and its proof essentially relies on $\|\tilde{u}\|_{_{X_T^{s+\sigma,b}}}$, instead of $\|\tilde{u}_0\|_{_{H^{s+\sigma}}}$ and $\|h\|_{_{L^2(0,T;H^{s+\sigma})}}$. For this reason, the assumption on $h$ is hidden from its statement.}, $\xi$ stands for a control to be chosen, and $\mathcal P_N\ (N\in\N^+)$ is the orthogonal projection from $L^2(0,T;H^s(\mathbb{T}))$ onto 
$$
\{\alpha_j^{\scriptscriptstyle T}(t)e_k(x);j,|k|\leq N\}
$$
with $N$ to be determined. When $\xi(t,x)\equiv 0$, the system is said to be uncontrolled, reading
\begin{equation}\label{Uncontrolled-Problem}
\left\{\begin{array}{ll}
i\tilde u_t+\tilde u_{xx}+ia(x)\tilde u=|\tilde u|^{p-1}\tilde u+h(t,x),\\
\tilde u(0,x)=\tilde u_0(x).
\end{array}\right.
\end{equation}

As described in Section \ref{Section-strategy}, the coupling condition is closely related to stabilization along uncontrolled solutions; see Section \ref{Section-verification} for details. To this end, we first define the terminology of ``stabilization along trajectory". A subtlety is that the stabilization may take place with respect to an equivalent norm $\|\cdot\|_{_{\tilde{H}^s}}$, rather than the original Sobolev norm $\|\cdot\|_{_{H^s}}$.

\begin{definition}[Stabilization along trajectory]\label{Definition-control}
Let $T>0$, $s\geq 1$, $N\in \mathbb{N}^+$, $d>0$ and $q'\in (0,1)$ be arbitrarily given.
The controlled system \eqref{Control-Problem-0} is said to be $(d,q')$-stabilized to $\tilde u\in C(0,T;H^s)$ with respect to $\|\cdot\|_{_{\tilde{H}^s}}$ {\rm(}equivalent to the standard $H^s$-norm{\rm)}, if for every $u_0\in H^s$, when $
\|u_0-\tilde u_0\|_{_{H^s}}\leq d$ there exists a control $\xi\in L^2(0,T;H^{s})$ such that
\begin{equation}\label{Squeezing-0}
\|u(T)-\tilde u(T)\|_{_{\tilde{H}^s}}\leq q' \|u_0-\tilde u_0\|_{_{\tilde{H}^s}},
\end{equation}
where $u\in X^{s,b}_T$ stands for the solution of {\rm(\ref{Control-Problem-0})}.
\end{definition}

The main result of this section is stated as follows, which means the controlled system {\rm(\ref{Control-Problem-0}) can be stabilized to uncontrolled solutions with a bit higher regularity in space.

\begin{theorem}\label{Theorem-control}
Let $T, R,\sigma>0$ and $s\ge 1$ be arbitrarily given. Then there exist constants $d>0$, $N\in\N^+$, $q'\in (0,1)$, an equivalent norm $\|\cdot\|_{_{\tilde{H}^s}}$ on $H^s$, and a map 
$
\Phi\colon \overline{B}_{X_T^{s+\sigma,b}}(R)\rightarrow \mathcal{L}_{\R}(H^s;L^2(0,T;H^{s}))
$
such that the following assertions hold.
\begin{enumerate}
\item[$(1)$] Let $\tilde u\in \overline{B}_{X_T^{s+\sigma,b}}(R)$ be a solution of uncontrolled system \eqref{Uncontrolled-Problem}. Then the controlled system \eqref{Control-Problem-0} is $(d,q')$-stabilized to $\tilde u$ with respect to $\|\cdot\|_{_{\tilde{H}^s}}$, in the sense of Definition {\rm\ref{Definition-control}}.

\item[$(2)$] The control verifying {\rm(\ref{Squeezing-0})} can be represented as
$$
\xi=\Phi(\tilde u)(u_0-\tilde u_0).
$$
Moreover, $\Phi$ is Lipschitz and continuously differentiable.
\end{enumerate}
\end{theorem}

\begin{remark}\label{Remark-norm}
The new norm $\|\cdot \|_{_{\tilde{H}^s}}$ will be defined in Lemma~{\rm\ref{Lemma-Hslineardecay}}, possessing the property that for any $T>0$, there exists a constant $q_0\in (0,1)$ so that
\[\|S_a(T)\|_{_{\mathcal L(\tilde{H}^s)}}\leq q_0.\]
One should distinguish this from \eqref{Decay-semigroup}, which merely infers that $S_a(T)$ is a contraction with respect to the standard $H^s$-norm for $T$ sufficiently large.\footnote{The existence of such equivalent norm $\|\cdot\|_{_{\tilde{H}^s}}$ follows from an abstract construction; see \cite[Chapter 1, Theorem 5.2]{Paz-83}. Indeed, due to \eqref{Decay-semigroup}, we can define this equivalent norm on $H^s$ by 
\[\interleave u\interleave :=\sup_{t\ge 0} e^{\beta t}\|S_a(t) u_0\|_{_{H^s}}.\]
Alternatively, in Lemma~\ref{Lemma-Hslineardecay} we provide a more explicit construction.}

In fact, the construction of $\|\cdot\|_{_{\tilde{H}^s}}$ is as follows. For $s=2n$ with $n\in \mathbb{N}$, we define
\[\|f\|_{_{\tilde{H}^{2n}}}:=\sum_{k=0}^n A^{n-k} \|(i\partial_x^2-a(x))^k f\|_{_{L^2}},\]
where $A>0$ is sufficiently large, depending only on $n$.
While for $s\in (2n,2n+2)$, we define $\tilde{H}^s$ as the complex interpolation space $(\tilde{H}^{2n},\tilde{H}^{2n+2})_{\theta}$, where $\theta\in (0,1)$ so that $s=(1-\theta)(2n)+\theta(2n+2)$. See, e.g. {\rm\cite{BL-76}} for the definition of complex interpolation method.
\end{remark}

By standard arguments, it suffices to
consider linearization along uncontrolled solution $\tilde u(t)$: 
\begin{equation}\label{Control-Problem-1}
\left\{\begin{array}{ll}
iv_t+v_{xx}+ia(x)v=\frac{p+1}{2}|\tilde u|^{p-1}v+\frac{p-1}{2}|\tilde u|^{p-3}\tilde u^2\bar v+\chi(x)\mathcal P_N\xi(t,x),\\
v(0,x)=v_0(x).
\end{array}
\right.
\end{equation}
In view of Proposition \ref{Proposition-linearproblem}(1), system (\ref{Control-Problem-1}) admits a unique solution in $ X^{s,b}_T$ 
for every $\tilde u\in X_T^{s,b}$, $v_0\in H^{s}$ and $\xi\in L^2_tH_x^s$. The solution is denoted by $\mathcal V_{\tilde u}(v_0,\chi\mathcal P_N\xi)$.

\begin{proposition}\label{Proposition-control-1}
Let $T,R,\sigma>0$ and $s\ge 1$ be arbitrarily given.
Then there exist constants $N\in\N^+$, $q_1\in (0,1)$, an equivalent norm $\|\cdot\|_{_{\tilde{H}^s}}$ on $H^s$, and a map
$
\Phi\colon \overline{B}_{X_T^{s+\sigma,b}}(R)\rightarrow \mathcal{L}_{\R}(H^s;L^2(0,T;H^{s}))
$
such that the following assertions hold.
\begin{enumerate}
\item[$(1)$]  For every $\tilde u\in \overline{B}_{X_T^{s+\sigma,b}}(R)$ and $v_0\in H^s$, there is a control $\xi\in L^2(0,T;H^{s})$ satisfying
\begin{equation}\label{Squeezing-1}
\|v(T)\|_{_{\tilde{H}^s}}\leq q_1 \|v_0\|_{_{\tilde{H}^s}},
\end{equation}
where $v=\mathcal V_{\tilde u}(v_0,\chi\mathcal P_N\xi)$.

\item[$(2)$]  The control $\xi$ verifying {\rm(\ref{Squeezing-1})} can be represented as 
$$
\xi=\Phi(\tilde u)v_0.
$$
Moreover, $\Phi$ is Lipschitz and continuously differentiable.
\end{enumerate}
\end{proposition}

\begin{remark}
Actually, the controls $\xi$ in Theorem {\rm\ref{Theorem-control}} and Proposition {\rm\ref{Proposition-control-1}} are the same, and possess the spatial regularity $H^{s+\sigma}$ higher than the phase space $H^s$. This relies on the $H^{s+\sigma}$-regularity of $\tilde u(t)$ and
plays an essential role in proving the stabilization along trajectory. 

Furthermore, it is worth emphasizing that the extra regularity of $\tilde u(t)$ is a ``sharp'' sufficient condition for the control property; see Section {\rm\ref{Section-scheme}} for more information. This reflects in nature the optimality of exponential asymptotic compactness in the problem of exponential mixing {\rm(}see also {\rm\cite{LWX-24}} for the case of wave equations{\rm)}. 
\end{remark}

The proof of Theorem \ref{Theorem-control} via Proposition \ref{Proposition-control-1} is based on a standard argument of perturbation, with $q'$ chosen to satisfy $q'\in (q_1,1)$; see Appendix \ref{Appendix-lineartest} for the details. Meanwhile, the analysis on Proposition \ref{Proposition-control-1} constitutes the bulk of this section. An overview of the proof will be provided in Section~\ref{Section-scheme}, and the later Sections~\ref{Section-LF}--\ref{Section-proof} include the details.

\subsection{Scheme of proof}\label{Section-scheme}

The underlying idea for Proposition \ref{Proposition-control-1} follows the strategy of frequency analysis. More precisely, one can split the system into {\it low-frequency} (LF)
system coupled with {\it high-frequency} (HF) system. The former is finite-dimensional and obtained by truncating the equation in low Fourier modes, while the latter is an infinite-dimensional system consisting of high Fourier modes.
Recall the subspace of low frequencies is denoted with $H_m={\rm span}\{e_k;|k|\le m\}$, the orthogonal projection to $H_m$ is $P_m$, and $Q_m=I-P_m$.

Our strategy of proof contains three parts:

\medskip

\noindent $\bullet$ {\bf LF controllability}. We initially find an $H^s$-control $\xi$ steering LF of system (\ref{Control-Problem-1}) to zero, i.e. 
\begin{equation}\label{LFC}
P_mv(T)=0.
\end{equation}
In comparison, the controllability in the full frequency (i.e.~$v(T)=0$) is unavailable, as the control in (\ref{Control-Problem-1}) is finite-dimensional. It is worth emphasizing that the control could unavoidably cause mystery in the HF dynamics. So, more efforts are needed.

\medskip

\noindent $\bullet$ {\bf HF dissipation in uncontrolled case}. Our discussion on the HF system starts with the uncontrolled case. We construct an equivalent norm $\|\cdot\|_{_{\tilde{H}^s}}$ on $H^s$ so that $S_a(T)$ is a contraction. Then provided $\tilde{u}(t)$ possesses a higher regularity $H^{s+\sigma}$, it will be derived that
\begin{equation}\label{highfreqdecay}
\|Q_m(v(T)-z(T))\|_{_{\tilde{H}^s}}\leq \varepsilon \|v_0\|_{_{\tilde{H}^s}},
\end{equation}
with $v=\mathcal V_{\tilde u}(v_0,0)$ and any $\varepsilon\in (0,1)$. Here, $m$ is sufficiently large, and $z$ is an auxiliary function defined in \eqref{Linear-problem-1}, satisfying that
\begin{equation}\label{z-squeezing}
\|z(T)\|_{_{\tilde H^s}}\le q_0\|v_0\|_{_{\tilde{H}^s}}.
\end{equation}

\medskip

\noindent $\bullet$ {\bf Influence of control on HF}. It remains to ensure the validity of (\ref{highfreqdecay}) for controlled solution of (\ref{Control-Problem-1}). For this purpose, we 
invoke again the extra regularity imposed on $\tilde u(t)$, and construct a control $\xi$ with higher regularity $H^{s+\sigma}$,
which indicates that $\xi$ does not act essentially on the HF. On the other hand, instead of null controllability (\ref{LFC}) the LF system verifies
\begin{equation}\label{LF-squeezing}
P_m v(T)=P_m z(T).
\end{equation}
Gathering \eqref{highfreqdecay}-\eqref{LF-squeezing}, we can conclude the proof of Proposition~\ref{Proposition-control-1}; the information on the structure of $\xi$ is included in its construction.

\begin{remark}
The aforementioned effects of control to system \eqref{Control-Problem-1} can be roughly summarized as:

\begin{table}[H]
\centering
\begin{tabular}{c|ccc}
&No control&$H^s$-control&$H^{s+\sigma}$-control \\
\hline 
LF&unknown&\hspace{2mm}controllable to $0$\hspace{2mm}&controllable to $z(T)$\\
HF&stabilizable to $z(T)$&unknown&stabilizable to $z(T)$
\end{tabular}
\end{table}
\noindent For more explanations, see Part I-III of the scheme below. 
\end{remark}

\medskip

\noindent{\it Part I: LF controllability.} 
Invoking the Hilbert uniqueness method (HUM), the problem of LF controllability (\ref{LFC}) is translated to whether the following truncated observability
\begin{equation}\label{obs-type}
\int_0^T\|\mathcal P_N(\chi\varphi)\|^2_{_{H^{-s-\sigma'}}}\gtrsim
\|\varphi_T\|^2_{_{ H^{-s-\sigma'}}}\quad {\rm with\ }\sigma'\in[0,\sigma]
\end{equation}
holds for an adjoint system of (\ref{Control-Problem-1}), reading
\begin{equation}\label{Adjoint-problem-1}
\left\{\begin{array}{ll}
i\varphi_t+\varphi_{xx}-ia(x)\varphi=\frac{p+1}{2}|\tilde u|^{p-1}\varphi-\frac{p-1}{2}|\tilde u|^{p-3}\tilde u^2\bar \varphi,\\
\varphi(T,x)=\varphi_T(x).
\end{array}
\right.
\end{equation}
The inequality (\ref{obs-type}) with $\sigma'=0$ indicates that one can arrive at (\ref{LFC}) via an $H^s$-control for every $v_0\in H^s$. However, 
the regularity of control $\xi$, constructed in the present part, could not exceed the regularity of state $v_0$. This makes it difficult for us to obtain useful information on the HF dynamics. To this end, we exploit the observability with $\sigma'>0$, which is referred to as ``low-regularity'' observability, and enables us to obtain an $H^{s+\sigma'}$-control verifying (\ref{LFC}) when the state $v_0$ is of $H^{s+\sigma'}$. This  controllability in a bit more regular space $H^{s+\sigma'}$ will be used in establishing the HF dissipation.

In order to derive (\ref{obs-type}), we adapt the argument of \cite[Proposition 7.1]{Laurent-10} (which applies directly to the case of $s=1$ and $\sigma'=0$), and obtain a ``full'' version of observability:
\begin{equation}\label{fullobs-type}
\int_0^T\|\chi\varphi\|^2_{_{H^{-s-\sigma'}}}\gtrsim
\|\varphi_T\|^2_{_{ H^{-s-\sigma'}}}.
\end{equation}
In this step, the  $H^{s+\sigma}$-regularity of $\tilde u(t)$ and several microlocal-analysis characterizations on linear Schr\"{o}dinger equation are needed. Finally, the deduction of (\ref{obs-type}) via (\ref{fullobs-type}) follows by approximation of type $\mathcal P_N(\chi\varphi)\rightarrow \chi\varphi$. See Section \ref{Section-LF} for more details.

\medskip

\noindent{\it Part II: HF dissipation in uncontrolled case.} The main technical challenge of (\ref{highfreqdecay}) lies in the fact that the Duhamel evolution for the Schr\"{o}dinger equation has no smoothing effect. 

Because of our emphasis on underlying ideas, let us restrict ourselves for the moment on the case of $p=3$. As the control $\mathcal P_N\xi$ does not act essentially on the HF (see Part III later), the essential problem arises in  uncontrolled case:
\begin{equation*}
\left\{\begin{array}{ll}
iv_t+v_{xx}+ia(x)v=2|\tilde u|^{2}v+\tilde u^2\bar v,\\
v(0,x)=v_0(x).
\end{array}
\right.
\end{equation*}
The stability of this linearized equation is hard to characterize, due to the presence of potential terms $2|\tilde{u}|^2 v+\tilde{u}^2 \bar{v}$. Specifically, an $H^s$-potential $\tilde{u}(t)$ could not guarantee \eqref{highfreqdecay}; see Remark~\ref{counterexample}. To our surprise, we observe that the resonant decomposition in the analysis of asymptotic compactness (Section \ref{Section-AC-1}) is applicable to this issue. This allows us to derive that if $\tilde{u}(t)$ is of $H^{s+\sigma}$, the HF dissipation \eqref{highfreqdecay} is valid.

More precisely,
the potential terms can be decomposed as 
$$
2|\tilde u|^2v+\tilde u^2\bar v=\frac{1}{\pi}\|\tilde u\|^2_{_{L^2}}v+\frac{2}{\pi}\tilde u{\rm Re}\,(\tilde u,v)+\mathcal R(\tilde u,v),
$$
where 
$$
\mathcal R(\tilde u,v)=2\mathcal T_{N}(\tilde u,\tilde u,v)+\mathcal T_{N}(\tilde u,v,\tilde u);
$$ 
recall $\mathcal T_N$ is defined by (\ref{smoothing-frequency}). If $\tilde u(t)$ is of $H^{s+\sigma}$, 
it possesses the smallness  in the HF, which together with the nonlinear smoothing of $\mathcal R(\tilde u,v)$ (see Lemma \ref{Lemma-smoothing}) gives rise to the smallness of 
$$
Q_m\left[\frac{2}{\pi}\tilde u{\rm Re}\,(\tilde u,v)+\mathcal R(\tilde u,v)\right]\quad {\rm in\ }H^s.
$$
Observe in addition that the term $\frac{1}{\pi}\|\tilde u\|^2_{_{L^2}}v$ can be removed by a norm-preserving transformation as in (\ref{w-variable}). This, to be combined with the spectral gap of damped Schr\"odinger operator $i\partial^2_x-a(x)$, yields a ``HF spectral gap'' of the perturbed operator 
$$v\to iv_{xx}-a(x)v-i(2|\tilde u|^2v+\tilde u^2\bar v).$$
See Section \ref{Section-HF} for more details.  

\medskip

\noindent{\it Part III: Influence of control on HF.} Heuristically, the proof of Proposition \ref{Proposition-control-1} would conclude by combining LF controllability and HF dissipation established in the last two parts. However, the control must be taken into concern for the HF dissipation. In order to ensure that the contribution of control is harmless, we need $\xi$ to possess higher regularity $ H^{s+\sigma}$, in spite of the $H^s$-regularity of  state $v_0$. This seems incompatible with the LF controllability at first sight, since the control $\xi$ verifying (\ref{LFC}) has the same regularity as $v_0$.

To make up for this discrepancy, we reproduce a trick used in \cite{LWX-24}. Let $z(t)$ be the solution of an auxiliary linear equation
\begin{equation}\label{Linear-problem-1}
\left\{\begin{array}{ll}
iz_t+ z_{xx}+ia(x)z=\frac{p+1}{4\pi}\|\tilde u\|^{p-1}_{_{L^{p-1}(\T)}}z,\\
z(0,x)=v_0(x).
\end{array}\right.
\end{equation}
Then we deduce a variant of LF controllability (\ref{LFC}): there exists $\xi\in L_t^2 H^{s+\sigma}_x$ such that (\ref{LF-squeezing}) holds.
Owing to the extra regularity of $\xi$, the HF dissipation (\ref{highfreqdecay}) remains valid. 

Finally, in view of \eqref{highfreqdecay} and \eqref{LF-squeezing}, it remains to address (\ref{z-squeezing}). 
This follows from the contraction of $S_a(T)$ in the norm $\|\cdot\|_{_{\tilde{H}^s}}$ (see Remark \ref{Remark-norm}), combined with $\|z(T)\|_{_{\tilde{H}^s}}=\|S_a(T) v_0\|_{_{\tilde{H}^s}}$. In comparison, though $S_a(t)$ decays exponentially in the standard norm $\|\cdot\|_{_{H^s}}$, it is unknown whether the operator norm at any $T$ is strictly less than $1$. 

\medskip

\noindent{\it Convention: In the remainder of this section, unless otherwise stated, the generic constant $C$ used in the proofs would not depend on the parameters $m,N$. Meanwhile, its dependence on other parameters will not be mentioned explicitly if there is no danger of confusion.}

\subsection{Low-frequency controllability}\label{Section-LF}

In this section, we prove the LF null controllability (\ref{LFC}) for system (\ref{Control-Problem-1}); see Proposition \ref{Proposition-LFC} later. In view of Proposition \ref{Proposition-linearproblem}(2), the adjoint system (\ref{Adjoint-problem-1}) admits a unique solution in $X^{-s,b}_T$, denoted by $\mathcal U_{\tilde u}(\varphi_T)$, for any $\tilde u\in X_T^{s,b}$ and $\varphi_T\in H^{-s}$. 

We begin with the following result.

\begin{lemma}[Full observability]\label{Lemma-fullobs}
Let $T,R,\sigma>0$, $s\ge 1$ and $\sigma'\in [0,\sigma]$ be arbitrarily given. Then there exists a constant $C_4>0$ such that
\begin{equation}\label{Full-observability}
\int_0^T\|\chi\varphi\|^2_{_{H^{-s-\sigma'}}}\geq C_4\|\varphi_T\|^2_{_{ H^{-s-\sigma'}}}
\end{equation}
for any $\varphi_T\in H^{-s-\sigma'}$, where 
$\varphi=\mathcal U_{\tilde u}(\varphi_T)$ with $\tilde u\in \overline{B}_{X_T^{s+\sigma,b}}(R)$.
\end{lemma}

When $s=1,\,p=3$ and $\sigma=\sigma'=0$, the observability of type (\ref{Full-observability}) has been obtained in \cite[Proposition 7.1]{Laurent-10}, where the author considered a more general setting of compact Riemannian manifold. The underlying idea used there is also valid for the problem that we are considering, except for some technical adaptations. See Appendix \ref{Appendix-control} for a proof of Lemma \ref{Lemma-fullobs}.

\medskip

The observability (\ref{Full-observability}) leads to null controllability $v(T)=0$ for system (\ref{Control-Problem-1}) with $\mathcal P_N$ replaced by the identity. This can be justified by the Dolecki--Russell duality between controllability and observability (or the HUM); see the monograph \cite{Coron-07}.

Analogously, the LF controllability is related to a truncated version of observability. 

\begin{lemma}[Truncated observability]\label{Lemma-obs}
Let $T,R,\sigma>0$, $s\ge 1$ and $\sigma'\in [0,\sigma]$ be arbitrarily given, and the constant $C_4$ established in Lemma  {\rm\ref{Lemma-fullobs}}. Then 
for every $m\in\N^+$,
there exists a constant $N\in\N^+$ such that
\begin{equation}\label{observability-verification}
\int_0^T\|\mathcal P_N(\chi\varphi)\|^2_{_{H^{-s-\sigma'}}}\geq \frac{C_4}{2}\|\varphi_T\|^2_{_{ H^{-s-\sigma'}}}
\end{equation}
for any $\varphi_T\in H_m$, where 
$\varphi=\mathcal U_{\tilde u}(\varphi_T)$ with $\tilde u\in \overline{B}_{X_T^{s+\sigma,b}}(R)$.
\end{lemma}

\begin{proof}[{\bf Proof of Lemma \ref{Lemma-obs}}]
To begin with, let us arbitrarily fix $0<\varepsilon\le \min \{\sigma,b\}$. We claim that there exists a constant $C(m)>0$ (omitting the dependence on $T,R$) such that for any $\varphi_T\in H_m$,
\begin{equation}\label{bound-2}
\|\chi\varphi\|^2_{_{X^{1,\varepsilon}_T}}\leq C(m)\int_0^T\|\chi\varphi\|^2_{_{H^{-s-\sigma'}}},
\end{equation}
where 
$\varphi=\mathcal U_{\tilde u}(\varphi_T)$. Indeed, noticing that $H^{1+\varepsilon}$- and $H^{-s-\sigma'}$-norms are equivalent in $H_m$, and $\tilde{u}\in X_T^{s+\sigma,b}$ ensures that $\varphi$ is a solution in $X_T^{1+\varepsilon,b}$; in view of Lemma~\ref{Lemma-multiplication}, Proposition \ref{Proposition-linearproblem}(1) and Lemma \ref{Lemma-fullobs}, we thus obtain
$$
\|\chi\varphi\|^2_{_{X^{1,\varepsilon}_T}}\le C\|\varphi\|_{_{X_T^{1+\varepsilon
,\varepsilon}}}^2\le C\|\varphi_T\|_{_{H^{1+\varepsilon}}}^2\le C(m)\|\varphi_T\|_{_{H^{-s-\sigma'}}}^2\le C(m)\int_0^T\|\chi\varphi\|^2_{_{H^{-s-\sigma'}}}.
$$

Owing to compact embedding $X^{1,\varepsilon}_T\Subset X_T^{-s-\sigma',0}=L^2_tH^{-s-\sigma'}_x$, there exists $\alpha_N\rightarrow \infty$ such that
$$
\int_0^T\|(1-\mathcal P_N)\phi\|^2_{_{H^{-s-\sigma'}}}\leq \alpha_N^{-1}\|\phi\|^2_{_{X^{1,\varepsilon}_T}}
$$
for any $\phi\in X^{1,\varepsilon}_T$.
This, combined with (\ref{bound-2}), implies that
\begin{align*}
\int_0^T\|\chi\varphi\|^2_{_{H^{-s-\sigma'}}}&\leq \int_0^T\|\mathcal P_N(\chi\varphi)\|^2_{_{H^{-s-\sigma'}}}+\int_0^T\|(1-\mathcal P_N)(\chi\varphi)\|^2_{_{H^{-s-\sigma'}}}\\
&\leq \int_0^T\|\mathcal P_N(\chi\varphi)\|^2_{_{H^{-s-\sigma'}}}+\alpha_N^{-1}\|\chi\varphi\|^2_{_{X^{1,b}_T}}\\
&\leq \int_0^T\|\mathcal P_N(\chi\varphi)\|^2_{_{H^{-s-\sigma'}}}+C(m)\alpha_N^{-1}\int_0^T\|\chi\varphi\|^2_{_{H^{-s-\sigma'}}}.
\end{align*}
Letting $N$ be sufficiently large (depending on $m$) so that $C(m)\alpha_N^{-1}\leq 1/2$, one gets easily that
$$
\int_0^T\|\chi\varphi\|^2_{_{H^{-s-\sigma'}}}\leq 2\int_0^T\|\mathcal P_N(\chi\varphi)\|^2_{_{H^{-s-\sigma'}}}.
$$
Finally, taking (\ref{Full-observability}) into account, inequality (\ref{observability-verification}) follows immediately. 
\end{proof}

It is time to establish the LF controllability (\ref{LFC}) for system (\ref{Control-Problem-1}). To this end, we introduce a real-valued functional $J$ on $H_m$ by
$$
J(\varphi_T)=\frac{1}{2}\int_0^T\|\mathcal P_N(\chi\varphi)\|^2_{_{H^{-s-\sigma'}}}+{\rm Re}\,(v_0,\varphi(0)),\quad \varphi_T\in H_m,
$$
where $v_0\in H^{s+\sigma'}$ (representing initial value for \eqref{Control-Problem-1}) and $\tilde{u}\in X_T^{s+\sigma,b}$ are given, and $\varphi=\mathcal U_{\tilde u}(\varphi_T)$.

\begin{proposition}[Low-frequency controllability]\label{Proposition-LFC}
Let $T,R,\sigma>0$, $s\ge 1$, $\sigma'\in [0,\sigma]$ and $m\in\N^+$ be arbitrarily given. Then there exists a constant $N\in\N^+$ such that the following assertions hold.
\begin{enumerate}
\item[$(1)$] For every $\tilde u\in \overline{B}_{X_T^{s+\sigma,b}}(R)$ and $v_0\in H^{s+\sigma'}$, the functional $J\colon H_m\to \mathbb{R}$ admits a unique global minimizer $\breve\varphi_{T}\in H_m$.
	
\item[$(2)$] There exists a constant $C>0$, independent of $m,N$, such that 
$$
P_m\breve v(T)=0\quad \text{and}\quad \int_0^T\|\breve\xi\|^2_{_{H^{s+\sigma'}}}\leq C\|v_0\|^2_{_{H^{s+\sigma'}}},
$$
where $\breve v:=\mathcal V_{\tilde u}(v_0,\chi\mathcal P_N\breve\xi)$ and 
$
\breve\xi:=i(1-\partial_{x}^2)^{-s-\sigma'}[\mathcal P_N(\chi\breve\varphi)] 
$ with $\breve\varphi=\mathcal U_{\tilde u}(\breve\varphi_T).$

\item[$(3)$] For every $\tilde u\in \overline{B}_{X_T^{s+\sigma,b}}(R)$, the HUM-type control map, defined by $$\Lambda(\tilde u)\colon H^{s+\sigma'}\rightarrow H^{-s-\sigma'},\quad\Lambda(\tilde u)(v_0)=\breve\varphi_T,$$ is $\R$-linear. Moreover, the map 
$
\overline{B}_{X_T^{s+\sigma,b}}(R)\ni \tilde u\mapsto \Lambda(\tilde u)\in\mathcal{L}_{\R}(H^{s+\sigma'},H^{-s-\sigma'})
$
is Lipschitz and continuously differentiable.
\end{enumerate}
\end{proposition}

As stated in Part I of Section \ref{Section-scheme}, 
the second conclusion with $\sigma'=0$ indicates the LF controllability for system (\ref{Control-Problem-1}) in $H^s$. The conclusions with $\sigma'>0$ will be useful in establishing the HF dissipation; see Section \ref{Section-proof} later.

\begin{proof}[{\bf Proof of Proposition \ref{Proposition-LFC}}]
Let $N\in\N^+$ be established in Lemma \ref{Lemma-obs}, and fix $\tilde u\in \overline{B}_{X_T^{s+\sigma,b}}(R)$ and $v_0\in H^{s+\sigma'}$. Thanks to the observability inequality \eqref{observability-verification}, it can be verified that
\[ J(\varphi_T)\rightarrow\infty\quad \text{as }\|\varphi_T\|_{_{H^{-s-\sigma'}}}\to \infty.\]
Moreover, $J$ is clearly continuous and convex. As a result, the minimizer of $J$ exists.

We proceed to verify the uniqueness of minimizer. Assume that $\breve\varphi_T,\tilde\varphi_T\in H_m$ are  minimizers of $J$, and denote $\breve\varphi=\mathcal U_{\tilde u}(\breve\varphi_T)$ and $\tilde \varphi=\mathcal U_{\tilde u}(\tilde\varphi_T)$. Note that
\begin{equation}\label{parallelogram-1}
\begin{aligned}
&\int_0^T \|\mathcal P_N[\chi(\frac{\breve\varphi-\tilde\varphi}{2})]\|_{_{H^{-s-\sigma'}}}^2+\int_0^T \|\mathcal P_N[\chi(\frac{\breve\varphi+\tilde\varphi}{2})]\|_{_{H^{-s-\sigma'}}}^2\\
&=\frac{1}{2}\int_0^T\|\mathcal P_N(\chi\breve\varphi)\|^2_{_{H^{-s-\sigma'}}}+\frac{1}{2}\int_0^T\|\mathcal P_N(\chi\tilde\varphi)\|^2_{_{H^{-s-\sigma'}}},
\end{aligned}
\end{equation}
by means of the parallelogram law. Adding ${\rm Re}\,(v_0,\breve\varphi(0)+\tilde\varphi(0))$ to (\ref{parallelogram-1}) leads to
\begin{equation}\label{parallelogram-2}
\int_0^T \|\mathcal P_N[\chi(\frac{\breve\varphi-\tilde\varphi}{2})]\|_{_{H^{-s-\sigma'}}}^2+2J(\frac{\breve\varphi_T+\tilde\varphi_T}{2})=2J(\breve\varphi_T).
\end{equation}
One can thus apply Lemma \ref{Lemma-obs} and the minimally of $\breve{\varphi}_T$ to compute that
$$
\text{(LHS of \eqref{parallelogram-2})}\geq \frac{C_4}{2}\|\frac{\breve\varphi_T-\tilde\varphi_T}{2}\|^2_{_{H^{-s-\sigma'}}}+2J(\breve\varphi_T),
$$
which implies $\breve\varphi_T=\tilde\varphi_T$ and the desired uniqueness. Hence conclusion (1) follows.

To prove conclusion (2), multiplying \eqref{Control-Problem-1} by $\bar{\varphi}$ and taking the imaginary part, one can derive the following  dual identity
\begin{equation}\label{duality-identity-1}
{\rm Re}\,\left[ (\breve v(T),\varphi_T)- (v_0,\varphi(0))\right]={\rm Re}\int_0^T(\breve\xi,i\mathcal P_N(\chi\varphi))
\end{equation}
for any $\varphi_T\in H_m$, where 
$\varphi=\mathcal U_{\tilde u}(\varphi_T)$ and $\breve v,\breve\xi$ are defined as in the statement of the proposition. On the other hand, due to the minimality of $\breve{\varphi}_T$, we have
\begin{equation}\label{derivative-0}
0=\lim_{\varepsilon\to 0} \frac{J(\breve{\varphi}_T+\varepsilon \varphi_T)-J(\breve{\varphi}_T)}{\varepsilon}={\rm Re}\int_0^T(\breve\xi,i\mathcal P_N(\chi\varphi))+{\rm Re}\,(v_0,\varphi(0))
\end{equation}
(in view of the construction of $\breve\xi$). Substituting (\ref{derivative-0}) into (\ref{duality-identity-1}), it follows that
$$
{\rm Re}\,(\breve v(T),\varphi_T)=0.
$$
Noticing in addition that ${\rm Im}\,(\breve v(T),\varphi_T)={\rm Re}\,(\breve v(T),i\varphi_T)$, one can thus conclude that 
$$
(\breve v(T),\varphi_T)=0
$$
for any $\varphi_T\in H_m$, and hence $P_m\breve v(T)=0$.

At the same time, taking $\varphi_T=\breve\varphi_T$ in (\ref{derivative-0}) enables us to find that 
\begin{equation}\label{derivative-1}
\int_0^T\|\breve\xi\|^2_{_{H^{s+\sigma'}}}=-{\rm Re}\,(v_0,\breve\varphi(0)).
\end{equation}
The combination of Proposition \ref{Proposition-linearproblem}(2) and Lemma \ref{Lemma-obs} leads to
$$
\text{(RHS of (\ref{derivative-1}))}\leq C \|v_0\|_{_{H^{s+\sigma'}}}\|\breve\varphi_T\|_{_{H^{-s-\sigma'}}}\leq C \|v_0\|_{_{H^{s+\sigma'}}}^2+\frac{1}{2}\int_0^T\|\breve\xi\|^2_{_{H^{s+\sigma'}}},
$$
where the constant $C$ does not depend on $m,N$. This completes the proof of conclusion (2).

It remains to prove conclusion (3). Let $v_0,v_0'\in H^{s+\sigma'}$ and $\alpha,\beta\in\R$. We then denote
$$
\breve\varphi_T=\Lambda(\tilde u)(v_0),\quad \breve\varphi_T'=\Lambda(\tilde u)(v_0'),\quad \breve\varphi_T''=\Lambda(\tilde u)(\alpha v_0+\beta v_0')
$$
and define $\breve\varphi,\breve\varphi',\breve\varphi''$ to be $\mathcal U_{\tilde u}(\varphi_T)$ with $\varphi_T=\breve\varphi_T,\breve\varphi_T',\breve\varphi_T''$, respectively. 
Invoking (\ref{derivative-0}) again,
\begin{align}
&{\rm Re}\int_0^T(\mathcal P_N(\chi\breve\varphi),\mathcal P_N(\chi\varphi))_{_{H^{-s-\sigma'}}}+{\rm Re}\,(v_0,\varphi(0))=0,\label{formula-1}\\
&{\rm Re}\int_0^T(\mathcal P_N(\chi\breve\varphi'),\mathcal P_N(\chi\varphi))_{_{H^{-s-\sigma'}}}+{\rm Re}\,(v_0',\varphi(0))=0,\label{formula-2}\\
&{\rm Re}\int_0^T(\mathcal P_N(\chi\breve\varphi''),\mathcal P_N(\chi\varphi))_{_{H^{-s-\sigma'}}}+{\rm Re}\,(\alpha v_0+\beta v_0',\varphi(0))=0\label{formula-4}
\end{align}
for any $\varphi_T\in H_m$, where 
$\varphi=\mathcal U_{\tilde u}(\varphi_T)$.
Then, computing $$\alpha\cdot(\ref{formula-1})+\beta\cdot(\ref{formula-2})-(\ref{formula-4})$$ and taking $\varphi_T=\alpha\breve\varphi_T+\beta\breve\varphi_T'-\breve\varphi_T''$, we conclude that
$$
\int_0^T\|\mathcal P_N[\chi(\alpha\breve\varphi+\beta\breve\varphi'-\breve\varphi'')]\|^2_{_{H^{-s-\sigma'}}}=0.
$$
This together with Lemma \ref{Lemma-obs} implies $\breve\varphi_T''=\alpha\breve\varphi_T+\beta\breve\varphi_T'$. Hence $\Lambda(\tilde u)\in \mathcal{L}_{\R}(H^{s+\sigma'},H^{-s-\sigma'})$.

Below is to investigate the map $\tilde u\mapsto\Lambda(\tilde u)$. From Proposition \ref{Proposition-linearproblem}(2) it follows that
\begin{equation}\label{Lipschitz-1}
\|\mathcal U_{\tilde u_1}(\varphi_T)-\mathcal U_{\tilde u_2}(\varphi_T)\|_{_{X_T^{-s-\sigma',b}}}\leq C \|\tilde u_1-\tilde u_2\|_{_{X_T^{s+\sigma,b}}}\|\varphi_T\|_{_{H^{-s-\sigma'}}}
\end{equation}
for any $\tilde u_1,\tilde u_2\in \overline{B}_{X_T^{s+\sigma,b}}(R)$ and $\varphi_T\in H^{-s-\sigma'}$. On the other hand, denoting 
$$
\breve\varphi_T^l=\Lambda(\tilde u_l)(v_0),\quad l=1,2
$$
with a given $v_0\in H^{s+\sigma'}$,
identity (\ref{derivative-0}) implies also that
\begin{align*}
& {\rm Re}\int_0^T(\mathcal P_N[\chi\mathcal U_{\tilde u_1}(\breve\varphi_T^1)],\mathcal P_N[\chi\mathcal U_{\tilde u_1}(\varphi_T)])_{_{H^{-s-\sigma'}}}\\
&-{\rm Re}\int_0^T(\mathcal P_N[\chi\mathcal U_{\tilde u_2}(\breve\varphi_T^2)],\mathcal P_N[\chi\mathcal U_{\tilde u_2}(\varphi_T)])_{_{H^{-s-\sigma'}}}\\
&+{\rm Re}\left[
(v_0,\mathcal U_{\tilde u_1}(\varphi_T)(0)-\mathcal U_{\tilde u_2}(\varphi_T)(0))
\right]=0
\end{align*}
for any $\varphi_T\in H_m$. Taking $\varphi_T=\breve\varphi_T^1-\breve\varphi_T^2$, we thus derive that
\begin{align*}
& \int_0^T
\|\mathcal P_N[\chi\mathcal U_{\tilde u_1}(\breve\varphi_T^1-\breve\varphi_T^2)]\|^2_{_{H^{-s-\sigma'}}}
\\
&=-{\rm Re}\int_0^T(\mathcal P_N[\chi(\mathcal U_{\tilde u_1}-\mathcal U_{\tilde u_2})(\breve\varphi_T^2)],\mathcal P_N[\chi\mathcal U_{\tilde u_1}(\breve\varphi_T^1-\breve\varphi_T^2)])_{_{H^{-s-\sigma'}}}\\
&\quad -{\rm Re}\int_0^T(\mathcal P_N[\chi\mathcal U_{\tilde u_2}(\breve\varphi_T^2)],\mathcal P_N[\chi(\mathcal U_{\tilde u_1}-\mathcal U_{\tilde u_2})(\breve\varphi_T^1-\breve\varphi_T^2)])_{_{H^{-s-\sigma'}}}\\
&\quad-{\rm Re}\left[
(v_0,\mathcal U_{\tilde u_1}(\breve\varphi_T^1-\breve\varphi_T^2)(0)-\mathcal U_{\tilde u_2}(\breve\varphi_T^1-\breve\varphi_T^2)(0))
\right].
\end{align*}
The integral in the LHS can be dealt with by Lemma \ref{Lemma-obs}, i.e.
$$
\int_0^T
\|\mathcal P_N[\chi\mathcal U_{\tilde u_1}(\breve\varphi_T^1-\breve\varphi_T^2)]\|^2_{_{H^{-s-\sigma'}}}
\geq \frac{C_4}{2} \|\breve\varphi_T^1-\breve\varphi_T^2\|^2_{_{H^{-s-\sigma'}}}.
$$
At the same time, those terms in the RHS are bounded above by 
\begin{equation*}
\begin{aligned}
&C\left(\|\tilde u_1-\tilde u_2\|_{_{X_T^{s+\sigma,b}}}\|\breve\varphi_T^2\|_{_{H^{-s-\sigma'}}}\|\breve\varphi_T^1-\breve\varphi_T^2\|_{_{H^{-s-\sigma'}}}+\|\tilde u_1-\tilde u_2\|_{_{X_T^{s+\sigma,b}}}\|v_0\|_{_{H^{s+\sigma'}}}\|\breve\varphi_T^1-\breve\varphi_T^2\|_{_{H^{-s-\sigma'}}}\right)\\
&\leq \frac{C_4}{4}\|\breve\varphi_T^1-\breve\varphi_T^2\|_{_{H^{-s-\sigma'}}}^2+C\|\tilde u_1-\tilde u_2\|_{_{X_T^{s+\sigma,b}}}^2\|v_0\|_{_{H^{s+\sigma'}}}^2,
\end{aligned}
\end{equation*}
in view of (\ref{Lipschitz-1}),
where the constant $C$ does not depend on $v_0,\tilde u_1,\tilde u_2$. Therefore,
$$
\|\breve\varphi_T^1-\breve\varphi_T^2\|_{_{H^{-s-\sigma'}}}\le C\|\tilde u_1-\tilde u_2\|_{_{X_T^{s+\sigma,b}}}\|v_0\|_{_{H^{s+\sigma'}}},
$$
which means the Lipschitz continuity of $\tilde u\mapsto\Lambda(\tilde u)$. Finally, the continuous differentiability of the map can be derived by combining identity (\ref{derivative-0}) with Proposition \ref{Proposition-linearproblem}(2) and the implicit function theorem (cf. \cite[Proposition 5.5]{Shi-15}). The proof is then complete.
\end{proof}

\subsection{High-frequency dissipation}\label{Section-HF}

We now consider the issue of HF dissipation (\ref{highfreqdecay}). Let us first introduce an equivalent norm $\tilde{H}^s$, which guarantees that $S_a(T)$ is a contraction.

\begin{lemma}\label{Lemma-Hslineardecay}
Let $s\ge 0$ be arbitrarily given. Then the Sobolev space $H^s$ admits an equivalent norm $\|\cdot\|_{_{\tilde{H^s}}}$ such that for every $T>0$, there exists a constant $q_0\in(0,1)$ satisfying
\begin{equation*}
\|S_a(T)\|_{_{\mathcal L(\tilde{H}^s)}}\leq q_0.
\end{equation*}
\end{lemma}

\begin{proof}[{\bf Proof of Lemma \ref{Lemma-Hslineardecay}}]
We first consider $s=0$. Denote  $L_a=i\partial_x^2-a(x)$. We invoke another observability inequality (or flux estimate) at the scale of $L^2$:
    
\begin{lemma}\label{Lemma-L^2observable}
Let $T>0$ be arbitrarily given. Then there exists a constant $c>0$ such that 
\[\int_{Q_T} a(x)|u|^2 \ge c\|u_0\|_{_{L^2}}^2\]
for any $u_0\in L^2(\T)$,
where $u(t)$ stands for the solution of damped linear Schr\"odinger equation $u_t=L_a u$ with initial condition $u(0,x)=u_0(x)$. 
\end{lemma}
Similar $L^2$-observability inequalities are well-known for the free linear Schr\"odinger equation (see, e.g. \cite{DGL-06,RZ-09}), and also established for the cubic NLS in \cite[Proposition 6.1]{Laurent-ECOCV}. Some technical adaptions are needed for our purpose. For the sake of completeness, we provide a short proof in Appendix~\ref{Appendix-L^2observable}, which is a simplified version of the proof of Lemma~\ref{Lemma-fullobs}.
    
Let us assume the validity of Lemma~\ref{Lemma-L^2observable}, and choose $c\in(0,1/2)$ without loss of generality.
Due to the energy identity \eqref{energy-1}, we have
\[\frac{1}{2}\|u(T)\|_{L^2}^2-\frac{1}{2}\|u(0)\|_{L^2}^2=-\int_{Q_T} a(x)|u|^2 dxdt\le -c \|u_0\|_{L^2}^2.\]
Thus, letting $q_0=\sqrt{1-2c}$,
we find $\|S_a(T)\|_{_{\mathcal L(L^2)}}\le q_0$. In particular, we can take $\|\cdot \|_{_{\tilde{L}^2}}=\|\cdot\|_{_{L^2}}$.

Next, for every $n\in \mathbb{N}^+$, we define an equivalent norm on $H^{2n}$ by
\[\|f\|_{_{\tilde{H}^{2n}}}:=\sum_{k=0}^n A^{n-k} \|L^k_a f\|_{_{L^2}}.\]
Here $A>1$ is a large constant depending on $n$. Indeed, for $1\le k\le n$, it is easy to see that
    \[\|\partial_x^{2n} f\|_{_{L^2}}-C\|f\|_{_{H^{2n-2}}}\le \|L_a^k f\|_{_{L^2}}\le \|\partial_x^{2n} f\|_{_{L^2}}+C\|f\|_{_{H^{2n-2}}}.\]
    Therefore $\|\cdot \|_{_{\tilde{H}^{2n}}}$ is equivalent to $\|\cdot\|_{_{H^{2n}}}$, provided $A$ is large enough. Note that as $L_a$ is the infinitesimal generator of $S_a(t)$, we have $[L_a,S_a(T)]=0$. Thus for any $u_0\in H^{2n}$, it follows that
    \[\|S_a(T) u_0\|_{_{\tilde{H}^{2n}}}=\sum_{k=0}^n A^{n-k} \|S_a(T)L_a^k u_0\|_{_{L^2}}\le q_0 \sum_{k=0}^n A^{n-k} \|L_a^k u_0\|_{_{L^2}}=q_0\|u_0\|_{_{\tilde{H}^{2n}}}.\]

Finally, to deal with $s\in (2n,2n+2)$, we will invoke the interpolation theory of linear operators. For the definition and basic properties of complex interpolation method, see, e.g.~\cite[Section 4.1]{BL-76}. Let $\tilde{H}^s$ be the complex interpolation space $(\tilde{H}^{2n},\tilde{H}^{2n+2})_{\theta}$, where $\theta\in (0,1)$ satisfies
\[s=(2n)(1-\theta)+(2n+2)\theta.\]
Then in view of \cite[Theorem 4.1.2]{BL-76}, we have
\[\|S_a(T)\|_{_{\mathcal L(\tilde{H}^s)}}\le \|S_a(T)\|_{_{\mathcal L(\tilde{H}^{2n})}}^{1-\theta}\|S_a(T)\|_{_{\mathcal L(\tilde{H}^{2n+2})}}^\theta\le q_0.\]
According to the definition of complex interpolation, it readily follows that equivalent normed spaces gives rise to equivalent interpolations. Note that $\tilde{H}^{2n}$ and $\tilde{H}^{2n+2}$ are equivalent to $H^{2n}$ and $H^{2n+2}$, respectively. This, together with 
\[(H^{2n},H^{2n+2})_{\theta}=H^{(2n)(1-\theta)+(2n+2)\theta}=H^s\]
(see, e.g. \cite[Theorem 6.4.5(7)]{BL-76}), indicates that  $\|\cdot\|_{_{\tilde{H}^s}}$ is an equivalent norm of $\|\cdot\|_{_{H^s}}$.
\end{proof}

In the sequel, we deal with the potential terms $\frac{p+1}{2}|\tilde u|^{p-1}v+\frac{p-1}{2}|\tilde u|^{p-3}\tilde u^2\bar v$. 
Our analysis for this purpose is inspired by 
the resonant decomposition that has been invoked in Section \ref{Section-AC-1}. More precisely, the term $|\tilde u|^{p-1}v$ can be rewritten as
\begin{align*}
|\tilde u|^{p-1}v&=\mathcal T(v,\tilde u,\cdots, \tilde u)=\frac{1}{2\pi}v\|\tilde u\|^{p-1}_{_{L^{p-1}(\T)}}+\frac{p-1}{4\pi}\tilde u\int_\T v\bar{\tilde u}|\tilde u|^{p-3}+\mathcal{T}_N(v,\tilde u,\cdots, \tilde u).
\end{align*}
Similarly, 
\begin{align*}
|\tilde u|^{p-3}\tilde u^2\bar v&=\mathcal T(\tilde u,v,\tilde u,\cdots, \tilde u)=\frac{p+1}{4\pi}\tilde u\int_\T|\tilde u|^{p-3}\tilde u\bar v+\mathcal{T}_N (\tilde u,v,\tilde u,\cdots, \tilde u).
\end{align*}
Accordingly, it follows that (recall $(\cdot,\cdot)$ refers to the complex $L^2$-inner product)
\begin{equation}\label{potential-decomposition}
\frac{p+1}{2}|\tilde u|^{p-1}v+\frac{p-1}{2}|\tilde u|^{p-3}\tilde u^2\bar v=\frac{p+1}{4\pi}v\|\tilde u\|^{p-1}_{_{L^{p-1}(\T)}}+\frac{p^2-1}{4\pi}\tilde u\re \, (|\tilde u|^{p-3}\tilde u,v)+\mathcal {R}(\tilde u,v),
\end{equation}
where
\[\mathcal{R}(\tilde u,v)=\frac{p+1}{2}\mathcal T_N(v,\tilde u,\cdots, \tilde u)+\frac{p-1}{2}\mathcal T_N(\tilde u,v,\tilde u,\cdots, \tilde u).\]
A straightforward but crucial observation is that the regularity of $\tilde{u} \re\, (|\tilde{u}|^{p-3}\tilde{u},v)$ can be directly improved via the potential $\tilde u$, while $\mathcal R(\tilde u,v)$ enjoys the nonlinear smoothing effect. In addition, one can remove $\frac{p+1}{4\pi}v\|\tilde u\|^{p-1}_{_{L^{p-1}(\T)}}$ by means of a norm-preserving transformation.

For $\tilde u\in X_T^{s+\sigma,b}$ and $v_0\in H^s$, as mentioned in Section~\ref{Section-scheme}, we define $z(t)$ to be the solution of an auxiliary linear problem \eqref{Linear-problem-1}. Using the norm-preserving transform
\begin{equation}\label{Z-equation}
Z(t,x):=e^{i\tilde{\theta} (t)} z(t,x)\quad \text{with }\tilde\theta(t)=\frac{p+1}{4\pi}\int_0^t \|\tilde{u}(s)\|_{_{L^{p-1}(\mathbb{T})}}^{p-1} ds\in \mathbb{R},
\end{equation}
we have $iZ_t+Z_{xx}+ia(x)Z=0$. It follows immediately from Lemma~\ref{Lemma-Hslineardecay} that
\begin{equation}\label{decay-z}
\|z(T)\|_{_{\tilde{H}^s}}=\|Z(T)\|_{_{\tilde{H}^s}}=\|S_a(T)v_0\|_{_{\tilde{H}^s}}\le q_0 \|v_0\|_{_{\tilde{H}^s}}.
\end{equation}

Based on the above analysis, we derive the following result. 

\begin{proposition}[High-frequency dissipation]\label{Prop-HF}
Let $T,R,\sigma>0$ and $s\ge 1$ be arbitrarily given. Then there exists a constant $C>0$ such that 
\begin{equation}\label{HF}
\|Q_m (v(T)-z(T))\|_{_{H^s}}\le 
Cm^{-\sigma} \left(\|v_0\|_{_{H^s}}+\|f\|_{X_T^{s+\sigma,-b'}}\right)
\end{equation}
for any $m\in\N^+$, $v_0\in H^s$ and $f\in X_T^{s+\sigma,-b'}$,
where $v=\mathcal V_{\tilde u}(v_0,f)$ with $\tilde{u}\in \overline{B}_{X_T^{s+\sigma,b}}(R)$, and $z(t)$ stands for the solution of \eqref{Linear-problem-1}.
\end{proposition}

In particular, the conclusion of the above proposition with $f(t,x)\equiv  0$, gives rise to the HF dissipation \eqref{highfreqdecay} for uncontrolled system, as stated in Part II of Section \ref{Section-scheme}.

\begin{proof}[{\bf Proof of Proposition \ref{Prop-HF}}]
Consider the transformation
\begin{equation*}
V(t,x)=e^{i\tilde\theta(t)} v(t,x)\quad \text{with }\tilde\theta(t)=\frac{p+1}{4\pi}\int_0^t \|\tilde{u}(s)\|_{_{L^{p-1}(\mathbb{T})}}^{p-1} ds\in \mathbb{R}.
\end{equation*}
Then the new variable $V(t)$ is the solution of
\begin{equation}\label{Vequation}
\left\{\begin{array}{ll}
iV_t+V_{xx}+ia(x)V=\tfrac{p^2-1}{4\pi} \tilde{U} \re\, (|\tilde{U}|^{p-3}\tilde{U},V)+\mathcal{R}(\tilde{U},V)+F(t,x),\\
V(0,x)=v_0(x),
\end{array}
\right.
\end{equation}
where $\tilde{U}(t,x)=e^{i\tilde\theta(t)} \tilde u(t,x)$ and $F(t,x)=e^{i\tilde\theta (t)} f(t,x)$. Thus $W:=V-Z$ satisfies the equation
\[iW_t+W_{xx}+ia(x)W=\tfrac{p^2-1}{4\pi} \tilde{U} \re\, (|\tilde{U}|^{p-3}\tilde{U},V)+\mathcal{R}(\tilde{U},V)+F(t,x)\]
with $W(0)=0$, where $Z(t)$ is defined via \eqref{Z-equation}. Thanks to Lemma~\ref{Lemma-multiplication}, we have
\begin{align}
&\|\tilde{U}\|_{_{X_T^{s+\sigma,b}}}\le C\|e^{i\tilde\theta(t)}\|_{_{H^1_t}}\|\tilde u\|_{_{X^{s+\sigma,b}_T}}\le C,\notag\\
&\|F\|_{_{X^{s+\sigma,-b'}_T}}\le C\|e^{i\tilde\theta(t)}\|_{_{H^1_t}}\|f\|_{_{X^{s+\sigma,-b'}_T}}\le C\|f\|_{_{X^{s+\sigma,-b'}_T}}.\label{Force-bound}
\end{align}
We proceed to estimate the other two terms on the RHS of \eqref{Vequation}. By Proposition \ref{Proposition-linearproblem}(1),
\[\|V\|_{_{X^{s,b}_T}}\le C\|e^{i\tilde{\theta}(t)}\|_{_{H_t^1}}\|v\|_{_{X^{s,b}_T}}\le C\left(\|v_0\|_{_{H^s}}+\|f\|_{X_T^{s,-b'}}\right).\]
Note that $X_T^{s+\sigma,b}\hookrightarrow X_T^{s+\sigma,0} (=L_t^2 H_x^{s+\sigma})\hookrightarrow X_T^{s+\sigma,-b'}$ and $(\cdot,\cdot)$ is a spatial integral. Thus
\begin{equation}\label{hifreqdecay1}
\begin{aligned}
    &\|\tilde{U} \re\, (|\tilde{U}|^{p-3}\tilde{U},V)\|_{_{X_T^{s+\sigma,-b'}}}\le \|\tilde{U} \re\, (|\tilde{U}|^{p-3}\tilde{U},V)\|_{_{L_t^2 H_x^{s+\sigma}}}\\&\le C\|\tilde{U}\|_{_{L_t^2 H_x^{s+\sigma}}}\|\tilde{U}\|_{_{L^\infty(Q_T)}}^{p-2}\|V\|_{_{L^\infty(Q_T)}}\le C\|\tilde{U}\|_{_{X_T^{s+\sigma,b}}}^{p-1} \|V\|_{_{X^{s,b}_T}}\le C\left(\|v_0\|_{_{H^s}}+\|f\|_{X_T^{s,-b'}}\right).
\end{aligned}
\end{equation}
Thanks to Lemma~\ref{Lemma-smoothing}, we have
\begin{equation}\label{hifreqdecay2}
\|\mathcal{R}(\tilde{U},V)\|_{_{X^{s+\sigma,-b'}_T}}\le C\|\tilde{U}\|_{_{X^{s,b}_T}}^{p-1} \|V\|_{_{X^{s,b}_T}}\le C\left(\|v_0\|_{_{H^s}}+\|f\|_{X_T^{s,-b'}}\right).
\end{equation}

Now, putting (\ref{Force-bound})-(\ref{hifreqdecay2}) all together, invoking Proposition \ref{Proposition-multilinear}(4), we have
\begin{align*}
&\|Q_m (v(T)-z(T))\|_{_{H^s}}=\|Q_m W(T)\|_{_{H^s}}\\
&\le m^{-\sigma}\|\int_0^T S_a(T-t)(\tfrac{p^2-1}{4\pi}\tilde{U} \re\, (|\tilde{U}|^{p-3}\tilde{U},V)+\mathcal{R}(\tilde{U},V)+F) dt\|_{_{H^{s+\sigma}}}\\
&\le Cm^{-\sigma}\left(\|\tilde{U} \re\, (|\tilde{U}|^{p-3}\tilde{U},V)\|_{_{X_T^{s+\sigma,-b'}}}+\|\mathcal{R}(\tilde{U},V)\|_{_{X^{s+\sigma,-b'}_T}}+\|F\|_{_{X^{s+\sigma,-b'}_T}}\right) \\
&\le Cm^{-\sigma} \left(\|v_0\|_{_{H^s}}+\|f\|_{_{X_T^{s+\sigma,-b'}}}\right),
\end{align*}
in view of the fact $\|Q_m\|_{_{\mathcal L(H^{s+\sigma};H^s)}}\le m^{-\sigma}$. In conclusion, \eqref{HF} is obtained.
\end{proof}

\begin{remark}[Counterexample of $H^s$-potential]\label{counterexample}

As illustrated below the statement of Proposition~{\rm\ref{Prop-HF}}, this result implies the HF dissipation \eqref{highfreqdecay} when $f(t,x)\equiv 0$. Let us point out that, if we merely assume $\|\tilde{u}\|_{_{X^{s,b}_T}}\le R$, then one cannot find $m\in \mathbb{N}^+$ to justify \eqref{highfreqdecay}. Indeed, from \eqref{HF} it is easy to see that \eqref{highfreqdecay} is equivalent to the HF decay of the following linear equation
\[iv_t+v_{xx}+ia(x)v=\tfrac{p^2-1}{4\pi}\tilde u\re \, (|\tilde u|^{p-3}\tilde u,v),\]
as the remaining parts of potential terms \eqref{potential-decomposition} either have extra regularity and are hence negligible in HF, or can be eliminated by the norm-preserving transformation $v\mapsto V$. An intuition is that if the decay of frequencies for $\tilde u(t)$ were not to exceed the solution $v(t)$, there is a danger that the RHS destroys the dissipation effect brought by the damping $ia(x)v$. 

For the sake of simplicity and concreteness, let us focus on the case where $p=3, s=1$ and $a(x)\equiv a>0$ is a constant. Define the potential 
$$
\tilde u(t)=\sqrt{\frac{\pi a}{2}}\frac{(1+i)v(t)}{\|v(t)\|_{_{L^2}}}.
$$
Then a direct computation yields
\[\frac{2}{\pi} \tilde{u}\re\, (\tilde{u},v)=iav.\]
In other words, we have $iv_t+v_{xx}=0$. As an aftermath, the magnitude of each Fourier mode of $v(t)$ is conservative, and hence the HF dissipation \eqref{highfreqdecay} becomes hopeless. Since $\|v(t)\|_{_{L^2}}$ is constant, $\tilde{u}$ has the same regularity as $v$, i.e.~$\|\tilde{u}\|_{_{X_T^{1,b}}}=C\|v\|_{_{X_T^{1,b}}}\le C$.
\end{remark}

\subsection{Completing the proof of Proposition \ref{Proposition-control-1}}\label{Section-proof}

As stated in Part III of Section \ref{Section-scheme}, we in what follows invoke the conclusions of  Proposition~\ref{Proposition-LFC}, and construct a control $\xi$ having the extra regularity $ H^{s+\sigma'}$ and steering the LF system to $P_mz(T)$. The regularity gained for $\xi$ ensures that the presence of $\chi \mathcal{P}_N \xi$ does not ruin the HF dissipation as in Proposition~\ref{Prop-HF}.

\begin{lemma}\label{Lemma-regularcontrol}
Let $T,R,s,m,N$ be the same as in Proposition~{\rm\ref{Proposition-LFC}}, and $\sigma'=\sigma\in (0,1/4]$. Then there exists a constant $C>0$, independent of $m,N$, such that for every $\tilde u\in \overline{B}_{X_T^{s+\sigma,b}}(R)$ and $v_0\in H^s$, there is a control $\xi\in L^2(0,T;H^{s+\sigma})$ satisfying 
\begin{equation}\label{LFC-1}
P_mv(T)=P_mz(T)\quad \text{and}\quad \int_0^T\|\xi\|_{_{H^{s+\sigma}}}^2\leq C\|v_0\|^2_{_{H^s}},
\end{equation}
where $v=\mathcal V_{\tilde u}(v_0,\chi\mathcal P_N\xi)$ and $z(t)$ stands for the solution of {\rm(\ref{Linear-problem-1})}. Moreover, the control $\xi$ has the structure described in Proposition {\rm\ref{Proposition-control-1}(2)}.
\end{lemma}

\begin{proof}[{\bf Proof of Lemma \ref{Lemma-regularcontrol}}]
Let $\tilde u\in X_T^{s+\sigma,b}$ and $v_0\in H^s$ be arbitrarily given. It is easy to see that
\begin{equation}\label{bound-6}
\|z\|_{_{X^{s,b}_T}}\leq C\|v_0\|_{_{H^s}}. 
\end{equation}
We introduce the difference $w(t):=v(t)-z(t)$, and infer by (\ref{potential-decomposition}) that
\begin{equation}\label{Equation-w}
\left\{\begin{array}{ll}
iw_t+ w_{xx} +ia(x)w=\frac{p+1}{2}|\tilde u|^{p-1}w+\frac{p-1}{2}|\tilde u|^{p-3}\tilde u^2\bar w+\chi(x)\mathcal P_N\xi(t,x)\\
\quad\quad\quad\quad\quad\quad\quad\quad\quad\quad+\tfrac{p^2-1}{4\pi}\tilde{u}\re\, (|\tilde{u}|^{p-3}\tilde{u},z)+\mathcal R(\tilde u,z),\\
w(0,x)=0,
\end{array}\right.
\end{equation}
where $v=\mathcal V_{\tilde u}(v_0,\chi\mathcal P_N\xi)$ with an undetermined control $\xi\in L^2_tH^{s+\sigma}_x$. Using the same reasoning for \eqref{hifreqdecay1} and \eqref{hifreqdecay2}, we obtain
\begin{align}
&\|\tilde{u}\re\, (|\tilde{u}|^{p-3}\tilde{u},z)\|_{_{X_T^{s+\sigma,-b'}}}\le C\|\tilde u\|_{_{X_T^{s+\sigma,b}}}^{p-1}\|z\|_{_{X_T^{s,b}}}\le C\|v_0\|_{_{H^s}},\label{bound-4}\\
&\|\mathcal R(\tilde u,z)\|_{_{X_T^{s+\sigma,-b'}}}\le C\|\tilde u\|_{_{X_T^{s,b}}}^{p-1}\|z\|_{_{X_T^{s,b}}}\le C\|v_0\|_{_{H^s}}.\label{bound-5}
\end{align}
The combination of (\ref{bound-4}), (\ref{bound-5}) and Proposition \ref{Proposition-linearproblem}(1) enables us to observe that the solution $w_1(t)$ of the backward equation
\begin{equation}\label{Equation-w1}
\left\{\begin{array}{ll}
iw_{1t}+w_{1xx} +ia(x)w_1=\frac{p+1}{2}|\tilde u|^{p-1}w_1+\frac{p-1}{2}|\tilde u|^{p-3}\tilde u^2\bar w_1\\
\quad\quad\quad\quad\quad\quad\quad\quad\quad\quad\quad+\tfrac{p^2-1}{4\pi}\tilde{u}\re\, (|\tilde{u}|^{p-3}\tilde{u},z)+\mathcal R(\tilde u,z),\\
w_1(T,x)=0
\end{array}\right.
\end{equation}
belongs to $X_T^{s+\sigma,b}$, and can be estimated via
\begin{equation}\label{bound-7}
\|w_1\|_{_{X_T^{s+\sigma,b}}}\le C\left(\|\tilde{u}\re\, (|\tilde{u}|^{p-3}\tilde{u},z)\|_{_{X_T^{s+\sigma,-b'}}}+\|\mathcal R(\tilde u,z)\|_{_{X_T^{s+\sigma,-b'}}}\right)\le C \|v_0\|_{_{H^s}}.
\end{equation}

Let $w_2(t)=w(t)-w_1(t)$.
In view of (\ref{Equation-w}) and (\ref{Equation-w1}), we deduce 
$w_2=\mathcal V_{\tilde u}(-w_1(0),\chi\mathcal P_N\xi)$.
This together with (\ref{bound-7}) and Proposition \ref{Proposition-LFC}(2)(3) implies that there exists a control $\xi\in L^2_tH^{s+\sigma}_x$, having the form
\begin{equation}\label{structure-control}
\xi=i(1-\partial_{x}^2)^{-s-\sigma}[\mathcal P_N(\chi\varphi)]\quad {\rm with\ }\varphi=\mathcal U_{\tilde u}(\Lambda(\tilde u)(-w_1(0))),
\end{equation}
such that
\begin{equation}\label{Null-controllability-1}
P_mw_2(T)=0\quad \text{and}\quad  \int_0^T \|\xi\|_{_{H^{s+\sigma}}}^2\le C\|w_1(0)\|^2_{_{H^{s+\sigma}}}\le C\|v_0\|_{_{H^s}}^2.
\end{equation} 
The identity in (\ref{Null-controllability-1}), combined with (\ref{Equation-w1}), gives rise to
\begin{equation}\label{Null-controllability-2}
P_mw(T)=0.
\end{equation}

In conclusion, the properties in  (\ref{LFC-1}) 
follow from (\ref{Null-controllability-1}) and (\ref{Null-controllability-2}), while (\ref{structure-control}) together with Proposition \ref{Proposition-linearproblem}(2) implies
the desire structure of the control. The proof is then complete.
\end{proof}

We are now in a position to conclude this section.

\begin{proof}[\bf Proof of Proposition~\ref{Proposition-control-1}]
We first point out that it suffices to consider $\sigma\in (0,1/4]$. In fact, we can always choose $0<\sigma''\le \max\{\sigma,1/4\}$, and the assumptions of Proposition~\ref{Proposition-control-1} are satisfied with $\sigma$ replaced by $\sigma''$, due to the continuous embedding $H^{s+\sigma}\hookrightarrow H^{s+\sigma''}$.

Fix a constant $q_1\in (q_0,1)$ with $q_0\in (0,1)$ specified by Lemma \ref{Lemma-Hslineardecay}. Let us continue to use the setting in the proof of Proposition~\ref{Prop-HF} and Lemma~\ref{Lemma-regularcontrol}, with the control $\xi\in L_t^2 H_x^{s+\sigma}$ satisfying \eqref{LFC-1} specified. Represent $w$ as
$$
w=\mathcal V_{\tilde u}(0,f),
$$
where $f=f(t,x)$ represents the sum of the last three terms in the RHS of \eqref{Equation-w}. From \eqref{LFC-1} and (\ref{bound-6}) it follows that 
\begin{equation*}
\begin{aligned}
\|f(t,x)\|_{_{X^{s+\sigma,-b'}_T}}&\leq C \left(\|\chi \mathcal{P}_N \xi\|_{L_t^2 H_x^{s+\sigma}}+\|\tilde{u}\re\, (|\tilde{u}|^{p-1}\tilde{u},z)\|_{_{X^{s+\sigma,-b'}_T}}+\|\mathcal R(\tilde u,z)\|_{_{X^{s+\sigma,-b'}_T}}\right)\\
&\leq C\left(\|\xi\|_{L_t^2 H_x^{s+\sigma}}+\|z\|_{_{X^{s,b}_T}}\right)\le 
C\|v_0\|_{_{H^s}}.
\end{aligned}
\end{equation*}
Thanks to Proposition~\ref{Prop-HF}, we have
\[\|Q_m w(T)\|_{_{H^s}}\le Cm^{-\sigma}\|f\|_{_{X^{s+\sigma,-b'}_T}}\leq Cm^{-\sigma}\|v_0\|_{_{H^s}}\]
for any $m\in \mathbb{N}^+$.
By \eqref{LFC-1} and $w=v-z$, we have $P_m w(T)=0$. Moreover, recall that the $\tilde{H}^s$-norm is equivalent to the $H^s$-norm. Thus we can choose $m$ sufficiently large, so that
\[\|w(T)\|_{_{\tilde{H}^s}}=\|Q_m w(T)\|_{_{\tilde{H}^s}}\le (q_1-q_0)\|v_0\|_{_{\tilde{H}^s}}.\]

Finally, 
taking \eqref{decay-z} into account, we conclude that
\[\|v(T)\|_{_{\tilde{H}^s}}\le \|w(T)\|_{_{\tilde{H}^s}}+\|z(T)\|_{_{\tilde{H}^s}}\le q_1\|v_0\|_{_{\tilde{H}^s}}.\]
Now the proof is complete, as the control $\xi$ is the same as in Lemma~\ref{Lemma-regularcontrol}, which has the structure described in Proposition {\rm\ref{Proposition-control-1}(2)}.
\end{proof}

\section{Exponential mixing}\label{Section-proofmain}

With the preparations from previous sections, we are now able to prove the \hyperlink{thm1}{\color{black}Main Theorem} for randomly forced NLS equation (\ref{Random-problem}). 

In Section \ref{Section-framework}, we recall general criterion for exponential mixing, which is established in the previous work \cite{LWX-24}. 
This criterion consists of three hypotheses: exponential asymptotic compactness ($\mathbf{EAC}$), irreducibility ($\mathbf{I}$) and coupling condition ($\mathbf{C}$).

The next thing to be done is to verify the hypotheses just mentioned. In particular, the verification of ($\mathbf{C}$) contributes to the main content. It involves the stabilization along trajectory for  associated controlled system, mingled with several results on optimal couplings and probability measures. In Section \ref{Section-coupling} we extract an abstract criterion of ($\mathbf{C}$) from technical reasoning which varies from example to example. In this criterion, the conditions are directly related to  stabilization along trajectory (Section~\ref{Section-control}), while the materials from other fields are included in the proof. 

Finally, the verification of all hypotheses will be finished in Section \ref{Section-verification}. The proof of ($\mathbf{EAC}$) and ($\mathbf{I}$) therein are easy applications of the exponential asymptotic compactness and global stabilization for the deterministic problem (\ref{Nonlinear-problem}) established in Sections \ref{Section-AC} and \ref{Section-stability}.

\subsection{Probabilistic framework based on  asymptotic compactness}\label{Section-framework}

In this subsection,
we are positioned in a setting of random dynamical system.
Let $(\mathcal X,\|\cdot\|)$ and $(\mathcal{Z},\|\cdot\|_{\mathcal{Z}})$ be separable real\footnote{When applying the general criterion to NLS, the complex function spaces (such as $H^s$) will be deliberately regarded as real Hilbert spaces. We will clarify this issue at that time.} Banach spaces. Assume that $$S\colon\mathcal{X} \times\mathcal{Z} \rightarrow \mathcal{X}$$ is a locally Lipschitz map, and $\{\xi_n;n\in\N\}$ stands for a sequence of $\mathcal{Z}$-valued i.i.d.~random variables. The common law of $\xi_n$ is $\ell$, whose support is denoted by $\mathcal{E}$. The Markov process considered here is given by 
\begin{equation}\label{RDS}
x_{n+1}=S(x_{n},\xi_{n}),\quad x_0=x\in\mathcal{X}.
\end{equation}
In order to indicate the initial condition and the random inputs, we also write 
\begin{equation*}
x_n=S_{n}(x;\xi_0,\cdots,\xi_{n-1})=S_{n}(x;\boldsymbol{\xi}),\quad n\in\N^+
\end{equation*}
with $\boldsymbol{\xi}:=(\xi_n;n\in\N)$.
Moreover, given a sequence $\boldsymbol{\zeta}=(\zeta_n;n\in\N)\in \mathcal{Z}^{\N}$, we denote by $$S_{n}(x ; \zeta_0, \cdots, \zeta_{n-1})=S_{n}(x;\boldsymbol{\zeta})$$ the corresponding deterministic process defined by (\ref{RDS}) replacing $\xi_n$ with $\zeta_n$.

With the above setting, system (\ref{RDS}) defines a Feller family of discrete-time Markov processes in $\mathcal{X}$; see, e.g. \cite[Section 1.3]{KS-12}. We denote the corresponding Markov family by $\Pb_x\ (x\in\mathcal{X})$, the expected values by $\E_x$, and the Markov transition functions by $P_n(x,\cdot)$, i.e. 
\begin{equation}\label{Markov-transition}
P_n(x,A)=\Pb_x(x_n\in A),\quad A\in\mathcal{B}(\mathcal{X}),\ n\in \N.
\end{equation}
We define the Markov semigroup $P_n\colon C_b(\mathcal{X})\rightarrow C_b(\mathcal{X})$ and its dual $P^*_n\colon\mathcal{P}(\mathcal{X})\rightarrow \mathcal{P}(\mathcal{X})$ by setting
\begin{equation}\label{Markov-semigroup}
P_nf(x)= \int_{\mathcal{X}}f(y)P_n(x,dy)\quad \text{and}\quad P_n^*\mu(A)=\int_{\mathcal{X}}P_n(x,A)\mu(dx),
\end{equation}
for $f\in C_b(\mathcal{X})$, $\mu\in\mathcal{P}(\mathcal{X})$, $x\in\mathcal{X}$ and $A\in\mathcal{B}(\mathcal{X})$. Recall that a probability measure $\mu\in \mathcal{P}(\mathcal{X})$ is called \textit{invariant} for $P_n^*$  if 
$$
P_n^*\mu=\mu,\ \forall n\in\N.
$$ 
Meanwhile, we say a subset $\mathcal Y\subset\mathcal X$ is invariant, if 
$$
S(\mathcal Y\times\mathcal E) \subset \mathcal Y.
$$

Recall a coupling between $\mu,\nu\in \mathcal{P}(\mathcal{X})$ refers to a pair of $\mathcal{X}$-valued random variables with marginal distributions equal to $\mu$ and $\nu$, respectively. The set of all couplings between $\mu$ and $\nu$ is denoted by $\mathscr{C}(\mu,\nu)$. 

Below is a list of hypotheses regarding the abstract criterion for exponential mixing. 

\begin{itemize}
\item [($\mathbf{EAC}$)]   ({\bf Exponential asymptotic compactness}) There exists a compact invariant subset $\mathcal{Y}$ of $\mathcal{X}$, a constant $\kappa>0$, and a measurable function $V\colon\mathcal{X}\rightarrow\R^+$, which sends bounded sets into bounded sets, such that
\begin{equation*}
\text{dist}_{\mathcal X}(S_n(x;\boldsymbol{\zeta}),\mathcal{Y})\leq V(x)e^{-\kappa n}
\end{equation*}
for any $x\in\mathcal{X},\boldsymbol{\zeta}\in \mathcal{E}^{\N}$ and $n\in\N$.

\item [($\mathbf{I}$)]   ({\bf  Irreducibility on $\mathcal{Y}$}) There exists a point $z\in\mathcal{Y}$ such that for every $\varepsilon>0$, there is an integer $m\in\N^+$ satisfying
\begin{equation*}
\inf_{y\in\mathcal{Y}}P_m(y,B_{\mathcal{X}}(z,\varepsilon))>0.
\end{equation*}
		
\item [($\mathbf{C}$)]   ({\bf Coupling condition on $\mathcal Y$}) There exist constants $q\in[0,1)$ and $C>0$, such that for every $y_1,y_2\in\mathcal{Y}$, one can find $(\mathcal{R}(y_1,y_2),\mathcal{R}'(y_1,y_2))\in\mathscr{C}(P_1(y_1,\cdot),P_1(y_2,\cdot))$ on a common probability space $(\Omega,\mathcal{F},\Pb)$, satisfying
\begin{equation*}
\Pb(\|\mathcal{R}(y_1,y_2)-\mathcal{R}'(y_1,y_2)\|>q\|y_1-y_2\|)\leq C\|y_1-y_2\|.
\end{equation*}
Moreover, the maps $\mathcal{R},\mathcal{R}'\colon\mathcal{Y}\times \mathcal{Y}\times \Omega\rightarrow \mathcal{X}$ are measurable.
\end{itemize}

We shall use the following notion of attainable set.
	
\begin{definition}
For $D\subset \mathcal{X}$, the attainable set $\mathcal{A}_n(D)$ at time $n$ is recursively defined by
\begin{equation*}
\mathcal{A}_n(D):=\{S(x,\zeta);x\in \mathcal{A}_{n-1}(D),\zeta\in \mathcal{E}\},\quad n\in\N^+
\end{equation*}
with $\mathcal{A}_0(D)=D$,
and the attainable set $\mathcal{A}(D)$ is given by
$
\mathcal{A}(D)=\overline{\bigcup_{n\in\N}\mathcal{A}_n(D)}.
$  
\end{definition} 

With the above hypotheses, the $(\mathbf{EAC})$-based criterion of exponential mixing is collected in the following proposition; see \cite[Theorem 2.1]{LWX-24} for its proof.

\begin{proposition}\label{Proposition-EX}
Assume that the support $\mathcal{E}$ of $\ell$ is compact in $\mathcal{Z}$, and hypotheses $(\mathbf{EAC})$, $(\mathbf{I})$ and $(\mathbf{C})$ are satisfied. Then the Markov process $\{x_n;n\in \mathbb{N}\}$ has a unique invariant measure $\mu \in \mathcal{P}(\mathcal{X})$ with compact support $\supp(\mu)=\mathcal{A}(\{z\})\subset \mathcal{Y}$. Moreover, there exist constants $C,\beta >0$ such that
\[\|P_n^* \nu-\mu\|_L^* \le Ce^{-\beta n}\left(1+\int_{\mathcal{X}} V(x)\nu(dx)\right)\]
for any $\nu\in \mathcal{P}(\mathcal{X})$ such that $\int_{\mathcal{X}} V(x)\nu(dx) < \infty$ and $n\in \mathbb{N}$.
\end{proposition}

\subsection{Abstract criterion for coupling condition}\label{Section-coupling}

The demonstration of coupling condition is less direct, involving optimal coupling and control theory. This has been done for, e.g.~Navier--Stokes equations \cite{Shi-15,Shi-21} and wave equations \cite{LWX-24}. We state an abstract result which relates $(\mathbf{C})$ to control property, which is general and does not depend on specific PDE model.

We continue with the setting in Section~\ref{Section-framework}. 
In the sequel, let us introduce the following conditions with an arbitrarily given $q\in(0,1)$.

\begin{itemize}
\item [($\mathbf{C1}$)] There exist subspaces $\mathcal{Z}_1,\mathcal{Z}_2$ of $\mathcal Z$ such that $\mathcal{Z}_1$ is finite dimensional,  $\mathcal{Z}=\mathcal{Z}_1\oplus \mathcal{Z}_2$ and 
$$
\ell=\ell_1\otimes \ell_2,
$$
where $\ell_i=({\rm Proj}_{\mathcal{Z}_i})_* \ell$, the map ${\rm Proj}_{\mathcal{Z}_i}$ denotes the projection onto  $\mathcal{Z}_i$, and $_*$ refers to the pushforward of probability measures. Moreover, the probability measure $\ell_1$ has a $C^1$ density function with respect to Lebesgue measure on $\mathcal{Z}_1$.

\item [($\mathbf{C2}$)] There exists a compact invariant subset $\mathcal{Y}$ of $\mathcal{X}$,
constants $d>0$ and $q'\in (0,q)$ such that if $\vec{y}=(y_1,y_2)\in \mathcal{Y}\times \mathcal{Y}$ with $\|y_1-y_2\|\le d$, then there is a continuously differentiable map
$\Phi^{\vec{y}}\colon \mathcal{Z}\to \mathcal{Z}_1$, such that for $\ell$-a.e. $\zeta\in \mathcal{Z}$,
\begin{equation*}
\|S(y_1,\zeta)-S(y_2,\zeta+\Phi^{\vec{y}}(\zeta))\|\le q' \|y_1-y_2\|.
\end{equation*}

\item [($\mathbf{C3}$)] There exists a constant $C>0$, independent of $\vec{y}$, such that for any $\zeta\in \mathcal{Z}$,
\begin{equation*}
\|\Phi^{\vec{y}}(\zeta)\|_{_{\mathcal{Z}}} \le C\|y_1-y_2\|  \quad\text{and} \quad \Lip(\Phi^{\vec{y}})\le C\|y_1-y_2\|.
\end{equation*}
\end{itemize}

Condition $(\mathbf{C1})$ provides a functional setting, while $(\mathbf{C2}),(\mathbf{C3})$ are associated with the stabilization along trajectory for controlled system. The control involved has certain structure as in Theorem \ref{Theorem-control}(2).

In our application to random NLS (\ref{Random-problem}), the compact set $\mathcal{Y}$ in $(\mathbf{C2})$ is exactly the attractor involved in $(\mathbf{EAC})$. As there is no danger of ambiguity, we slightly abuse the notations.

\medskip

Generally speaking, the following lemma indicates that the coupling condition in the probabilistic setting can be translated to the issue of control property.

\begin{lemma}\label{couplingcriterion}
Let $0<q'<q<1$ be arbitrarily given, and assume the conditions $(\mathbf{C1})$-$(\mathbf{C3})$ are satisfied. Then the hypothesis $(\mathbf{C})$ holds. More precisely, there exists a constant $C>0$, a probability space $(\Omega,\mathcal{F},\mathbb{P})$ and measurable maps $\mathcal{R},\mathcal{R}'\colon \mathcal{Y}\times \mathcal{Y}\times \Omega\to \mathcal{X}$, such that for any $y_1,y_2\in\mathcal{Y}$, the pair $(\mathcal{R}(y_1,y_2),\mathcal{R}'(y_1,y_2))\in \mathscr{C}(P_1(y_1,\cdot),P_1(y_2,\cdot))$, and
\begin{equation}\label{couplingcondition1}
\mathbb{P}(\|\mathcal{R}(y_1,y_2)-\mathcal{R}'(y_1,y_2)\|>q\|y_1-y_2\|)\le C\|y_1-y_2\|.
\end{equation}
\end{lemma}

\begin{proof}[{\bf Proof of Lemma \ref{couplingcriterion}}]
Let $\mathcal Y\subset\mathcal X$ and $d>0$ be established in condition $(\mathbf{C2})$.
It suffices to define the desired coupling $\mathcal{R},\mathcal{R}'$ for $\vec{y}$ belonging to the subset $\boldsymbol{A}$ of $\mathcal{Y}\times \mathcal{Y}$, where
\[\boldsymbol{A}:=\{\vec{y}\in \mathcal{Y}\times \mathcal{Y}; \|y_1-y_2\|\le d\|\}.\]
In fact, for $\vec{y}\in \mathcal{Y}\times \mathcal{Y}\setminus \boldsymbol{A}$, we can choose $\xi'$ to be an independent copy of $\xi$, and set
\[\mathcal{R}=S(y_1,\xi),\quad \mathcal{R}'=S(y_2,\xi').\]
Then \eqref{couplingcondition1} holds, up to replacing $C$ by $\max\{C,d^{-1}\}$, since the LHS does not exceed $1$, and $\|y_1-y_2\|>d$ for $\vec{y}\not \in \boldsymbol{A}$. The rest of the proof is devoted to the construction of $\mathcal{R},\mathcal{R}'$ on $\boldsymbol{A}$

We will invoke a measurability result on optimal couplings \cite[Proposition A.1]{LWX-24}. For $\boldsymbol{\varepsilon}=(\varepsilon_1,\varepsilon_2)$ with $0\leq \varepsilon_2\leq \varepsilon_1<\infty$,  define a functional $\rho_{\boldsymbol{\varepsilon}}\colon \mathcal{X}\times \mathcal{X}\rightarrow [0,1]$ by 
    \[\rho_{\boldsymbol{\varepsilon}}(x_1,x_2)=\varphi_{\boldsymbol{\varepsilon}}(\|x_1-x_2\|),\]
    where $\varphi_{\boldsymbol{\varepsilon}}\colon\R^+\rightarrow[0,1]$ is given by
    \begin{equation*}
        \varphi_{\boldsymbol{\varepsilon}}(s)=\begin{cases}
        1 &\text{ for } s>\varepsilon_1,\\
        \frac{s-\varepsilon_2}{\varepsilon_1-\varepsilon_2} &\text{ for } \varepsilon_2<s\leq \varepsilon_1,\\
        0 &\text{ for } 0\leq s\leq \varepsilon_2.
        \end{cases}
    \end{equation*}
    For $\mu,\nu \in \mathcal{P}(\mathcal{X})$, let us also set
    \begin{equation*}
        \|\mu-\nu\|_{\boldsymbol{\varepsilon}}=\inf_{(\xi,\eta)\in\mathscr{C}(\mu,\nu)} \E \rho_{\boldsymbol{\varepsilon}}(\xi,\eta).
    \end{equation*}
    
For $\vec{y}=(y_1,y_2)\in \boldsymbol{A}$, we  define a non-negative function $\lambda(\vec{y})$ as
\begin{equation*}
\lambda(\vec{y})=\|y_1-y_2\|.
\end{equation*}
Owing to \cite[Proposition A.1]{LWX-24}, there exists a probability space $(\Omega,\mathcal{F},\mathbb{P})$ and measurable maps $\mathcal{R},\mathcal{R}'\colon \boldsymbol{A}\times\Omega\rightarrow \mathcal{X}$
    such that 
    $(\mathcal{R}(\vec{y}),\mathcal{R}'(\vec{y}))\in\mathscr{C}(P_1(y_1,\cdot),P_1(y_2,\cdot))$ and 
    \begin{equation*}
        \E \rho_{(q\lambda(\vec{y}),q\lambda(\vec{y}))}(\mathcal{R}(\vec{y}),\mathcal{R}'(\vec{y}))\leq \|P_1(y_1,\cdot)-P_1(y_2,\cdot)\|_{(q\lambda(\vec{y}),q'\lambda(\vec{y}))}.
    \end{equation*}
    Accordingly, due to the definition of $\rho_{\boldsymbol{\varepsilon}}$ and $\lambda$,
\begin{equation}\label{Control-squeezing-1}
\mathbb{P}(\|\mathcal{R}(\vec{y})-\mathcal{R}'(\vec{y})\|>q\|y_1-y_2\|)\leq \|P_1(y_1,\cdot)-P_1(y_2,\cdot)\|_{(q\lambda(\vec{y}),q'\lambda(\vec{y}))}.
\end{equation} 

Thanks to \cite[Proposition 5.2]{Shi-15} (see also \cite[Lemma A.2]{LWX-24}), condition ($\mathbf{C2}$) implies that
\begin{equation}\label{Control-squeezing-3}
\|P_1(y_1,\cdot)-P_1(y_2,\cdot)\|_{(q\lambda(\vec{y}),q'\lambda(\vec{y}))}\leq 2\|\ell-(I+\Phi^{\vec{y}})_*\ell\|_{\rm TV},
\end{equation}
where $I$ is the identity map, and $\|\ell_1-\ell_2\|_{\rm TV}$ denotes the total variation distance between two probability measures $\ell_1$ and $\ell_2$. 
To estimate the RHS, we apply \cite[Proposition 5.6]{Shi-15} (see also \cite[Lemma A.1]{LWX-24}). The assumptions involved are justified by $(\mathbf{C1})$ and $(\mathbf{C3})$. As a result,
\begin{equation}\label{Control-squeezing-4}
\|\ell-(I+\Phi^{\vec{y}})_*\ell\|_{\rm TV}\leq C\|y_1-y_2\|.
\end{equation}
Putting (\ref{Control-squeezing-1})-(\ref{Control-squeezing-4}) all together, the proof is now complete.
\end{proof}

\subsection{Proof of Main Theorem}\label{Section-verification} 

Let us verify hypotheses $(\mathbf{EAC})$, $(\mathbf{I})$ and $(\mathbf{C})$ for random NLS (\ref{Random-problem}) at the scale of $H^s$ with $s\ge 1$, which would conclude the proof of the \hyperlink{thm1}{\color{black}Main Theorem}.

Under the settings $(\mathbf{S1}),(\mathbf{S2})$ (see Section~\ref{Section-setup}), we introduce the time-$T$ solution map 
$$
S\colon H^s\times L^2_t H^s_x\rightarrow H^s,\quad S(u_0,f)=u(T),
$$
where $u(t)$ stands for the solution of (\ref{Nonlinear-problem}). For the sake of clarity, we point out the complex-valued function spaces $H^s$ and $L^2_t H^s_x$ can be viewed as real Hilbert spaces (cf.~the setting of Section~\ref{Section-framework}). Indeed, if $X$ is a complex Hilbert space with inner product $(\cdot,\cdot)_X$, then $\re\,  (\cdot,\cdot)_X$ is a real-inner product on $X$, turning $X$ into a real Hilbert space.

It is easy to deduce that the map $S$ is locally Lipschitz and continuously differentiable (see Proposition \ref{Proposition-nonlinear}). We set $\mathcal X=\tilde{H}^s$, which is equivalent to $H^s$ and defined in Lemma~\ref{Lemma-Hslineardecay}, and
\[\mathcal Z=\overline{{\rm span}\{\chi \alpha_j^T(t) e_k(x);j\in \mathbb{N}^+ ,k\in \mathbb{Z}\}}^{L_t^2 H_x^{s+\sigma}}\subset L_t^2 H_x^{s+\sigma}.\]
Note that the closure is taken within $L_t^2 H_x^{s+\sigma}$, with extra spatial regularity. Then equation (\ref{Random-problem}) induces a Markov process $(u_n,\mathbb{P}_u)$ given by 
\begin{equation*}
u_{n+1}=S(u_n,\eta_n),\quad 
u_0\in H^s. 
\end{equation*}
The corresponding Markov objects $P_n(u,\cdot),P_n,P_n^*$ are defined as in (\ref{Markov-transition}) and (\ref{Markov-semigroup}). 

Taking (\ref{bounded-noise-0}) into account, 
we observe that the sample paths of $\eta_n$ are contained in a bounded subset of $L^2(0,T;H^{s+\sigma})$. That is, there exists a constant $R_0>0$, depending on $B_0$, such that
\[\eta_n\in{B}_{L^2(0,T;H^{s+\sigma})}(R_0)\quad\text{almost surely}.\]
Moreover, using a diagonal argument, it is easy to see from $(\mathbf{S2})$ that $\mathcal{E}$ is compact in $\mathcal{Z}$.

\medskip
\noindent \textbf{Verification of hypothesis $(\mathbf{EAC})$.} 
For every $\boldsymbol{\zeta}=(\zeta_n;n\in\N
)\in \mathcal{E}^{\N}$, the corresponding concatenation $f\colon\R^+\rightarrow H^{s+\sigma}$, i.e. 
$$
f(t,x)=\zeta_n(t-nT,x),\quad t\in[nT,(n+1)T),\,n\in\N,
$$ 
belongs to ${B}_{L^2_b(\R^+;H^{s+\sigma})}(\lceil 1/T\rceil R_0)$. Then, Theorem~\ref{Theorem-AC} and Theorem~\ref{Theorem-AC-Hs} yield that
\begin{equation*}
\text{dist}_{H^s} (S_n(u_0;\boldsymbol{\zeta}),\mathscr B_{s,s+\sigma})\leq V_s(u_0) e^{-\kappa Tn},
\end{equation*}
for any $u_0\in H^s$, $\boldsymbol{\zeta}\in \mathcal{E}^{\N}$ and $n\in \N$, where $\mathscr{B}_{s,s+\sigma}$ is a bounded subset of $H^{s+\sigma}$, and
\[V_s(u_0)=\begin{cases}
C(1+E(u_0))&\text{\rm for } s=1,\\
C(1+\|u_0\|_{_{H^s}})^{C(1+E(u_0))^{(p-1)\lceil 4s-3\rceil/2}}&\text{\rm for } s>1.
\end{cases}\]
Due to the uniform boundedness (\ref{H^s-boundedness}), we find that $\mathcal{Y}:=\mathcal{A}(\mathscr{B}_{s,s+\sigma})$ is a bounded subset of $H^{s+\sigma}$, which is in turn compact in $\mathcal{X}$. As $\mathcal{Y}$ is clearly invariant due to the definition of attainable set, this completes the verification of  $(\mathbf{EAC})$.

\medskip
\noindent\textbf{Verification of hypothesis $(\mathbf{I})$.} 
From Proposition~\ref{Corollary-stability} and Proposition~\ref{Prop-Hsabsorbing} it follows that that for every $\varepsilon>0$, there exists an integer $m\in\N^+$ such that
\begin{equation*}
\|S_m(u_0;\boldsymbol{0})\|_{_{H^s}}<\frac{\varepsilon}{2}
\end{equation*} 
for any $u_0\in\mathcal Y$, where $\boldsymbol{0}$ stands for a sequence of zeros.
Notice that the map 
$$
\mathcal Y\times\mathcal E^m\ni(u_0,\boldsymbol{\zeta})\mapsto S_m(u_0;\boldsymbol{\zeta})\in H^s
$$
is uniformly continuous, since $\mathcal{E}$ is compact. Specifically, there exists a constant $\delta>0$ such that 
\begin{equation*}
\|S_m(u_0;\boldsymbol{\zeta})\|_{_{H^s}}<\varepsilon
\end{equation*}
for any $u_0\in\mathcal Y$ and $\boldsymbol{\zeta}=(\zeta_n)$ with $\zeta_0,\cdots, \zeta_{m-1}\in \mathcal E\cap B_{L^2_tH_x^s}(\delta)$. We then conclude that 
\begin{equation*}
\begin{aligned}
P_m(u_0,B_{H^s}(\varepsilon)) \geq \Pb (\|\eta_n\|_{L^2_tH^s_x}<\delta,\,\forall\,0\leq n\leq m-1) = \ell (B_{L^2_tH^s_x}(\delta))^{m}>0;
\end{aligned}
\end{equation*}
the last step is due to $0\in \mathcal E$, which is assured by $\rho_{jk}(0)>0$. The hypothesis $(\mathbf{I})$ is then verified. 

\medskip
\noindent\textbf{Verification of hypothesis $(\mathbf{C})$.} 
In view of Lemma~\ref{couplingcriterion}, it suffices to justify conditions $(\mathbf{C1})$-$(\mathbf{C3})$, where the compact invariant set $\mathcal Y$ is taken as that in $(\mathbf{EAC})$.

Condition $(\mathbf{C1})$ follows from the noise structure \eqref{bounded-noise-0} and (\ref{non-degenerate}), by setting
\[\mathcal{Z}_1={\rm span}\{\chi \alpha_j^T(t)e_k(x);j,|k|\le N\} \quad \text{and}\quad \mathcal{Z}_2=\overline{{\rm span}\{\chi \alpha_j^T(t)e_k(x);j,|k|>N\}}^{L_t^2 H_x^{s+\sigma}}.\]
Next we verify $(\mathbf{C2})$ and $(\mathbf{C3})$. We can choose $R_1>0$ sufficiently large so that 
\begin{equation*}
\mathcal Y\subset \overline{B}_{H^{s+\sigma}}(R_1)\quad \text{and}\quad \mathcal E\subset\overline{B}_{L^2_tH^{s+\sigma}_x}(R_1).
\end{equation*}
Moreover, there exists a constant $R>0$ such that 
$
u\in \overline{B}_{X_T^{s+\sigma,b}}(R)
$
for any solution $u(t)$ of (\ref{Nonlinear-problem}) with $u_0\in \overline{B}_{H^{s+\sigma}}(R_1)$ and $f\in \overline{B}_{L^2_tH^{s+\sigma}_x}(R_1+2)$. We then apply Theorem \ref{Theorem-control} with such $R$, and deduce that there exist constants $d>0,N\in\N^+,q'\in (0,1)$ and a map 
$$
\Phi'\colon \overline{B}_{H^{s+\sigma}}(R_1)\times \overline{B}_{L^2_tH^{s+\sigma}_x}(R_1+2)\rightarrow \mathcal{L}_{\R}(H^s;L_t^2H^s_x)
$$ 
such that 
\begin{equation}\label{Control-squeezing-2}
\|S(u_0,\zeta)-S(v_0,\zeta+\chi\mathcal P_N\Phi'(u_0,\zeta)(v_0-u_0))\|_{_{\tilde{H}^s}}\leq q'\|u_0-v_0\|_{_{\tilde{H}^s}}
\end{equation}
for any $u_0,v_0\in\overline{B}_{H^{s+\sigma}}(R_1)$ with $\|u-v\|_{_{H^s}}\leq d$ and $\zeta\in \overline{B}_{L^2_tH^{s+\sigma}_x}(R_1+1)$. Moreover, the map $\Phi'$ is Lipschitz and continuously differentiable, as it is the composition of two maps of such type. In the sequel, we assume (\ref{non-degenerate}) with $N$ just established. 

Finally, let us define
\[\Phi^{(u_0,v_0)}(\zeta)=\phi(\|\zeta\|_{_{L_t^2H_x^{s+\sigma}}}^2)\chi\mathcal{P}_N \Phi'(u_0,\zeta)(v_0-u_0),\]
where $\phi\colon \mathbb{R}\to \mathbb{R}^+$ is a smooth cutoff function, such that $\phi(r)=1$ for $0\le r\le R_1^2$ and $\phi(r)=0$ for $r\ge (R_1+1)^2$. Then $(\mathbf{C2})$ and $(\mathbf{C3})$ are satisfied, owing to \eqref{Control-squeezing-2}. Consequently, hypothesis $(\mathbf{C})$ is valid for any $q\in (q',1)$.

\medskip

The conclusions of \hyperlink{thm1}{\color{black}Main Theorem} are now obtained, as we have accomplished the verification of all hypotheses involved in Proposition~\ref{Proposition-EX}.

\appendix

\section{Bourgain spaces and global well-posedness}\label{Appendix-GWP}

We collect the global well-posedness result and some basic estimates of the Schr\"{o}dinger equations under consideration. These results can be derived within the framework of restricted norm method due to Bourgain \cite{Bourgain-book}.

\subsection{Basic estimates in Bourgain spaces}\label{Appendix-Bourgainspace}

The Bourgain space is a powerful tool in the study of dispersive PDEs, initially applied to low-regularity well-posedness. We recall the basic properties and multilinear estimates needed in this paper. Most of the proofs can be found in existing literature, while some require adaptions due to the nonlinearity of general order $p\ge 3$.

\begin{definition}
    Let $s,b\in \mathbb{R}$ be arbitrarily given. The Bourgain space $X^{s,b}$ consists of tempered distributions $u$ on $\R\times\T$ for which the norm defined by
    \[\|u\|^2_{_{X^{s,b}}}:=\sum_{k\in\Z}\int_\R \langle k\rangle^{2s}\langle \tau+k^2\rangle^{2b}|\widehat u(\tau,k)|^2d\tau\]
    is finite. For $T>0$, the restriction space $X^{s,b}_T$ to the time interval $[0,T]$ is equipped with norm
    \[\|u\|_{_{X^{s,b}_T}}=\inf\{\|\tilde u\|_{_{X^{s,b}}};\tilde u=u\ {\rm on\ }[0,T]\times\T\}.\]
    For a bounded interval $I$, the associated restricted space $X^{s,b}_I$ can be defined similarly.
\end{definition}

The following properties are easily derived from the definition.

\begin{lemma}\label{Lemma-Bourgainspace}
Let $s,b\in \mathbb{R}$ and $T>0$ be arbitrarily given. Then the following assertions hold.

\begin{itemize}
\item[$(1)$] If $u\in X^{s,b}$, then $\|u\|_{_{X^{s,b}}}=\|S(-t)u(t)\|_{_{H_t^b H_x^s}}$. Here $S(t)=e^{i t\partial_x^2}$ and
\[\|f\|_{_{H_t^b H_x^s}}^2=\sum_{k\in \mathbb{Z}}\int_{\mathbb{R}}  \langle \tau\rangle^{2b} \langle k\rangle^{2s} |\widehat{f}(\tau,k)|^2 d\tau.\]

\item[$(2)$] If $b>1/2$, then $X^{s,b}\hookrightarrow C(\mathbb{R};H^s(\mathbb{T}))$ and $X_T^{s,b}\hookrightarrow C(0,T;H^s(\mathbb{T}))$.
        
\item[$(3)$] If $s_1\le s_2$ and $b_1\le b_2$, then $X^{s_2,b_2}\hookrightarrow X^{s_1,b_1}$. Moreover, if $s_1<s_2$ and $b_1<b_2$, then this embedding is compact. The same results hold for the restricted space.

\item[$(4)$] The dual space of $X^{s,b}$ is $X^{-s,-b}$, and the dual space of $X_T^{s,b}$ is $X_T^{-s,-b}$.
\end{itemize}
\end{lemma}

For the reader's convenience, we provide a brief proof.

\begin{proof}[{\bf Proof of Lemma~\ref{Lemma-Bourgainspace}}]
Assertion (1) follows from definition and a direct computation that
\[\mathcal{F}(S(-t)u(t)) (\tau,k)=\widehat{u}(\tau-|k|^2,k).\]

We only prove the assertions in (2)-(4) involving $X^{s,b}$, as the corresponding results for restricted space $X_T^{s,b}$ follow easily from them, by considering extensions of $u$.
    
    The embedding (2) follows from Sobolev embedding $H_t^b\hookrightarrow C(\mathbb{R})$ when $b>1/2$, which implies $S(-t)u(t)\in C(\mathbb{R};H^s(\mathbb{T}))$ by (1), and hence $u\in C(\mathbb{R};H^s(\mathbb{T}))$ due to the continuity of $S(-t)$.
    
    The continuity of embedding in (3) is trivial by definition. As for compactness, exploit (1) and compact embedding of Sobolev spaces $H_t^{b_2} H_x^{s_2}\Subset H_t^{b_1} H_x^{s_1}$ whenever $s_1<s_2$ and $b_1<b_2$.
    
    The duality (4) follows from $(H_t^b H_x^s)^*=H_t^{-b} H_x^{-s}$ and that $S(-t)$ is a unitary operator. 
\end{proof}

Recall $S_a(t)$ is the group generated by damped operator $i\partial_x^2-a(x)$. The following estimates related to $S_a(t)$ are crucial to well-posedness, as well as deriving basic estimates for the solution. In particular, the estimate (\ref{Multilinear-estimate-3modify}) provides the smallness needed for fixed-point argument.

\begin{proposition}\label{Proposition-S(t)estimate}
For every $T>0$, $b\in(1/2,1)$ and $s\in \mathbb{R}$, the following assertions hold.

\begin{enumerate}
\item[$(1)$] There exists a constant $C>0$, such that for any $u_0\in H^s(\mathbb{T})$,
\begin{equation}\label{Multilinear-estimate-2}
\|S_a(t)u_0\|_{_{X_T^{s,b}}}\leq C\|u_0\|_{_{H^s}}.
\end{equation}

\item[$(2)$] There exists a constant $C>0$, such that for any $F\in X_T^{s,b-1}$,
\begin{equation}\label{Multilinear-estimate-3}
\|\int_0^tS_a(t-\tau)F(\tau)d\tau\|_{_{X_T^{s,b}}}\leq C\|F\|_{_{X_T^{s,b-1}}}.
\end{equation}

\item[$(3)$] Assume further that a parameter $b'\in (0,1/2)$ satisfies $b+b'\le 1$. Then there exists a constant $C>0$, such that for any $0<T<1$ and $F\in X_T^{s,-b'}$,
\begin{equation}\label{Multilinear-estimate-3modify}
\|\int_0^t S(t-\tau)F(\tau)d\tau\|_{_{X_T^{s,b}}}\leq CT^{1-b-b'}\|F\|_{_{X_T^{s,-b'}}}.
\end{equation}
\end{enumerate}
\end{proposition}
    
    The proof of (1) and (2) can be found in \cite{RZ-09}. We mention that without the damping $a(x)$, the estimates in (1) and (2) with respect to $S(t)$ are standard.
    
    For the reader's convenience, we sketch the proof of estimate of type (\ref{Multilinear-estimate-2}) for $S(t)$, namely,
    \[\|S(t)u_0\|_{X_T^{s,b}}\le C \|u_0\|_{_{H^s}}.\]
    Indeed, choose a cut-off function $\psi\in C_c^\infty(\mathbb{R})$ so that $\psi(t)=1$ for $t\in [0,T]$, then
    \[\|S(t)u_0\|_{_{X_T^{s,b}}}\le \|\psi(t) S(t) u_0\|_{_{X^{s,b}}}=\|\psi(t) u_0\|_{_{H_t^b H_x^s}}\le C\|u_0\|_{_{H^s}}.\]
    Here we tacitly used Lemma~\ref{Lemma-Bourgainspace}(1) and that $\psi(t)$ commutes with $S(t)$.

    The proof of (\ref{Multilinear-estimate-3}) when $a(x)$ vanishes is similar to Proposition~\ref{Proposition-S(t)estimate}(3), which we now present.

    \begin{proof}[{\bf Proof of Proposition~\ref{Proposition-S(t)estimate}(3)}]
        Choose an extension of $F$ so that $\|F\|_{_{X^{s,-b'}}}\le 2\|F\|_{_{X^{s,-b'}_T}}$.
        We quote from \cite[Lemma 1.3]{Laurent-10} that, if $\psi\in C_c^\infty(\mathbb{R})$, then 
        \[\|\psi(t/T)\int_0^t \phi(\tau)d\tau\|_{_{H^b}}\le C T^{1-b-b'} \|\phi\|_{_{H^{-b'}}}.\]
        In particular, if we choose $\psi$ so that $\psi(t)=1$ for $t\in [0,1]$ and $\phi(t)=S(-t)u(t)$, then
        \begin{align*}
            &\|\int_0^t S(t-\tau)F(\tau)d\tau\|_{_{X_T^{s,b}}}\le \|\psi(t/T)\int_0^t S(-\tau)F(\tau)d\tau\|_{_{_{H_t^b H_x^s}}}\\
            &\le CT^{1-b-b'}\|S(-\tau) F(\tau)\|_{_{H_t^{-b'}H_x^s}}=CT^{1-b-b'}\|F\|_{_{X_T^{s,-b'}}}.\qedhere
        \end{align*}
    \end{proof}
    
\medskip

In applications, the function $F$ in \eqref{Multilinear-estimate-3} represents either the external force or nonlinear terms. The next proposition serves to deal with the nonlinear term $|u|^{p-1} u$.

\begin{proposition}\label{Proposition-multilinear}
Let the operator $\mathcal T$ be defined by {\rm(\ref{multilinear-operator})}. For every $T>0$, $s\ge 1$ and constants $b,b'$ satisfying $0<b'<1/2<b<1$ and $b+b'\le 1$, the following assertions hold.
\begin{enumerate}
\item[$(1)$]  There exists a constant $C>0$, such that for any $u_1,\cdots,u_p\in X_T^{s,b}$,
\begin{equation}\label{Multilinear-estimate}
\|\mathcal T(u_1,\cdots,u_p)\|_{_{X_T^{s,-b'}}}\leq C\sum_{j=1}^p \|u_j\|_{_{X_T^{s,b}}} \prod_{l\not =j}\|u_l\|_{_{X_T^{1,b}}}.
\end{equation}
In particular, if $u_1,\cdots ,u_p\in X_T^{s,b}$, then
\begin{equation}\label{Multilinear-estimate-homogeneous}
\|\mathcal T(u_1,\cdots,u_p)\|_{_{X_T^{s,-b'}}}\leq C\prod_{l=1}^p \|u_l\|_{_{X_T^{s,b}}}.
\end{equation}

\item[$(2)$] There exists a constant $C>0$, such that for $u_1\in X_T^{-s,b}$ and $u_2,\cdots ,u_p\in X_T^{s,b}$,
\begin{equation}\label{Multilinear-estimate-dual}
    \|\mathcal{T}(u_1,\cdots ,u_p)\|_{_{X_T^{-s,-b'}}}\le C\|u_1\|_{_{X_T^{-s,b}}}\prod_{l=2}^p \|u_l\|_{_{X_T^{s,b}}}.
\end{equation}
\end{enumerate}
\end{proposition}

\begin{proof}[{\bf Proof of Proposition \ref{Proposition-multilinear}}]
    When $p=3$, these estimate are well-known, and valid for $s\ge 0$; see, e.g.~\cite[Section V.2]{Bourgain-book}. However, as for general odd $p\ge 3$, we could not find an explicit reference. For this reason, we provide here a sketched proof. First note that (2) is a corollary of (1), owing to duality for the operator $u_1\mapsto \mathcal{T}(u_1,\cdots ,u_p)$.
    
    To establish (1), we use the dual method as in Step 1 of the proof of Lemma~\ref{Lemma-smoothing}. More precisely, for each $k\in \mathbb{Z}$, we split the configurations $(k_1,\cdots ,k_p)\in \mathbb{Z}^p$ with $k=k_1-k_2+\cdots +k_p$ via which $k_j$ has the maximal modulus. Then the dual method as in Lemma~\ref{Lemma-smoothing} leads to
    \[\|\mathcal{T}(u_1,\cdots ,u_p)\|_{_{X_T^{s,-b'}}}\le \sum_{j=1}^p M_j^{1/2} \|u_j\|_{_{X_T^{s,b}}} \prod_{l\not =j}\|u_l\|_{_{X_T^{1,b}}},\]
    where (note that if $|k_j|$ has maximal modulus, then $|k_j|\ge |k|/p$)
    \[M_j:=\sup_{k\in \mathbb{Z}} \sum_{\substack{k=k_1-k_2+\cdots +k_p\\|k_j|\ge |k|/p}}\frac{\langle k\rangle^{2s}}{\langle k_j\rangle^{2s} \prod_{l\not =j} \langle k_l\rangle^{2}}\le C \prod_{l\not =j} \sum_{k_l\in \mathbb{Z}} \frac{1}{\langle k_l\rangle^2}=C.\]
    The boundedness of $M_j$ yields estimate \eqref{Multilinear-estimate}.
\end{proof}

The last thing we need is multiplying spatial and temporal smooth functions.

\begin{lemma}\label{Lemma-multiplication}
For every $T>0$, $s\in \mathbb{R}$, $b\in [-1,1]$ and functions $\phi(x) \in C^\infty(\mathbb{T})$, there exists a constant $C>0$ such that for any $u\in X_T^{s,b}$,
\[\|\phi(x) u\|_{_{X_T^{s-|b|,b}}}\le C \|u\|_{_{X_T^{s,b}}}.\]
Moreover, there exists a constant $C>0$ such that for any $\psi\in H^1(0,T;\mathbb{C})$ and $u\in X_T^{s,b}$,
\[\|\psi(t)u\|_{_{X_T^{s,b}}}\le C\|\psi\|_{_{H^1(0,T;\mathbb{C})}}\|u\|_{_{X_T^{s,b}}}.\]
\end{lemma}

For a proof, see, e.g.~\cite[Lemma 1.2]{Laurent-ECOCV}. It might be surprising at first sight that multiplying $\phi(x)\in C^\infty(\mathbb{T})$ does not preserve $X^{s,b}$ functions; an example can also be found in \cite{Laurent-ECOCV}.

\subsection{Global well-posedness}

To begin with, we concentrate on the NLS equation (\ref{Nonlinear-problem}).

\begin{proposition}[Global well-posedness of NLS]\label{Proposition-nonlinear}
Let $T>0$, $s\geq 1$ and $b\in(1/2,1)$ be arbitrarily given. Then for every $u_0\in H^s$ and $f\in L^2(0,T;H^s)$, problem {\rm(\ref{Nonlinear-problem})} admits a unique solution $u\in X^{s,b}_T$. Moreover, the solution map
$$
H^s\times L^2(0,T;H^s)\ni (u_0,f)\mapsto u\in X^{s,b}_T,
$$
is locally Lipschitz and continuously differentiable.
\end{proposition}

\begin{proof}[{\bf Sketched proof of Proposition~\ref{Proposition-nonlinear}}]

The proof is based on a fixed-point argument similar to that in \cite[Theorem 2.1]{Laurent-ECOCV}. For the sake of completeness, we provide a brief sketch. 
One can first perform a contraction argument in  $X^{s,b}_T$-space and  derive the local existence of solutions. In this step, the estimates (\ref{Multilinear-estimate-2}), (\ref{Multilinear-estimate-3modify}) and (\ref{Multilinear-estimate-homogeneous}) come into play.
Moreover, the multilinear estimate (\ref{Multilinear-estimate}) enables us to reduce the issue of global existence to the case $s=1$, which means $H^1$-norm does not blow up. This could be done by deriving from (\ref{energy-2}) that 
$$
\sup_{\tau\in[0,t]} E_u(\tau)-E_u(0)\lesssim \int_0^{t}\|u(\tau)\|^2_{_{L^2}}d\tau+
\sup_{\tau\in[0,t]}\left[
E_u^{1/2}(\tau)+E_u^{p/(p+1)}(\tau)
\right] 
\int_0^{t}\|f(\tau)\|_{_{H^1}}d\tau.
$$
Here, the defocusing trait of nonlinearity plays an essential role. 
Finally, the Lipschitz property and differentability of solution map can be obtained by applying the multilinear estimates on small intervals, and following an induction procedure.
\end{proof}

In the sequel, we present a profile of the linearization of (\ref{Nonlinear-problem}), reading
\begin{equation}\label{Linear-problem}
\left\{\begin{array}{ll}
iv_t+v_{xx}+ia(x)v=\frac{p+1}{2}|u|^{p-1}v+\frac{p-1}{2}|u|^{p-3}u^2\bar v+f(t,x),\\
v(0,x)=v_0(x),
\end{array}
\right.
\end{equation}
where $v_0\in H^s$, $f\in L^2_tH^{s}_x$ and $u\in X^{s,b}_T$ with $b\in (1/2,1)$.
We are also interested in an adjoint problem for (\ref{Linear-problem}) which is homogeneous and backward in time. That is,
\begin{equation}\label{Adjoint-problem}
\left\{\begin{array}{ll}
i\varphi_t+\varphi_{xx}-ia(x)\varphi=\frac{p+1}{2}|u|^{p-1}\varphi-\frac{p-1}{2}|u|^{p-3}u^2\bar \varphi,\\
\varphi(T,x)=\varphi_T(x),
\end{array}
\right.
\end{equation}
where $v_T\in H^{-s}$.
The solution maps of (\ref{Linear-problem}) and (\ref{Adjoint-problem}) (defined also via the Duhamel formula) are denoted by $v=\mathcal V_u(v_0,f)$ and $\varphi=\mathcal U_u(\varphi_T)$, respectively. Thanks to the estimates \eqref{Multilinear-estimate-2}, \eqref{Multilinear-estimate-3}, \eqref{Multilinear-estimate-homogeneous} and \eqref{Multilinear-estimate-dual}, one can derive the global existence of problems \eqref{Linear-problem} and \eqref{Adjoint-problem}. Some apriori estimates in need are gathered below. The proof is omitted for the sake of simplicity.

\begin{proposition}[Basic estimates of linear equations]\label{Proposition-linearproblem}
Let $T,R>0$ and $s\ge 1$ be arbitrarily given, and assume constants $b,b'$ satisfy $0<b'<1/2<b$ and $b+b'<1$. Then there exists a constant $C>0$, such that the following assertions hold. 
\begin{enumerate}
\item[$(1)$] If $v_0\in H^s$, $f\in X^{s,-b'}_T$ and $u\in \overline{B}_{X_T^{s,b}}(R)$, then the linearized equation \eqref{Linear-problem} admits a unique solution $v=\mathcal V_u(v_0,f)\in X_T^{s,b}$, and
\begin{equation*}
\|v\|_{_{X^{s,b}_T}}\leq C\left(
\|v_0\|_{_{H^s}}+\|f\|_{_{X^{s,-b'}_T}}
\right).
\end{equation*}
		
\item[$(2)$] If $\varphi_T\in H^{-s}$ and $u\in \overline{B}_{X_T^{s,b}}(R)$, then the adjoint problem \eqref{Adjoint-problem} admits a unique solution $\varphi=\mathcal U_u(\varphi_T)\in X_T^{-s,b}$, and
\begin{equation*}
\|\varphi\|_{_{X^{-s,b}_T}}\leq C
\|\varphi_T\|_{_{H^{-s}}}.
\end{equation*}
Moreover, the solution map 
$$
\overline{B}_{X_T^{s,b}}(R)\ni u\mapsto \mathcal U_u(\cdot)\in \mathcal{L}_{\R}(H^{-s};X^{-s,b}_T)
$$
is Lipschitz and continuously differentiable.
\end{enumerate}
\end{proposition}

\section{Carleman estimate}\label{appendix_Carleman}

his appendix is devoted to a detailed proof of the general Carleman estimate proposed in Proposition~\ref{Proposition-Carleman}.

\vspace{2mm}

Let us introduce the transformation
\begin{equation}\label{transformation-0}
U(t,x)=e^{-s\alpha(t,x)}u(t,x),\quad H(t,x)=e^{-s\alpha(t,x)}h(t,x),
\end{equation}
where $s\geq 1$ is a large parameter. The new variable $U$ verifies that
\begin{equation*}
\left\{\begin{array}{ll}
iU_t+ U_{xx}
+is\alpha_tU+s\alpha_{xx}U+s^2\alpha_x^2U+2s\alpha_xU_x-g(|u|^2)U=H(t,x),\\
U(0,x)=U(T,x)\equiv 0.
\end{array}\right.
\end{equation*}
We then split the LHS by $P_1 U+P_2 U$, where
\[P_1U= is\alpha_tU+s\alpha_{xx}U+2s\alpha_xU_x\quad \text{and}\quad 
P_2U= iU_t+ U_{xx}+s^2\alpha_x^2U-g(|u|^2)U.\]
Then one can readily see
\begin{equation}\label{estimate-1}
2{\rm Re}\int_{Q_T}P_1U\overline{P_2U}\leq \|P_1 U+P_2 U\|_{_{L^2(Q_T)}}^2=\|H\|^2_{_{L^2(Q_T)}}.
\end{equation}
To continue, we write the LHS of \eqref{estimate-1} as $I_1+I_2+I_3+I_4$, where 
\begin{align*}
I_1&= 2{\rm Re}\int_{Q_T}(s\alpha_{xx}U+2s\alpha_xU_x)(\bar U_{xx}+s^2\alpha_x^2\bar U-g(|u|^2)\bar U),\\
I_2&=-2{\rm Re}\left(i\int_{Q_T}(s\alpha_{xx}U+2s\alpha_xU_x)\bar{U}_t\right),\\
I_3&= 2{\rm Re}\int_{Q_T} is\alpha_tU(-i\bar U_t+\bar U_{xx}),\\
I_4&= 2{\rm Re}\int_{Q_T}is\alpha_tU(s^2\alpha_x^2\bar U-g(|u|^2)\bar U)=0.
\end{align*}
Here $I_4=0$ due to $\alpha$ and $g$ being real-valued. We will deal with the terms $I_1,I_2,I_3$ separately. 

Let us begin with $I_1$. Integrating by parts and noticing $\partial_x |U|^2= 2{\rm Re}(U\bar{U}_x)$,
we derive that
\begin{align*}
&2{\rm Re}\int_{Q_T} s\alpha_{xx}U\bar U_{xx}=-2s{\rm Re}\int_{Q_T}(\alpha_{xxx}U+\alpha_{xx}U_x)\bar U_x=s\int_{Q_T}(\alpha_{xxxx}|U|^2-2\alpha_{xx}|U_x|^2),\\
&2{\rm Re}\int_{Q_T}2s\alpha_xU_x\bar U_{xx}=-2s\int_{Q_T} \alpha_{xx}|U_x|^2,\\
&2{\rm Re}\int_{Q_T} 2s\alpha_xU_x s^2\alpha_x^2\bar U=-6s^3\int_{Q_T}\alpha_x^2 \alpha_{xx} |U|^2,\\
&-2{\rm Re}\int_{Q_T}2s\alpha_xU_xg(|u|^2)\bar{U}=-2s\int_{Q_T} \alpha_x g(|u|^2)\partial_x |U|^2.
\end{align*}
The last expression can be further reduced by
\begin{align*}
g(|u|^2)\partial_x |U|^2&=e^{-2s\alpha}g(e^{2s\alpha}|U|^2)\left[\partial_x (e^{2s\alpha}|U|^2)-2s\alpha_x e^{2s\alpha}|U|^2\right]\\
&=e^{-2s\alpha}\left[\partial_x G(e^{2s\alpha}|U|^2)-2s\alpha_x g(|u|^2)|u|^2\right]\\
&=e^{-2s\alpha}\left[\partial_x G(|u|^2)-2s\alpha_x g(|u|^2)|u|^2\right].
\end{align*}
Therefore, we conclude that
\begin{align}
I_1=&\int_{Q_T}(-4s^3\alpha_x^2\alpha_{xx}+s\alpha_{xxxx})|U|^2-4s\int_{Q_T} \alpha_{xx}|U_x|^2-2s\int_{Q_T} \alpha_{xx}g(|u|^2) |U|^2\notag\\
&-2s\int_{Q_T} \alpha_x e^{-2s\alpha}\left[\partial_x G(|u|^2)-2s\alpha_x g(|u|^2)|u|^2\right]\label{estimate-I1}\\
=&\int_{Q_T}(-4s^3\alpha_x^2\alpha_{xx}+s\alpha_{xxxx})|U|^2-4s\int_{Q_T} \alpha_{xx}|U_x|^2+\int_{Q_T} (4s^2\alpha_x^2-2s\alpha_{xx}) e^{-2s\alpha} \Psi(|u|^2).\notag
\end{align}
Next, using $2{\rm Re}\, z=z+\bar{z}$, we rewrite $I_2$ as
\begin{align}
I_2&=-\int_{Q_T}i(s\alpha_{xx}U+2s\alpha_xU_x)\bar U_t+\int_{Q_T}i(s\alpha_{xx}\bar U+2s\alpha_x\bar U_x)U_t,\notag\\
&=is\int_{Q_T}\left(\alpha_{txx}|U|^2+\alpha_{xx}U_t\bar U+2\alpha_{tx}U_x\bar U+2\alpha_xU_{tx}\bar U\right)\notag\\
&\quad+is\int_{Q_T} \alpha_{xx} \bar{U}U_t-2is\int_{Q_T}\left(\alpha_{xx}U_t+\alpha_xU_{tx}\right)\bar U \label{esimate-I2}\\
&=is\int_{Q_T}\alpha_{txx}|U|^2+2is\int_{Q_T}\alpha_{tx}U_x\bar U\notag\\
&=-is\int_{Q_T} a_{tx} U \bar{U}_x+is\int_{Q_T}\alpha_{tx}U_x\bar U=-2s{\rm Re}\left(
i\int_{Q_T} \alpha_{tx} U \bar U_x
\right).\notag
\end{align}
Finally, to deal with $I_3$, let us write
\begin{equation}\label{estimate-I3}
\begin{aligned}
I_3&=s\int_{Q_T}\alpha_t\partial_t|U|^2+2s{\rm Re}\left(
i\int_{Q_T}\alpha_tU\bar U_{xx}
\right)\\
&=-s\int_{Q_T}\alpha_{tt}|U|^2-2s{\rm Re}\left(
i\int_{Q_T}\alpha_{tx}U\bar U_{x}
\right).
\end{aligned}
\end{equation}

Now, substituting (\ref{estimate-I1})-(\ref{estimate-I3}) into \eqref{estimate-1}, it hence follows that
\begin{equation}\label{formula-3}
\begin{aligned}
\int_{Q_T} |H|^2\ge &\int_{Q_T}\left(s\alpha_{xxxx}-4s^3\alpha_x^2\alpha_{xx}-s\alpha_{tt}\right)|U|^2&\\
&-4s\int_{Q_T} \alpha_{xx}|U_x|^2-4s{\rm Re}\left(
i\int_{Q_T}\alpha_{tx}U\bar U_x 
\right)\\
&+\int_{Q_T} (4s^2\alpha_x^2-2s\alpha_{xx}) e^{-2s\alpha} \Psi(|u|^2).
\end{aligned}
\end{equation}
In the sequel, the following facts will be used without explicit mention: 
\begin{equation*}
\begin{aligned}
& \alpha_x=-\lambda\beta\phi',\quad \beta_x=\lambda\beta\phi',\quad \alpha_{xx}=-\lambda\beta\phi''-\lambda^2|\phi'|^2\beta \\
& |\alpha_{tx}|\leq C\lambda T\beta^2,\quad |\alpha_{tt}|\leq CT^2\beta^3,\quad |\alpha_{xx}|\leq C \lambda^2 \beta,\quad |\alpha_{xxxx}|\leq C\lambda^4\beta,
\end{aligned}
\end{equation*}
where the generic constant $C>0$ depends only on $\phi$. Thus for any $x\in \mathbb{T}$,
\begin{equation*}
\begin{aligned}
&|s\alpha_{xxxx}-4s^3\alpha_x^2 \alpha_{xx}-s\alpha_{tt}|\leq C (s\lambda^4 \beta+s^3\lambda^4\beta^3+sT^2\beta^3)\leq C s^3 \lambda^4\beta^3,\\
&|4s \alpha_{xx}|\leq C s\lambda^2 \beta,\\
&|4s^2\alpha_x^2-2s\alpha_{xx}|\leq C s^2\lambda^2 \beta^2+s\lambda^2 \beta\leq Cs^2\lambda^2 \beta^2,
\end{aligned}
\end{equation*}
provided that $s$ and $\lambda$ are sufficiently large (depending only on $T$). In addition, we recall (\ref{weight-function-0}) and find that for $x\in \mathbb{T}\setminus \overline{\mathcal{I}}$,
\begin{equation}\label{estimate-4}
\begin{aligned}
&s\alpha_{xxxx}-4s^3\alpha_x^2 \alpha_{xx}-s\alpha_{tt}\geq -Cs\lambda^4\beta+4c^4 s^3\lambda^4\beta^3-CsT^2\beta^3\ge 2c^4 s^3\lambda^4\beta^3,\\
&-4s\alpha_{xx}\ge 4c^2 s\lambda^2\beta,\\
&4s^2\alpha_x^2-2s\alpha_{xx}\ge 4c^2s^2\lambda^2 \beta^2+2c^2 s\lambda^2\beta^2\ge 4c^2s^2\lambda^2 \beta^2,
\end{aligned}
\end{equation}
provided that $s$ and $\lambda$ are sufficiently large.
Combining (\ref{formula-3}) with (\ref{estimate-4}), we conclude that
\begin{align*}
&s^3\lambda^4 \int_{Q_T\setminus q_T} \beta^3 |U|^2+s\lambda^2 \int_{Q_T\setminus q_T} \beta |U_x|^2+s^2\lambda^2 \int_{Q_T\setminus q_T} \beta^2 e^{-2s\alpha} \Psi(|u|^2)\\
&\leq C\left[ s^3\lambda^4 \int_{q_T}\beta^3|U|^2+s\lambda^2 \int_{q_T} \beta |U_x|^2+s^2\lambda^2\int_{q_T} \beta^2 e^{-2s\alpha} \Psi(|u|^2)\right.\\
&\quad\quad\ \left. +\int_{Q_T} |H|^2+4s{\rm Re}\left(
i\int_{Q_T}\alpha_{tx}U\bar U_x 
\right)\right]\\
&\leq C\left[s^3\lambda^4 \int_{q_T}\beta^3|U|^2+s\lambda^2 \int_{q_T} \beta |U_x|^2+s^2\lambda^2\int_{q_T} \beta^2 e^{-2s\alpha}\Psi(|u|^2)\right.\\
&\quad\quad\  \left.+\int_{Q_T} |H|^2+4s\lambda T\int_{Q_T} \beta^2 |U||U_x|\right]. 
\end{align*}
As a result, we obtain
\begin{equation}\label{C-estimate-5}
\begin{aligned}
&s^3\lambda^4 \int_{Q_T} \beta^3 |U|^2+s\lambda^2 \int_{Q_T} \beta |U_x|^2+s^2\lambda^2 \int_{Q_T} \beta^2 e^{-2s\alpha} \Psi(|u|^2)\\
&\leq C\left[s^3\lambda^4 \int_{q_T}\beta^3|U|^2+s\lambda^2 \int_{q_T} \beta |U_x|^2+s^2\lambda^2\int_{q_T} \beta^2 e^{-2s\alpha} \Psi(|u|^2)\right.\\
&\quad\quad\ \left. +\int_{Q_T} |H|^2+4s\lambda T\int_{Q_T} \beta
^2 |U||U_x|\right].
\end{aligned}
\end{equation}
The last term can be absorbed into the LHS, by setting $s,\lambda$ to be sufficiently large and using
\[4s\lambda T\int_{Q_T} \beta^2 |U||U_x|\le 2s\lambda T\int_{Q_T} (\beta^3
|U|^2+\beta |U_x|^2).\]
Finally, the desired conclusion can be derived from \eqref{C-estimate-5} (with the last term removed) by recalling (\ref{transformation-0}). And Proposition~\ref{Proposition-Carleman} is proved.

\section{Auxiliary proofs in control problems}\label{Appendix-controlproblem}

This appendix consists of three proofs related to  control property studied in Section \ref{Section-control}. These arguments are not novel and are produced here for the sake of completeness.

\subsection{Proof of Theorem {\rm\ref{Theorem-control}} via Proposition {\rm\ref{Proposition-control-1}}}\label{Appendix-lineartest}

Let $\tilde u\in \overline{B}_{X_T^{s+\sigma,b}}(R) $ be the solution of the uncontrolled system (\ref{Uncontrolled-Problem}) with $\tilde u_0\in H^{s+\sigma}$ and $h\in L^2_tH^{s+\sigma}_x$. Also let $u(t)$ be the controlled solution of (\ref{Control-Problem-0}) with undetermined $\xi\in L_t^2 H_x^s$. Consider the equation for $r=u-\tilde{u}$, which reads
\begin{equation*}
\left\{\begin{array}{ll}
ir_t+ r_{xx}+ia(x)r=\frac{p+1}{2}|\tilde u|^{p-1}r+\frac{p-1}{2}|\tilde u|^{p-3}\tilde u^2\bar r+F(r,\tilde u)+\chi(x)\mathcal{P}_N \xi(t,x),\\
r(0,x)=r_0(x)(:=u_0(x)-\tilde{u}_0(x)),
\end{array}\right.
\end{equation*}
where $
F(r,\tilde u)=\sum_{D} c_D \mathcal T(x_1,\cdots,x_p),
$
with the set $D$ consisting of those $(x_1,\cdots,x_p)$ so that at least two  $x_l$ coincide with $r$ and the others equal to $\tilde u$. 

Proposition \ref{Proposition-control-1} yields that there exist constants $N\in\N^+$ and $q_1\in (0,1)$ such that \begin{equation}\label{Squeezing-2}
\|v(T)\|_{_{\tilde{H}^s}}\leq q_1 \|v_0\|_{_{\tilde{H}^s}},
\end{equation}
where $v(t)$ stands for the solution of {\rm(\ref{Control-Problem-1})} with $v_0:=r_0$. Recall $\|\cdot \|_{_{\tilde{H}^s}}$ is equivalent to the standard $H^s$-norm. Moreover, the size of control $\xi$ can be bounded by
\begin{equation}\label{Squeezing-3}\int_0^T\|\xi\|^2_{_{H^s}}\leq C\|v_0\|_{_{H^s}}^2.
\end{equation}
We will fix this $\xi$ in the rest of the proof.

Applying the multilinear estimate (\ref{Multilinear-estimate}), it can be seen that
\begin{equation}\label{Multilinear-estimate-5}
\|F(r,\tilde u)\|_{_{X_T^{s,-b'}}}\leq C \sum_{l=2}^p \|r\|_{_{X_T^{s,b}}}^l.
\end{equation}
It is easy to see that the solution map is locally Lipschitz, whose proof is similar to Proposition~\ref{Proposition-nonlinear}. Therefore if $\|r_0\|_{_{H^s}}\le 1$, owing to \eqref{Squeezing-3} and $v_0=r_0$, we have
\begin{equation}\label{bound-0}
\|r\|_{_{X^{s,b}_{T}}}\leq C\left(\|r_0\|_{_{H^s}}+\|\xi\|_{_{L^2_tH^s_x}}\right)\le C\|r_0\|_{_{H^s}}.
\end{equation}

Letting $y(t)=r(t)-v(t)$, one can readily see that
\begin{equation*}
	\left\{\begin{array}{ll}
		iy_t+y_{xx}+ia(x)y=\frac{p+1}{2}|\tilde u|^{p-1}y+\frac{p-1}{2}|\tilde u|^{p-3}\tilde u^2\bar y+F(r,\tilde u),\\
		y(0,x)=0.
	\end{array}\right.
\end{equation*}
Using (\ref{Multilinear-estimate-5}), \eqref{bound-0} and Proposition \ref{Proposition-linearproblem}(1), it follows that for any $q'\in (q_1,1)$, we have
\[\|y(T)\|_{_{\tilde{H}^s}}\le C\|y\|_{_{X^{s,b}_T}}\leq C  \|F(r,\tilde u)\|_{_{X^{s,-b'}_T}}\leq C \|r_0\|_{_{H^s}}\sum_{l=1}^{p-1}\|r_0\|_{_{H^s}}^l\leq (q'-q_1)\|r_0\|_{_{\tilde{H}^s}}\]
whenever $\|r_0\|_{_{H^s}}\leq d$ with $0<d\le 1$ sufficiently small. This together with (\ref{Squeezing-2}) and $r=v+y$ implies the first conclusion of Theorem \ref{Theorem-control}. Finally, the second conclusion follows directly from the above construction of control and Proposition \ref{Proposition-control-1}(2). 

\subsection{Proof of Lemma {\rm\ref{Lemma-fullobs}}}\label{Appendix-control}

To begin with, we recall some propagation results for linear Schr\"{o}dinger equations, which can be found in \cite[Theorem 4.1 and Theorem 5.1]{Laurent-ECOCV}.

\begin{lemma}\label{Lemma-propagation}
Let $T>0,\,c\in[0,1)$ and $r\in\R$ be arbitrarily given, and let $\mathcal I$ be a nonempty open subset of $\T$. Then the following assertions hold.
\begin{enumerate}
\item[$(1)$] {\rm(}Propagation of regularity{\rm)} Let $\varphi\in X_T^{r,c}$ be a weak solution (in the distribution sense) of 
$$
i\varphi_t+\varphi_{xx}=f(t,x)
$$
with $f\in X_T^{r,-c}$.
Assume further that $\varphi\in L^2_{loc}(0,T;H^{r+\rho}(\mathcal I))$ for some $\rho\leq (1-c)/2$. Then $\varphi\in L^2_{loc}(0,T;H^{r+\rho})$.

\item[$(2)$] {\rm(}Propagation of compactness{\rm)} Let $\{\varphi^n\}\subset X_T^{0,c}$ be a sequence of weak solutions of
$$
i\varphi_t^n+\varphi_{xx}^n=f^n(t,x)
$$
such that $\{\varphi^n\}$ is bounded in $X_T^{0,c}$, $\varphi^n\rightarrow 0$ in $X_T^{-1+c,-c}$ and $f^n\rightarrow 0$ in $X_T^{-1+c,-c}$.
Assume further that $\varphi^n\rightarrow 0$ in $L^2(0,T;L^2(\mathcal I))$. Then $\varphi^n\rightarrow 0$ in $L^2_{loc}(0,T;L^2(\T))$.
\end{enumerate}
\end{lemma}

See also \cite{DGL-06,Laurent-10} for the corresponding results under the setting of compact Riemannian
manifold of dimension $\le 3$, whose proof lies within the framework of microlocal analysis. For reader's convenience, we briefly sketch the proof of conclusion (2), and the argument for (1) is similar.

\begin{proof}[{\bf Sketched proof of Lemma~{\rm\ref{Lemma-propagation}(2)}}]
    Without loss of generality, we can assume the functions are smooth. Let $\omega(t)\in C_c^\infty(0,T)$ and $\chi(x)\in C_c^\infty(\mathcal{I})$ be cutoff functions. It suffices to show
    \begin{equation}\label{propacompact-1}
    \left(\omega(t) \partial_x \eta(x) \varphi_n,\varphi_n\right)_{L^2_tL^2_x}\to 0
    \end{equation}
    for any $\eta\in C^\infty(\mathbb{T})$. In fact, given $x_0\in \mathbb{T}$, we can find $\eta$ such that $\partial_x \eta=\chi-\chi(\cdot-x_0)$. Thus
    \[\left(\omega(t) (\chi(x)-\chi(x-x_0))\varphi_n,\varphi_n\right)_{L^2_tL^2_x}\to 0.\]
    As $\omega(t)\chi(x) u_0\to 0$ in $L_t^2 L_x^2$ by assumption, we deduce that
    \[\left(\omega(t) \chi(x-x_0)\varphi_n,\varphi_n\right)_{L^2_tL^2_x}\to 0.\]
    And conclusion (2) follows by partition of unity. 
    
    The proof of \eqref{propacompact-1} relies on the commutator $[A,B]:=AB-BA$ of pseudo-differential operator $A:=\omega(t) D^{-1} \eta$ and Schr\"odinger operator $B:=i\partial_t+\partial_x^2$, where $D^r(r\in \mathbb{R})$ is defined by
    \[\widehat{D^r f}(k)=\begin{cases}
        \widehat{f}(k),&k=0,\\
        {\rm sgn}(k) |k|^r \widehat{f}(k),&k\not =0.
    \end{cases}\]
    On the one hand, one can derive from the boundedness of $\varphi_n$ and the decay of $f_n$ that
    \[([A,L]\varphi_n,\varphi_n)=(f_n,A^* \varphi_n)-(A\varphi_n,f_n)\to 0.\]
    On the other hand, direct manipulation gives
    \[[A,L]=-2\omega(t)(\partial_x \eta)\partial_x D^{-1}-\omega(t) \partial_x^2 \eta D^{-1}.\]
    A similar argument yields the contribution of the second term tends to $0$, and hence
    \[(\omega(t) \partial_x \eta(x) \partial_x D^{-1}\varphi_n,\varphi_n)_{L_t^2 L_x^2}\to 0.\]
    This is not far from \eqref{propacompact-1}, as $\partial_x D^{-1}$ is almost equal to identity. In fact, $-i\partial_x D^{-1}$ is the orthogonal projection to the complement of $\{f; \widehat{f}=0\}$. The rest of the proof is easy.
\end{proof}

Another tool needed for observability is unique continuation. Consider the linear equation
\begin{equation}\label{Linear-Schrodinger}
i\varphi_t+\varphi_{xx}=V_1\varphi+V_2\bar \varphi,
\end{equation}
where $V_i=V_i(t,x)$ stand for the potentials.

\begin{lemma}\label{Lemma-UCP}
Let $T>0$, $V_1,V_2\in C([0,T];H^1)$ and open subset $\mathcal I$ of $\T$ be arbitrarily given. Then if a solution $\varphi\in C([0,T];H^1)$ of \eqref{Linear-Schrodinger} satisfies $\varphi(t,x)=0$ on $[0,T]\times \mathcal I$, then $\varphi(t,x)=0$ on $[0,T]\times\T$.
\end{lemma}

\begin{proof}[{\bf Proof of Lemma \ref{Lemma-UCP}}]

We apply Proposition \ref{Proposition-Carleman} to (\ref{Linear-Schrodinger}), where the Carleman estimate holds in fact for the $H^1$-solution $\varphi$ (due to an approximation argument). In the present situation, 
\begin{equation}\label{Carleman-estimate-v}
\begin{aligned}
&s^3\lambda^4\int_{Q_T}\beta^3e^{-2s\alpha}|\varphi|^2+s\lambda^2\int_{Q_T}\beta e^{-2s\alpha}| \varphi_x|^2
\\
&\leq C\left[s^3\lambda^4\int_{\tilde q_T}\beta^3e^{-2s\alpha}|\varphi|^2+s\lambda^2\int_{\tilde q_T}\beta e^{-2s\alpha}| \varphi_x|^2+\int_{Q_T}e^{-2s\alpha}|V_1\varphi+V_2\bar \varphi|^2\right],
\end{aligned}
\end{equation}
where $s,\lambda$ are sufficiently large. 
To continue, using the assumption $V_1,V_2\in C([0,T];H^1)$, the last integral on the RHS of (\ref{Carleman-estimate-v}) can be bounded by 
\[\int_{Q_T}e^{-2s\alpha}|V_1\varphi+V_2\bar \varphi|^2\leq C\left(\|V_1\|^2_{_{L_t^\infty H^1_x}}+\|V_2\|^2_{_{L_t^\infty H^1_x}}\right)\int_{Q_T}e^{-2s\alpha}|\varphi|^2\leq C\int_{Q_T}\beta^3 e^{-2s\alpha}|\varphi|^2.\]
Thus, it can be absorbed by the LHS of (\ref{Carleman-estimate-v}), provided that $s$ is large enough. As a consequence,
\[s\lambda\int_{Q_T}\beta e^{-2s\alpha}| \varphi_x|^2+s^3\lambda^3\int_{Q_T}\beta^3e^{-2s\alpha}|\varphi|^2\leq C\left[s\lambda^2\int_{\tilde q_T}\beta e^{-2s\alpha}| \varphi_x|^2+s^3\lambda^4\int_{\tilde q_T}\beta^3e^{-2s\alpha}|\varphi|^2\right].\]
Finally, if $\varphi(t,x)=0$ on $\tilde q_T$, then the RHS  equals $0$. Hence, $\varphi(t,x)= 0$ on $Q_T$.
\end{proof}

The interested reader is referred to \cite[Appendix B]{Laurent-10} for the unique continuation on manifolds.

\begin{proof}[\bf Proof of Lemma~\ref{Lemma-fullobs}] The proof is divided into four steps.

\textit{Step 1: contradiction argument.} Assume the conclusion of this lemma fails, then there exist sequences $\{\varphi_T^k\}\subset H^{-s-\sigma'}$ and $\{\tilde u^k\}\subset \overline{B}_{X_T^{s+\sigma,b}}(R)$ such that $\|\varphi^k_T\|_{_{H^{-s-\sigma'}}}=1$ while 
\begin{equation}\label{converge-1}
\chi\varphi^k\rightarrow 0\quad {\rm in\ }L^2_tH^{-s-\sigma'}_x,
\end{equation}
where $\varphi^k=\mathcal U_{\tilde u^k}(\varphi^k_T)$.
In view of Proposition \ref{Proposition-linearproblem}(2),
it can be also derived that $\{\varphi^k\}$ is bounded in $X_T^{-s-\sigma',b}$. Meanwhile, as $H^{s+\sigma'}$ is an algebra, by duality we obtain that
$$
\||\tilde u^k|^{p-1}\varphi^k\|_{_{H^{-s-\sigma'}}}+\||\tilde u^k|^{p-3}(\tilde u^k)^2\bar \varphi^k\|_{_{H^{-s-\sigma'}}}\leq  C\|\tilde u^k\|^{p-1}_{_{H^{s+\sigma'}}}\|\varphi^k\|_{_{H^{-s-\sigma'}}},
$$
and hence the sequence of lower-order terms
$$
f^k:=ia(x)\varphi^k+\tfrac{p+1}{2}|\tilde u^k|^{p-1}\varphi^k-\tfrac{p-1}{2}|\tilde u^k|^{p-3}(\tilde u^k)^2\bar \varphi^k
$$
is bounded $L^\infty_tH^{-s-\sigma'}_x$ and hence in $X_T^{-s-\sigma',-b'}$. In summary, there exist $\varphi\in X_T^{-s-\sigma',b}$, $\tilde u\in \overline{B}_{X_T^{s+\sigma,b}}(R)$ and $f\in X_T^{-s-\sigma',-b'}$ such that up to extracting subsequences,
\begin{equation}\label{convergence-1}
\begin{aligned}[c]
\varphi^k\rightharpoonup \varphi\quad {\rm in\ }X_T^{-s-\sigma',b},\quad &\tilde u^k\rightharpoonup \tilde u\quad {\rm in\ }X_T^{s+\sigma,b},\quad f^k\rightharpoonup f \quad {\rm in\ }X_T^{-s-\sigma',-b'}.
\end{aligned}
\end{equation}
Moreover, $\varphi\in X_T^{-s-\sigma',b}$ is a mild solution of equation
\[i\varphi_t+\varphi_{xx}=f(t,x).\]In particular, the combination of (\ref{converge-1}) and (\ref{convergence-1}) leads to 
\begin{equation}\label{limit-2}
\chi\varphi(t,x)\equiv 0\quad {\rm on\ }Q_T.
\end{equation}

\medskip

\textit{Step 2: strong convergence of $\varphi^k$.} What follows is to deduce that
\begin{equation}\label{converge-5}
\varphi^k\rightarrow \varphi\quad {\rm in\ }L^2_{loc}(0,T;H^{-s-\sigma'}),
\end{equation}
by adopting the propagation of compactness.
Let us introduce 
\begin{align*}
&\Phi^k=(1-\partial^2_x)^{-(s+\sigma')/2}\varphi^k,\quad \Phi=(1-\partial^2_x)^{-(s+\sigma')/2}\varphi,\\
&F^k=(1-\partial^2_x)^{-(s+\sigma')/2}f^k,\quad
F=(1-\partial^2_x)^{-(s+\sigma')/2}f.
\end{align*}
Clearly, the sequence $\{\Phi^k\}$ is bounded in $X^{0,b}_T$. 
It also follows from (\ref{convergence-1}) that
\begin{align*}
& \Phi^k\rightarrow \Phi \quad {\rm in\ }X_T^{-1+b,-b},\quad F^k\rightarrow F\quad {\rm in\ }X_T^{-1+b,-b},\\
& i\Phi_t^k+\Phi_{xx}^k=F^k(t,x),\quad i\Phi_t+\Phi_{xx}=F(t,x),
\end{align*}
in view of the compact embeddings $X_T^{0,b}\Subset X_T^{-1+b,-b}$ and $X_T^{0,-b'}\Subset X_T^{-1+b,-b}$. Notice that 
\[\chi\Phi^k=[\chi,(1-\partial^2_x)^{-(s+\sigma')/2}]\varphi^k+(1-\partial^2_x)^{-(s+\sigma')/2}(\chi\varphi^k)\]
and communicator $[\chi,(1-\partial^2_x)^{-(s+\sigma')/2}]$ is bounded from $H^{-s-\sigma'-1}$ into $L^2(\T)$ (see, e.g.~\cite[Lemma A.1]{Laurent-ECOCV}). These, together with \eqref{convergence-1} and the compact embedding $X_T^{-s-\sigma',b}\Subset X_T^{-s-\sigma'-1,0}=L_t^2 H_x^{-s-\sigma'-1},$ imply that
$\chi\Phi^k\rightarrow \chi\Phi$ in $L^2(Q_T)$ and hence $\Phi^k\rightarrow \Phi $ in $L^2(0,T;\mathcal I_2).$ 
Therefore, an application of Lemma \ref{Lemma-propagation}(2) yields that $\Phi^k\rightarrow \Phi$ in $L^2_{loc}(0,T;L^2(\T))$, leading to (\ref{converge-5}).

\medskip

\textit{Step 3: expression of $f$.} We proceed to show that 
\begin{equation}\label{limit-4}
f=ia(x)\varphi+\tfrac{p+1}{2}|\tilde u|^{p-1}\varphi-\tfrac{p-1}{2}|\tilde u|^{p-3}\tilde u^2\bar \varphi.
\end{equation}
Indeed, by \eqref{converge-5}, we immediately obtain
\begin{equation}\label{convergence-7}
ia(x)\varphi^k\rightarrow ia(x)\varphi\quad {\rm in\ }L^2_{loc}(0,T;H^{-s-\sigma'})
\end{equation}
Also, we may take $\varepsilon>0$ sufficiently small, then the compact embedding $X_T^{s+\sigma,b}\Subset X_T^{s+\sigma-\varepsilon,b-\varepsilon} \hookrightarrow L^\infty (Q_T)$ allows us to extract a subsequence so that
\begin{equation}\label{converge-6}
\tilde u^k\rightarrow \tilde u\quad {\rm in\ }L^\infty(Q_T).
\end{equation}
Thus for any interval $J\Subset (0,T)$ and $\psi\in L^2(J;H_x^{s+\sigma'})$, we have
\begin{align*}
    &|\int_J \langle |\tilde u^k|^{p-1}\varphi^k-|\tilde u|^{p-1}\varphi,\psi  \rangle_{_{H^{-s-\sigma'},H^{s+\sigma'}}}|\\
    &=|\int_J \langle (|\tilde u^k|^{p-1}-|\tilde u|^{p-1})\varphi^k,\psi  \rangle_{_{H^{-s-\sigma'},H^{s+\sigma'}}}|+|\int_J \langle \varphi^k-\varphi,|\tilde u^k|^{p-1}\psi  \rangle_{_{H^{-s-\sigma'},H^{s+\sigma'}}}|\\
    &\le C\||\tilde{u}^k|^{p-1}-|\tilde{u}|^{p-1}\|_{_{L^\infty(Q_T)}}\|\psi\|_{_{L_2(J;H^{s+\sigma'})}}+C\|\varphi^k-\varphi\|_{L^2(J;H^{s+\sigma'})}\|\psi\|_{_{L^2(J;H^{s+\sigma'})}}.
\end{align*}
Here we tacitly used $\|\tilde{\varphi}^k\|_{_{X^{-s-\sigma',b}_T}}\le C$ and $\|\tilde{u}^k\|_{_{X_T^{s+\sigma,b}}}\le C$. By \eqref{converge-5} and \eqref{converge-6}, we obtain
\begin{equation}\label{convergence-9}
\tfrac{p+1}{2}|\tilde u|^{p-1}\varphi \rightharpoonup \tfrac{p+1}{2}|\tilde u|^{p-1}\varphi\quad {\rm in\ }L^2(J;H^{-s-\sigma'}).
\end{equation}
In the same manner, one may derive
\begin{equation}\label{convergence-10}
\tfrac{p-1}{2}|\tilde u|^{p-3}\tilde u^2\bar \varphi \rightharpoonup \tfrac{p-1}{2}|\tilde u|^{p-3}\tilde u^2\bar \varphi \quad {\rm in\ }L^2(J;H^{-s-\sigma'}).
\end{equation}
By uniqueness, the identity \eqref{limit-4} follows from \eqref{convergence-7}, \eqref{convergence-9} and \eqref{convergence-10}.

\medskip

\textit{Step 4: conclusion.} Now $\varphi(t)$ is a mild solution of linearized Schr\"{o}dinger equation
\[i\varphi_t+\varphi_{xx}-ia(x)\varphi=\tfrac{p+1}{2}|\tilde u|^{p-1}\varphi-\tfrac{p-1}{2}|\tilde u|^{p-3}\tilde u^2\bar \varphi.\]
Due to (\ref{limit-2}), one can apply Lemma~\ref{Lemma-propagation}(1) to derive that $\varphi\in L^2_{loc}(0,T;H^{-s-\sigma'+\rho})$ with $\rho=(1-b)/2$. Then, taking $t_0\in(0,T)$ so that $\varphi(t_0)\in H^{-s-\sigma'+\rho}$, we consider the following initial-value problem
\begin{equation}\label{Adjoint-problem-2}
\left\{\begin{array}{ll}
i\phi_t+\phi_{xx}-ia(x)\phi=\frac{p+1}{2}|\tilde u|^{p-1}\phi-\frac{p-1}{2}|\tilde u|^{p-3}\tilde u^2\bar \phi,\\
\phi(t_0,x)=\varphi(t_0,x).
\end{array}
\right.
\end{equation}
The conclusion of Proposition \ref{Proposition-linearproblem}(2) implies that (\ref{Adjoint-problem-2}) admits a solution $\phi\in X_T^{-s-\sigma'+\rho,b}$. It thus follows by uniqueness that $\varphi=\phi$ and $\varphi\in X_T^{-s-\sigma'+\rho,b}$. Iterating this procedure, we conclude that $\varphi\in C([0,T];H^1)$, which implies that $\varphi$ satisfies the regularity assumption for Lemma \ref{Lemma-UCP}. Accordingly, we conclude
\begin{equation}\label{limit-3}
\varphi(t,x)\equiv 0\quad {\rm on\ }Q_T.
\end{equation}

Combining (\ref{limit-3}) with (\ref{converge-5}), one can take $t_0'\in(0,T)$ so that $\varphi^k(t'_0)\rightarrow 0$ in $H^{-s-\sigma'}$. Moreover, we recall the hypothesis $\{\tilde u^k\}\subset \overline{B}_{X_T^{s+\sigma,b}}(R)$. This together with Proposition \ref{Proposition-linearproblem}(2) implies that there exists a constant $C>0$, not depending on $k$, such that
$$
\|\varphi_T^k\|_{_{H^{-s-\sigma'}}}\leq C\|\varphi^k(t_0')\|_{_{H^{-s-\sigma'}}}\rightarrow 0,
$$
which is contraditory to the assumption that $\|\varphi_T^k\|_{_{H^{-s-\sigma'}}}=1$.
The proof is then complete.
\end{proof}

\subsection{Proof of Lemma {\rm\ref{Lemma-L^2observable}}}\label{Appendix-L^2observable}

This is parallel to (and simpler than) the proof of Lemma~\ref{Lemma-fullobs} carried out in Appendix~\ref{Appendix-control} above, since Step 3 is superfluous at present and it is rather easy to see that the limit function verifies damped linear Schr\"odinger equation. To be precise, we assume on the contrary that there exists a sequence $\{u^k_0\}\subset L^2$ such that $\|u^k_0\|_{_{L^2}}=1$ while
\begin{equation}\label{au^kto 0}
\int_{Q_T} a(x)|u^k|^2 \to 0,
\end{equation}
where $u^k$ is the solution of $iu^k_t+u^k_{xx}=ia(x)u^k$ with $u^k(0)=u^k_0$.

Up to a subsequence, we may assume $u^k\rightharpoonup u$ in $X_T^{0,b}$. As a result, we also have $ia(x)u^k\rightharpoonup ia(x)u$ in $X_T^{0,b}$. Moreover, $u$ is a mild solution of $iu_t+u_{xx}=ia(x)u$. By \eqref{au^kto 0}, we have $u^k\to 0$ in $L^2([0,T]\times \mathcal{I}_1)$, and thus $u(t,x)=0$ for $x\in \mathcal{I}_1$. 

Invoking Lemma~\ref{Lemma-propagation}(2), as $u^k$ is bounded in $X_T^{0,b}$, $u^k$ and $ia(x)u^k$ converges strongly to $u$ and $ia(x)u^k$ in $X_T^{1-b,-b}$ respectively, and $u^k\to 0=u$ in $L^2(0,T;\mathcal{I}_1)$, we conclude that $u^k\to u$ in $L^2_{loc}(0,T;L^2)$ (cf.~Step 2 in the proof of Lemma~\ref{Lemma-fullobs}). 

Finally, thanks to Lemma~\ref{Lemma-propagation}(1), we can deduce that $u\in C(0,T;H^1)$ (cf.~Step 4 in the proof of Lemma~\ref{Lemma-fullobs}). And hence by unique continuation (Lemma~\ref{Lemma-UCP} with $V_1=ia(x)$ and $V_2=0$) we obtain $u(t,x)\equiv 0$. We may find $t_0\in [0,T]$ such that $u^k(t_0)\to u(t_0)$ in $L^2$ due to strong convergence in $L^2_{loc}(0,T;L^2)$, which implies via apriori estimate (cf.~Proposition~\ref{Proposition-linearproblem}(2)) that
\[\|u_0^k\|_{_{L^2}}\le C\|u^k(t_0)\|_{_{L^2}}\to 0.\]
This exhibits a contradictory to our assumption that $\|u_k^0\|_{_{L^2}}=1$. Now the proof is complete.

\medskip

\noindent\textbf{Acknowledgments} \; 
The authors would like to thank Ziyu Liu for valuable discussions and suggestions during the preparation of the paper. The authors also appreciate
the anonymous referees for insightful comments and suggestions.

\vspace{2mm}

\noindent\textbf{Funding} \; Shengquan Xiang is partially supported by NSFC 12301562. Zhifei Zhang is partially supported by NSFC 12288101. Jia-Cheng Zhao is supported by China Postdoctoral Science Foundation 2024M750044.

\vspace{2mm}

\noindent\textbf{Data availability statement}\; No datasets were generated or analyzed during the current study.

\vspace{2mm}

\noindent\textbf{Conflict of interest}\; 
The authors have no conflict of interest to declare.

\normalem
\bibliographystyle{plain}
\bibliography{References}
	
\end{document}